\documentclass{amsart}

\usepackage{amsthm}
\usepackage{amsfonts}
\usepackage{amsmath}
\usepackage{amssymb}
\usepackage{enumitem}
\usepackage{fancyhdr}

\usepackage{tikz}
\usepackage{tikz-cd}
\usetikzlibrary{arrows,decorations.markings}

\usepackage{hyperref}

\usepackage[numbers]{natbib}

\usepackage{mathtools}

\setlist[description]{leftmargin=\parindent,labelindent=\parindent}

\usepackage{graphicx}
\graphicspath{ {images/} }

\newtheorem{theorem}{Theorem}
\numberwithin{theorem}{section}
\newtheorem{lemma}{Lemma}
\numberwithin{lemma}{section}
\newtheorem{proposition}{Proposition}
\numberwithin{proposition}{section}
\newtheorem{corollary}{Corollary}
\numberwithin{corollary}{section}

\numberwithin{equation}{section}
\counterwithout{figure}{section}

\theoremstyle{definition}
\newtheorem{definition}{Definition}
\numberwithin{definition}{section}

\theoremstyle{remark}

\newtheorem{remark}{Remark}
\numberwithin{remark}{section}
\newtheorem{example}{Example}
\numberwithin{example}{section}
\newtheorem{assumption}{Assumption}
\numberwithin{assumption}{section}

\title[Mirrors to lines in $\mathbb{P}^3$]{Cusped hyperbolic Lagrangians as mirrors to lines in three-space}
\author{Sebastian Haney}

\DeclareMathOperator\Hom{Hom}
\DeclareMathOperator\Ext{Ext}

\DeclareMathOperator\Log{Log}
\DeclareMathOperator\hol{hol}
\DeclareMathOperator\Int{Int}

\DeclareMathOperator\Coh{Coh}
\DeclareMathOperator\ev{ev}

\newcommand{\Lpants}{L_{\mathrm{pants}}}
\newcommand{\Limm}{L_{\mathrm{im}}}
\newcommand{\tLimm}{\widetilde{L}_{\mathrm{im}}}
\newcommand{\Lsing}{L_{\mathrm{sing}}}

\begin{document}
\begin{abstract}
We construct a Lagrangian submanifold in the cotangent bundle of a $3$-torus whose projection to the fiber is a neighborhood of a tropical curve with a single $4$-valent vertex. This Lagrangian submanifold has an isolated conical singular point, and its smooth locus is diffeomorphic to the minimally-twisted five component chain link complement, a cusped hyperbolic $3$-manifold. From this singular Lagrangian, we construct a Lagrangian immersion, and identify a moduli space of objects in the wrapped Fukaya category that it supports. We show that for a generic line in projective $3$-space, there is a local system on this Lagrangian immersion such that the resulting object of the wrapped Fukaya category is homologically mirror to an object of the derived category supported on the line. In the course of the proof, we construct a version of the wrapped Fukaya category with objects supported on Lagrangian immersions, which may be of independent interest.
\end{abstract}

\maketitle

\tableofcontents

\section{Introduction} Mirror symmetry is a family of conjectural equivalences originating in string theory between the symplectic geometry ($A$-model) and algebraic geometry ($B$-model) of so-called mirror pairs of spaces. One such duality is Kontsevich's homological mirror symmetry conjecture~\cite{KontsevichICM}, which says that for mirror pairs $(X,\check{X})$ of Calabi--Yau $n$-folds, there is an equivalence of triangulated categories between the (derived) Fukaya category $D^{\pi}\mathrm{Fuk}(X)$ and the derived category of coherent sheaves $D^b\mathrm{Coh}(\check{X})$ on the mirror. A geometric picture of mirror symmetry was proposed by Strominger, Yau, and Zaslow in~\cite{SYZ}, which is now known as SYZ mirror symmetry. Under SYZ mirror symmetry, certain mirror pairs should admit dual Lagrangian torus fibrations $X\to Q$ and $\check{X}\to Q$ over the same base manifold. In general the base $Q$ should carry a singular affine structure, and the fibration should have some singular fibers. From this perspective one can compare the symplectic geometry of $X$ with the algebraic geometry of $\check{X}$ via the tropical geometry of $Q$.

A relatively simple incarnation of the correspondence between complex and tropical geometry is that if one is given complex submanifold $C$ of $(\mathbb{C}^*)^n$, then for suitable collections of polynomials defining $C$, one can produce a collection of corresponding tropical polynomials, which define a tropical subvariety of $\mathbb{R}^n$. This tropical subvariety should be thought of as the image of $C$ under the valuation map $\Log_t\colon(\mathbb{C}^*)^n\to\mathbb{R}^n$ defined by
\begin{align*}
\Log_t(z_1,\ldots,z_n) = (\log_t|z_1|,\ldots,\log_t|z_n|)
\end{align*}
in the limit as $t\to\infty$.

One expects to have a Lagrangian submanifold of the mirror $T^*T^n$ which also projects to a neighborhood of this tropical subvariety, in this case under the map to the cotangent fiber. Such tropical Lagrangian lifts have been constructed for smooth tropical hypersurfaces in $\mathbb{R}^n$ in~\cite{Matessi2} and~\cite{HicksUnobstructed}  and for tropical curves with trivalent vertices in symplectic toric manifolds in~\cite{Mikhalkin}. See also~\cite{MakRuddat} and~\cite{HicksDimer} for constructions of tropical Lagrangian submanifolds of more general symplectic manifolds.

One of the main results of this paper is the construction of Lagrangian lifts of tropical curves that have $4$-valent vertices. A tropical curve with a single $4$-valent vertex is either contained in a $2$-dimensional affine subspace of $\mathbb{R}^n$, or else it is the image, under an affine transformation in $\mathbb{R}^n\rtimes GL(n,\mathbb{Z})$, of the $4$-valent curve depicted in~\ref{intro-fig}. This tropical curve cannot be written as the intersection of two tropical hyperplanes in three dimensions. From the $B$-model perspective, such curves arise as tropicalizations of lines in $(\mathbb{C}^*)^3$ which are symmetric with respect to the $\mathbb{Z}/3$-action given by cyclic permutations of the coordinates. Somewhat surprisingly, the Lagrangian lifts of such curves are not immersed Lagrangian submanifolds.
\begin{theorem}[Paraphrasing of Theorem~\ref{singularlift}]\label{mainthm0}
The $4$-valent tropical curve $V$ in $\mathbb{R}^3$ depicted in Figure~\ref{intro-fig} has a Lagrangian lift $\Lsing$ with one singular point whose neighborhood is homeomorphic to the cone over a $2$-torus. Away from the singular point, $\Lsing$ is diffeomorphic to a hyperbolic link complement in $S^3$, called the minimally twisted five-component chain link, which is also drawn in Figure~\ref{intro-fig}.
\end{theorem}
\begin{remark} Under the embedding of the minimally twisted five-component chain link complement into $T^*T^3$ described in Theorem~\ref{mainthm0}, the cusps of the link complement labeled $1$, $2$, $3$, and $4$ are mapped to the conormal $2$-tori to the legs of $V$ and the cusp labeled $0$ is the link of the cone point. This labeling of the cusps is arbitrary in the sense that there are diffeomorphisms of the link complement which realize any permutation of the cusps~\cite[\S 1.2]{MPR}.
\end{remark}
\begin{figure}
\begin{tikzpicture}
\begin{scope}\
  \draw[thick] (0,0,0) -- (-4,0,0) node[anchor = north west]{$(-1,0,0)$};
  \draw[thick] (0,0,0) -- (0,-4,0) node[anchor = south west]{$(0,0,-1)$};
  \draw[thick] (0,0,0) -- (0,2,-3) node[anchor = north west]{$(1,1,1)$};
  \draw[thick] (0,0,0) -- (-4,-4,-4) node[anchor = north west]{$(0,-1,0)$};
\end{scope}

\begin{scope}[yshift = -250, scale = 1.5]
\node[] at (2,0) {1};
\node[] at (0.62,1.90) {0};
\node[] at (-1.62,1.17) {4};
\node[] at (-1.62,-1.17) {3};
\node[] at (0.62,-1.90) {2};

\draw(2,0)[line width = 0.25 mm] ellipse (0.5 and 1.5);
\draw[rotate = 72, line width = 0.25 mm] (2,0) ellipse (0.5 and 1.5);
\draw[rotate = 144, line width = 0.25 mm] (2,0) ellipse (0.5 and 1.5);
\draw[rotate = 216, line width = 0.25 mm] (2,0) ellipse (0.5 and 1.5);
\draw[rotate = 288, line width = 0.25 mm] (2,0) ellipse (0.5 and 1.5);

\draw (2.5,0)[white, double = black, ultra thick] arc (0:100:0.5 and 1.5);
\draw (2.5,0)[white, double= black, ultra thick] arc (0:-100:0.5 and 1.5);
\draw[white, double = black, ultra thick, rotate = 72] (2.5,0) arc (0:100:0.5 and 1.5);
\draw[white, double = black, ultra thick, rotate = 144] (1.5,0) arc (180:240:0.5 and 1.5);
\draw[white, double = black, ultra thick, rotate = 72] (1.5,0) arc (180:240:0.5 and 1.5);

\draw[white, double = black, ultra thick, rotate = 144] (1.5,0) arc (180:120:0.5 and 1.5);

\draw[white, double = black, ultra thick, rotate = 216] (2.5,0) arc (0:-100:0.5 and 1.5);

\draw[white, double = black, ultra thick, rotate = 216] (2.5,0) arc (0:100:0.5 and 1.5);

\draw[white, double = black, ultra thick, rotate = 288] (1.5,0) arc (180:240:0.5 and 1.5);
\draw[white, double = black, ultra thick, rotate = 288] (1.5,0) arc (180:120:0.5 and 1.5);
\end{scope}
\end{tikzpicture}

\caption{The $4$-valent tropical curve $V$ (top) and the minimally twisted five-component chain link (bottom).}\label{intro-fig}
\end{figure}

\begin{corollary}
Any tropical curve in $\mathbb{R}^n$ with vertices of valence at most $4$ has a Lagrangian lift in $T^*T^n$ which is immersed away from a discrete set of isolated conical singular points, which project to the $4$-valent vertices under projection to the cotangent fiber.
\end{corollary}
\begin{proof}
Any tropical curve with a single $4$-valent vertex is either 
\begin{itemize}
\item[(a)] contained in a $2$-dimensional affine subspace of $\mathbb{R}^n$, or;
\item[(b)] any three of its four edges are linearly independent.
\end{itemize}
We will explain how to construct Lagrangian submanifolds lifting these two types of curve separately.

Let $W\subset\mathbb{R}^n$ be a tropical curve with a single ($4$-valent) vertex as in case (a). Then the edges of $W$ span a $2$-dimensional affine subspace of $\mathbb{R}^n$, implying that we can view $W$ as a tropical hypersurface in $\mathbb{R}^2$. Lagrangian lifts of tropical hypersurfaces were constructed in~\cite[Definition 3.15]{HicksUnobstructed}, and by taking the product of this lift with an isotropic copy of $T^{n-2}$ dual to the plane containing $W$, we obtain the desired Lagrangian lift of the curve in $\mathbb{R}^n$.

Next let $W\subset\mathbb{R}^n$ be a tropical curve with a single vertex as in (b). In this case, the legs of $W$ span a three-dimensional affine subspace, so we view $W$ as a tropical curve in $\mathbb{R}^3$. By taking a cover of the Lagrangian lift from Theorem~\ref{mainthm0}, using the argument in~\cite[Example 5.3]{Matessi}, we obtain a lift of $W$. Crossing this (singular) tropical Lagrangian submanifold with an isotropic copy of $T^{n-3}$ gives us the desired lift of $W\subset\mathbb{R}^n$.

Finally suppose that $W\subset\mathbb{R}^n$ is a tropical curve with arbitrarily many vertices, all of which have valence at most $4$. We can produce a Lagrangian lift of $W$ by gluing together the lifts of curves with a single vertex obtained above over the edges.
\end{proof}

The link of the cone point of $\Lsing$, viewed as a Legendrian submanifold of $S^5$, has three (non-isotopic) smooth asymptotically conical Lagrangian fillings. We will show that all three of these fillings can be used to produce smooth Lagrangian submanifolds of $T^*T^3$ following~\cite{Joyce}, but these will not be Lagrangian lifts of $V$. Instead, the Lagrangian submanifolds obtained by desingularizing $\Lsing$ correspond to resolutions of the tropical curve $V$, where the $4$-valent vertex is replaced by a pair of trivalent vertices.

Consider the tropical curve $V_3\subset\mathbb{R}^3$ given by
\begin{align*}
V_3 \coloneqq \left\lbrace(-\epsilon-t,-\epsilon,0)\colon t\in[0,\infty)\right\rbrace &\cup\left\lbrace(-\epsilon,-\epsilon-t,0)\colon t\in[0,\infty)\right\rbrace \\
&\cup\left\lbrace(t,t,0)\colon t\in[-\epsilon,\epsilon]\right\rbrace \\
&\cup\left\lbrace(\epsilon,\epsilon,-t)\colon t\in[0,\infty)\right\rbrace \\
&\cup\left\lbrace(\epsilon+t,\epsilon+t,t)\colon t\in[0,\infty)\right\rbrace \, .
\end{align*}
One can define the other desingularizations $V_1$ and $V_2$ by reordering the coordinates on $\mathbb{R}^3$ appropriately. Lagrangian lifts $L_{V_i}$ of the tropical curves $V_i$ for $i = 1,2,3$ have been constructed in~\cite{HicksReal, Mikhalkin, MakRuddat}. Each $L_{V_i}$ is diffeomorphic to the graph manifold
\begin{equation}\label{graphmfd}
(P\times S^1)\bigcup_{\begin{pmatrix} 0 & 1 \\ 1 & 0 \end{pmatrix}}(P\times S^1),
\end{equation}
where $P\coloneqq\mathbb{CP}^1\setminus\lbrace 0,1,\infty\rbrace$ denotes the pair of pants and the matrix $\begin{pmatrix} 0 & 1 \\ 1 & 0 \end{pmatrix}$ signifies that the two copies of $P\times S^1$ are glued together along one of their $T^2$-boundary components using a diffeomorphism that interchanges the $S^1$-factor with the boundary circle of $P$. It will turn out that the smoothings of $\Lsing$ will share this topological type, and can be thought of as Lagrangian lifts of $V_i$ in a weak sense.
\begin{theorem}\label{smoothings}
There are three smooth Lagrangian submanifolds $L_i\subset T^*T^3$, for $i = 1,2,3$, obtained by replacing a neighborhood of the cone point in $\Lsing$ with a smooth filling of the Legendrian link. Each $L_i$ is diffeomorphic to the graph manifold~\eqref{graphmfd}. The projection of $L_i$ to the cotangent fiber agrees with $V_i$ away from a neighborhood of the edge of the finite edge. The Lagrangian submanifolds $L_1$, $L_2$, and $L_3$ can be identified with the Dehn fillings of the zeroth cusp of $L'$ with filling coefficient $\infty$, $1$, and $0$, respectively.
\end{theorem}
\begin{remark}
It turns out that these are the only three exceptional, i.e. non-hyperbolic, fillings of the minimally twisted five-chain link complement along one of its components~\cite{MPR}.
\end{remark}

One would expect a stronger relationship between the tropical Lagrangian lifts of a tropical subvariety and the corresponding complex subvariety $C$ of the mirror. Specifically, a suitable tropical Lagrangian lift should be mirror to the structure sheaf $\mathcal{O}_C$ under the equivalence between the wrapped Fukaya category of $T^*T^n$ and $D^b\mathrm{Coh}((\mathbb{C}^*)^n)$. For smooth tropical hypersurfaces, a mirror relation of this type is proven by Hicks~\cite{HicksUnobstructed}, using the fact that tropical Lagrangian hypersurfaces are mapping cones in the Fukaya category, essentially by construction.

To prove a similar result for the Lagrangian submanifold $\Lsing$ of Theorem~\ref{mainthm0}, one would have to make sense of Lagrangian Floer theory for Lagrangian submanifolds with conical singularities. Rather than taking this approach, we construct a Lagrangian immersion in $T^*T^3$ by using a `doubling' trick, similar to~\cite{seidelgenus2},~\cite{Sheridan},~\cite{AbouzaidSylvan}, in which Lagrangian immersions are morally constructed by gluing two copies of a singular Lagrangian submanifold together. 

Let $\mathbb{K}$ be a field, and identify $\Hom(\pi_1(T^3),GL(1,\mathbb{K}))$ with $(\mathbb{K}^*)^3$. Our other main result is the following.
\begin{theorem}\label{main}
There is an exact Lagrangian immersion $\tLimm\overset{\mathrm{im}}{\longrightarrow} T^*T^3$ whose image $\Limm$ is $C^0$-close to $\Lsing$, and with the property that for a generic line $C\subset\mathbb{P}^3$, there is a $GL(1)$-local system $\nabla_C$ on $\tLimm$ such that the set
\[ \lbrace\nabla_p\in\Hom(\pi_1(T^3),GL(1,\mathbb{K}))\colon HF^*((\Limm,\nabla_C),(T^3,\nabla_p))\neq0\rbrace\]
is equal to the affine part $C\cap(\mathbb{K}^*)^3$.
\end{theorem}
We call the set of such local systems the Floer-theoretic support of $\Limm$ following~\cite{HicksReal}. By a generic line, we mean one which lies in a certain Zariski open subset of the Grassmannian $Gr(2,4)$: if we embed $Gr(2,4)$ in $\mathbb{P}^5$ via the Pl{\"u}cker embedding, then the lines to which our theorem applies are those in $Gr(2,4)\cap(\mathbb{K}^*)^5$. This restriction comes from the fact that the holonomy of a local system on $\Limm$ must always be nonzero.

Under homological mirror symmetry, the Floer-theoretic support corresponds to the usual support of the mirror sheaf. This is closely related to the support of the mirror sheaf constructed using family Floer theory~\cite{AbouzaidFFT}, except that instead of considering families of Lagrangian tori in $T^*T^3$, most of which are not exact, we work with a single exact Lagrangian torus and equip it with non-unitary local systems. Consequently, our mirror is defined over $\mathbb{K}$, as opposed to the Novikov field. While significantly less general than family Floer theory, our technical setup does allow us to state Theorem~\ref{main} without reference to the SYZ fibration or to tropical geometry.

\begin{remark}\label{extendedms} Our motivation for working in this setting is to eventually consider Lagrangian submanifolds in the quintic $3$-fold mirror to lines, for which a Lagrangian torus fibration suitable for proving an analogue of Theorem~\ref{main} has not yet been constructed to the author's knowledge. It was first shown by van Geemen, and written in~\cite{Katz}, that the mirror quintic contains one-dimensional families of lines. Certain special lines in this family should tropicalize to $4$-valent tropical curves of the same combinatorial type as the one in Figure~\ref{intro-fig}. We expect to use Theorem~\ref{main} combined with Sheridan's proof of homological mirror symmetry for Calabi--Yau hypersurfaces~\cite{Sheridan2} to construct Lagrangian submanifolds in the quintic that are mirror to these families of lines. These Lagrangian submanifolds are constructed by lifting tropical curves in the intersection complex of the quintic threefold, and are thus related to, but distinct from, the Lagrangian submanifolds of mirror quintic threefolds constructed by~\cite{MakRuddat}.

Assuming certain structural properties of the Fukaya category related to the existence of certain chain-level enhancements of the proper Calabi--Yau structure, an appropriate extension of Theorem~\ref{main} to the quintic combined with computation of the normal function associated to these lines in~\cite{Walcher} should be sufficient to prove extended mirror symmetry, in the sense of~\cite{HahnWalcher}, for lines in the mirror quintic. In particular, this would determine the open Gromov--Witten potential for certain Lagrangian submanifolds in the quintic threefold, and the coefficients of this potential--the open Gromov--Witten invariants of these Lagrangian submanifolds-- would lie in $\mathbb{Q}(\sqrt{-3})$ (cf.~\cite[\S 6.1]{Walcher}).
\end{remark}

The rest of this article is structured as follows. Section~\ref{iwf} constructs a version of the wrapped Fukaya category whose objects are certain Lagrangian immersions equipped with local systems. In Section~\ref{backgroundSYZ} we will review SYZ mirror symmetry for $T^*T^n$ and the construction of a tropical Lagrangian pair of pants. This will serve as motivation for Section~\ref{singularconstruction}, where we prove Theorem~\ref{mainthm0}. The proof of Theorem~\ref{smoothings} is given in Section~\ref{nbdhsmthsect}, and makes use of a Lagrangian neighborhood theorem for (cylindrical at infinity) Lagrangians with conical singularities, also proved therein. In Section~\ref{IL} we construct the immersed Lagrangian submanifold $\Limm$, and study its associated moduli space of unobstructed local systems. Finally, in Section~\ref{supportsection}, we will prove Theorem~\ref{main}.
\subsection*{Acknowledgments} I would like to thank my advisor, Mohammed Abouzaid, for his support and encouragement, for many enlightening discussions, and for comments on earlier drafts of this paper. I would also like to thank Jeff Hicks, Cheuk Yu Mak, Grigory Mikhalkin, Mike Miller Eismeier, Nick Sheridan, Jake Solomon, Nicol{\'a}s Vilches, and Johannes Walcher for helpful conversations at various stages of this project. Finally, I thank the anonymous referee for a very thorough reading of this paper and many helpful comments and suggestions. This work was partially supported by the NSF through Abouzaid's grant DMS-2103805 and by the Simons Collaboration on Homological Mirror Symmetry.

\section{Immersed wrapped Floer theory}\label{iwf}
In this section, we will construct the wrapped Fukaya category of a Liouville manifold $M$, allowing objects to be supported on certain Lagrangian immersions. Intuitively, one might expect the wrapped Floer $A_{\infty}$-algebra of a Lagrangian immersion $\iota\colon\widetilde{L}\to M$ to have generators corresponding either to self-intersection points of the immersion or to Hamiltonian chords in the cylindrical end of $M$. In that case, the $A_{\infty}$-operations would need to count disks with both corners and boundary punctures asymptotic to Hamiltonian chords. To simplify matters, we will only consider Lagrangian immersions for which the curvature, as one would expect to see in the full curved $A_{\infty}$-structure as described above, vanishes. We will then define the Floer cochain complex of $\widetilde{L}$ using Hamiltonian perturbations, as in~\cite{SeidelBook}. This means that we will not need to consider, say, singular chains on the self-intersections of $\iota(\widetilde{L})$.

More precisely, for a Lagrangian immersion $\iota\colon\widetilde{L}\to M$ equipped with a rank-one local system $\nabla$, we will define a curvature term $\mathfrak{m}_0(\widetilde{L},\nabla)$ which counts pseudoholomorphic teardrops (weighted by holonomy) on $L$ of the appropriate index, and we will only take branes for which $\mathfrak{m}_0(\widetilde{L},\nabla) = 0$ to be objects of the wrapped Fukaya category. To rule out certain disk bubbles of nonpositive index, we must also impose a semi-positivity assumption, Definition~\ref{monotonicity}, on the immersed Lagrangian submanifolds that we consider. The construction of Floer theory for so-called \textit{cleanly positive} Lagrangian immersions is analogous to the construction of Floer theory of monotone Lagrangian embeddings. We will also only consider Lagrangian immersions which are embedded away from a compact subset of $M$, so that this first step can be carried out entirely in the interior part of $M$.

When we construct the wrapped $A_{\infty}$-operations, we will define them by counting inhomogeneous pseudoholomorphic polygons with boundary punctures asymptotic to Hamiltonian chords. The $1$-dimensional moduli spaces may have boundary strata arising from bubbling of pseudoholomorphic teardrops, but standard gluing results and the vanishing of $\mathfrak{m}_0(\widetilde{L},\nabla)$ imply that weighted counts of points in these $0$-dimensional strata cancel, and so we are able to prove the uncurved $A_{\infty}$-relations.

\subsection{Lagrangian immersions}\label{immersionbackground}
Let $M$ be a Liouville manifold, meaning that it is equipped with a one-form $\lambda$ such that $\omega\coloneqq d\lambda$ is a symplectic form, with the property that the \textit{Liouville vector field} $Z_{\lambda}$, characterized by the property that
\[ \omega(Z_{\lambda},\cdot)\equiv\lambda(\cdot) \]
generates a complete expanding flow. Outside of a compact subset of $M$, we assume that the flow is modeled on multiplication in the positive end of the symplectization of a contact manifold.
We will write
\[ M = M^{in}\cup_{\partial M}(\partial M\times[1,\infty)),\]
where $M^{in}$ is a (compact) Lioville domain. In the cylindrical end $M\times[1,\infty)$, the Liouville form is given by $r\lambda|_{\partial M}$, where $r$ is the $[1,\infty)$-factor, and hence the Liouville vector field is $\frac{\partial}{\partial r}$ in the cylindrical end. Also assume that $c_1(M) = 0$.

\begin{definition}
A Lagrangian immersion $\iota\colon\widetilde{L}\to M$ is said to have clean self-intersections if
\begin{itemize}
\item[(i)] The fiber product
\[ \widetilde{L}\times_{\iota}\widetilde{L} = \lbrace (p_{-},p_{+})\in\widetilde{L}\times\widetilde{L}\colon\iota(p_{-}) = \iota(p_{+})\rbrace \]
is a smooth submanifold of $\widetilde{L}\times\widetilde{L}$.

\item[(ii)] At each point $(p_{-},p_{+})$ of the fiber product, the tangent space is given by the fiber product of tangent spaces, i.e.
\[ T_{(p_{-},p_{+})}(\widetilde{L}\times_{\iota}\widetilde{L}) = \lbrace (v_{-},v_{+})\in T_{p_{-}}\widetilde{L}\times T_{p_{+}}\widetilde{L}\colon d\iota_{p_{-}}(v_{-}) = d\iota_{p_{+}}(v_{+})\rbrace.\]
\end{itemize}
\end{definition}
We will often let $L\coloneqq\iota(\widetilde{L})$ denote the image of such an immersion.
\begin{definition}
Suppose that $\iota\colon\widetilde{L}\to M$ is a Lagrangian immersion with clean self-intersections. We say that $\widetilde{L}$ is \textit{exact} if there is a smooth function $f_L\colon\widetilde{L}\to\mathbb{R}$ such that $df_L = \iota^*\lambda$.
\end{definition}
This definition of exactness is the same one used by Alston--Bao~\cite{AlstonBao}. Exactness rules out the existence of holomorphic disks with \textit{smooth} boundary on $L$, but it does not rule out the existence of holomorphic teardrops with boundary on $L$.
\begin{definition}
A Lagrangian immersion $\iota\colon\widetilde{L}\to M$ is said to be cylindrical at infinity if there is a closed embedded Legendrian $\Lambda\subset\partial M$ such that $L$ satisfies
\[ L\cap(\partial M\times[1,\infty)) = \Lambda\times[1,\infty) \]
and such that the restriction of $\iota$ to $\iota^{-1}(\Lambda\times[1,\infty))$ is an embedding.
\end{definition}
For a Lagrangian immersion with clean self-intersections, we write
\[ \widetilde{L}\times_{\iota}\widetilde{L} = \coprod_{a\in A}L_a = \Delta_{\widetilde{L}}\sqcup\coprod_{a\in A\setminus\lbrace 0\rbrace} L_a,\]
where $A$ is an index set with a choice of distinguished element $0\in A$. Here, $L_0 = \Delta_{\widetilde{L}}$ denotes the diagonal component of the fiber product. Note that $\Delta_{\widetilde{L}}$ is disconnected if $\widetilde{L}$ is disconnected. Each other component $L_a$ of the disjoint union is a non-diagonal connected component of the fiber product, which we will call switching components. Observe that there is a natural involution on $A\setminus\lbrace 0\rbrace$ induced by the function on $\widetilde{L}\times_{\iota}\widetilde{L}$ which swaps coordinates.

\subsection{Obstructions}~\label{obs}
In this subsection, we will define a certain count $\mathfrak{m}_0(\widetilde{L},\nabla)\in\mathbb{K}$ of pseudoholomorphic disks on $L$ with one corner weighted by holonomy, where $\mathbb{K}$ is a field and $\nabla$ is a $GL(1,\mathbb{K})$-local system on $\widetilde{L}$. 
\begin{definition}
Consider a Lagrangian immersion $\iota\colon\widetilde{L}\to M$ which is exact, cylindrical at infinity (or compact), and has at worst clean self-intersections. Additionally, we assume that all switching components of $\widetilde{L}\times_{\iota}\widetilde{L}$ are closed manifolds. A Lagrangian brane consists of such a Lagrangian immersion $\widetilde{L}$, together with a $GL(1,\mathbb{K})$-local system $\nabla$, a spin structure (and induced orientation), and a grading $\alpha^{\#}\colon\widetilde{L}\to\mathbb{R}$. We will denote such an object by $(\widetilde{L},\nabla)$. Sometimes we will also write this as $L$, which we have also used for the image of the immersion, in an abuse of notation.
\end{definition}

To show that the moduli spaces of pseudoholomorphic disks on $L$ are cut out transversely, we will impose a constraint on $\widetilde{L}$ which is similar to monotonicity. We will now recall the notion of index from~\cite[\S{2.3}]{AlstonBao} to state this assumption properly.

\begin{definition}
Let $(V,\omega)$ be a symplectic vector space with a compatible almost complex structure $J$, and let $\Lambda_0,\Lambda_1\subset V$ be Lagrangian subspaces. Choose a path $\Lambda_t$, for $t\in[0,1]$, of Lagrangian subspaces from $\Lambda_0$ to $\Lambda_1$ which satisfies the following:
\begin{itemize}
\item $\Lambda_0\cap\Lambda_1\subset\Lambda_t\subset\Lambda_0+\Lambda_1$ for all $t$, and
\item $\Lambda_t/(\Lambda_0\cap\Lambda_1)\subset(\Lambda_0+\Lambda_1)/(\Lambda_0\cap\Lambda_1)$ is a path of positive-definite subspaces (with respect to the metric induced from $\omega$ and $J$) from $\Lambda_0/(\Lambda_0\cap\Lambda_1)$ to $\Lambda_1/(\Lambda_0\cap\Lambda_1)$.
\end{itemize}
Next consider a path $\alpha_t$ in $\mathbb{R}$ such that
\[ \exp(2\pi i\alpha_t) = \mathrm{det}^2(\Lambda_t) \]
for all $t\in[0,1]$. Then we define the angle between $\Lambda_0$ and $\Lambda_1$ to be
\[ \mathrm{Angle}(\Lambda_0,\Lambda_1) = \alpha_1-\alpha_0 \]
\end{definition}
This definition gives rise to the appropriate notion of index for switching components of a Lagrangian immersion.
\begin{definition}\label{index}
Given a switching component $L_{a_{-}}$ of $\widetilde{L}\times_{\iota}\widetilde{L}$, let $L_{a_{+}}$ denote the corresponding switching component under the involution on $A$. Then for any points $p_{-}\in L_{a_{-}}$ and $p_{+}\in L_{a_{+}}$, we define the index
\[ \deg(p_{-},p_{+})\coloneqq n+\alpha^{\#}(p_{+})-\alpha^{\#}(p_{-})-2\cdot \mathrm{Angle}(d\iota(T_{p_{-}}L),d\iota(T_{p_{+}}L)). \]
Since this is independent of $(p_{-},p_{+})\in L_{a_{-}}\times L_{a_{+}}$, we set 
\[\deg(L_{a_{-}})\coloneqq \deg(L_{a_{-}},L_{a_{+}}) \coloneqq \deg(p_{-},p_{+}).\]
\end{definition}
The following definition should be thought of as an analogue of monotonicity for exact Lagrangian immersions in symplectic Calabi--Yau manifolds, which we will use to verify the wrapped $A_{\infty}$-relations.
\begin{definition}[Clean positivity]~\label{monotonicity}
Given an exact cylindrical Lagrangian immersion with clean self-intersections we say that it is \textit{cleanly positive} if for any \textit{positive-dimensional} switching components $L_{a_{+}}$ and $L_{a_{-}}$ of $\widetilde{L}\times_{\iota}\widetilde{L}$ which are ordered such that the action satisfies 
\[f_L(L_{a_{+}})-f_L(L_{a_{-}})>0\]
one also has that
\[\deg(L_{a_{-}})\geq 2.\]
\end{definition}
\begin{remark}
This condition precludes the existence of teardrops with a corner on a positive-dimensional switching component of degree at most $1$. Teardrops with a corner on a transverse double point of index less than $2$ generically do not exist, as Lemmata~\ref{simplicitylemma} and~\ref{teardropregularity} below show.

Our notion of clean positivity is weaker than~\cite[Condition 1.2]{AlstonBao}, since we do not impose any restrictions on the $0$-dimensional self-intersection points. We will not need such a condition since we will establish that all pseudoholomorphic teardrops under consideration are regular.
\end{remark}
 
Objects of the wrapped Fukaya category will be supported on Lagrangian submanifolds satisfying the following assumption.
\begin{assumption}\label{restrictionassumption}
All Lagrangian immersions $\iota\colon\widetilde{L}\to M$ under consideration are exact, cylindrical (at infinity), spin, have clean self-intersections, and are cleanly positive in the sense of Definition~\ref{monotonicity}. We assume moreover that for any point $p\in M$, the preimage $\iota^{-1}(p)$ consists of at most two points, and also that $\iota\colon\widetilde{L}\to L$ is not a double cover.
\end{assumption}
The last sentence of Assumption~\ref{restrictionassumption} allows us to appeal directly to the regularity and compactness results of~\cite{AlstonBao} (cf. Def. 3.2, \S{8.3}, and \S{9} therein).

We turn now to the construction of the curvature term. Let $S$ be a Riemann surface with one boundary puncture, denoted $\xi_0$, whose compactification is a disk. Consider the space $\mathcal{J}(M)$ of $\omega$-compatible almost complex structures on $M$ which are of contact type when restricted to the cylindrical end $\partial M\times[1,\infty)$, meaning that $\lambda\circ J = dr$.
\begin{definition}
Let $a\in A\setminus\lbrace 0\rbrace$ be a label for a switching component of $\iota\colon\widetilde{L}\to M$, and fix some $J\in\mathcal{J}(M)$. A pseudoholomorphic teardrop with corners on $L$ consists of the data $(S,u,\widetilde{\partial u})$, where
\begin{itemize}
\item[(i)] $u\colon S\to M$ is continuous and $u(\partial S)\subset L$.

\item[(ii)] $\widetilde{\partial u}\colon\partial S\to\widetilde{L}$ is a continuous map such that $\iota\circ\widetilde{\partial u} = u\vert_{\partial S}$ on $\partial S$.

\item[(iii)] The restriction of $u$ to the interior of $S$ is $J$-holomorphic, meaning that $(du)^{0,1} = 0$.

\item[(iv)] We have that
\[ \left(\lim_{z\to\xi_0^{-}}\widetilde{\partial u}(z),\lim_{z\to\xi_0^{+}}\widetilde{\partial u}(z)\right)\in L_{a}.\]
\end{itemize}
Let $\mathcal{M}_1(L,a;J)$ denote the moduli space of $J$-holomorphic teardrops with a corner on $L_a$ modulo automorphisms of the domain.
\end{definition}
We will write $\mathcal{M}_1(L,a)\coloneqq\mathcal{M}_1(L,a;J)$ when the choice of $J$ is understood from context. By adapting the main results of~\cite{Lazzarini} or~\cite{Perrier} to our situation, one can prove that pseudoholomorphic teardrops are simple for generic almost complex structures.
\begin{lemma}\label{simplicitylemma}
Assume that $\dim\widetilde{L}\geq3$. Then there is a second category subset of $\mathcal{J}(M)$ such that for any $J$ in this subset, all $J$-holomorphic teardrops with boundary on $L$ are simple. \qed
\end{lemma}
The proof of this lemma in~\cite[Cor. 4]{Perrier} is written for compact Lagrangians with transverse double-points, but it carries over to clean, cylindrical Lagrangian immersions in Liouville manifolds with no essential changes. Note that by our assumptions on $M$, all pseudoholomorphic teardrops are contained in the compact part $M^{in}$ of $M$. We expect that the same result can be proven if $\dim\widetilde{L}\leq 2$, but this would require a more intricate combinatorial argument along the lines of~\cite{BiranCornea}. It now follows from the standard Sard--Smale argument that we achieve regularity for teardrops. 
\begin{lemma}\label{teardropregularity}
If $\dim\widetilde{L}\geq3$ and $\widetilde{L}$ is equipped with a choice of spin structure, then for generic $J\in\mathcal{J}(M)$, the moduli spaces $\mathcal{M}_1(L,a)$, defined with respect to the domain-independent almost complex structure $J$, are smooth manifolds of dimension
\[ \deg(L_a)-2+\dim(L_a) \]
\qed
\end{lemma}
\begin{proof}
The proof that these moduli spaces are oriented uses the presence of a spin structure on $\widetilde{L}$, as discussed in~\cite{FOOO} or~\cite{Fukaya}. The computation of the dimension formula is similar to the one in~\cite[(33), (141)]{AkahoJoyce}.
\end{proof}

The moduli spaces $\mathcal{M}_{1}(L,a)$ admits a natural evaluation maps
\[ \ev\colon\mathcal{M}_1(L,a)\to L_{a}\]
which has the value $\left(\lim_{z\to\xi_0^{-}}\widetilde{\partial u}(z),\lim_{z\to\xi_0^{+}}\widetilde{\partial u}(z)\right)$ at $(u,\widetilde{\partial u})$.
When $\deg(L_a) = 2$, the moduli space $\mathcal{M}_1(L,a)$ is compact, so we can push the fundamental class forward to obtain
\[ \ev_*[\mathcal{M}_1(L,a)]\in H_*(L_a)\]
of top dimension. Thus by Poincar{\'e} duality we can interpret this as the count of pseudoholomorphic teardrops with a corner on $L_a$. This justifies the following definition.
\begin{definition}\label{tautologicallyunobstructed}
Let $(\widetilde{L},\nabla)$ be a pair consisting of a Lagrangian immersion and rank-one local system. For any relative homology class $\beta\in H_2(M,L;\mathbb{Z})$, write $\mathcal{M}(L,a;\beta)$ for the connected component of $\mathcal{M}(L,a)$ consisting of all teardrops representing the class $\beta$. Then define the curvature term
\[ \mathfrak{m}_0(\widetilde{L},\nabla)\coloneqq\sum_{a\in A\setminus\lbrace 0\rbrace}\sum_{\beta\in H_2(M,L;\mathbb{Z})}\hol_{\nabla}(\partial\beta) PD(\ev_*[\mathcal{M}_1(L,a,\beta)] )\in\bigoplus_{a\in A}H^0(L_a;\mathbb{K}) \,.\]
We say that the Lagrangian brane $(\widetilde{L},\nabla)$ is \textit{unobstructed} if $\mathfrak{m}_0(\widetilde{L},\nabla) = 0$.
\end{definition}
We need to show that the unobstructedness of $(\widetilde{L},\nabla)$ does not depend on the choice of almost complex structure. This is similar to the proof, in the curvature term in the Fukaya $A_{\infty}$-algebra of a monotone Lagrangian submanifold is well-defined. 
\begin{lemma}
Assuming that $\dim\widetilde{L}\geq 3$, let $J_0$ and $J_1$ be two almost complex structures satisfying the conclusion of Lemma~\ref{teardropregularity}. Then a generic path of tame almost complex structures determines an oriented cobordism between the moduli spaces $\mathcal{M}_1(L,a;J_0)$ and $\mathcal{M}_1(L,a;J_1)$, provided that $\deg(L_a) = 2$.
\end{lemma}
\begin{proof}
For a path $\lbrace J_t\rbrace_{t\in[0,1]}$ in $\mathcal{J}(M)$, we consider the moduli space $\mathcal{M}^*_1(L;a;J_t)$ of \textit{simple} $J_t$-holomorphic teardrops with a corner at $a\in A$. Again using the Sard--Smale argument, this moduli space is a smooth $1$-dimensional manifold for generic choices of $\lbrace J_t\rbrace$ with the given endpoints. This moduli space is compact, with boundary given by $\mathcal{M}_1(L,a;J_0)\cup\mathcal{M}_1(L,a;J_1)$, where $\mathcal{M}_1(L,a;J_0)$ is given the opposite orientation. If there were any other boundary strata in the Gromov compactification of $\mathcal{M}^*_1(L;a;J_t)$, they would necessarily consist of nodal $J_t$-holomorphic disks containing, as an irreducible component, a teardrop with a corner on a switching component of degree $1$. But such teardrops cannot exist, either by regularity (for $0$-dimensional switching components) or by Definition~\ref{monotonicity} (for positive-dimensional switching components).
\end{proof}

\subsection{The wrapped Fukaya category}~\label{wrappedFloer}
Consider a finite collection $\mathrm{Ob}(\mathcal{W})$ of unobstructed Lagrangian branes $(\widetilde{L},\nabla)$ in $M$. Here $\widetilde{L}\to M$ is an exact Lagrangian immersion which is cylindrical at infinity and has clean self-intersections, and $\nabla$ is a rank one local system. Each brane in $\mathrm{Ob}(\mathcal{W})$ is also equipped with a grading and spin structure. If $\dim\widetilde{L}\leq 2$, we also assume that $L = \iota(\widetilde{L})$ can be written as a union of embedded Lagrangians in $M$ intersecting each other cleanly. Apart from the discussion of curvature, this subsection follows~\cite{AbouzaidGen} closely. Since proofs of most results in this section are standard, we will omit many details, emphasizing the new issues that appear in the presence of Lagrangian immersions. 

Let $\mathcal{H}(M)$ denote the set of all Hamiltonians which are of the form
\[ H(r,y) = r^2\]
away from a compact subset of $M$. Fix some $H\in\mathcal{H}(M)$. Let $X$ denote the Hamiltonian flow of $H$ and, for a pair $L_0,L_1\in\mathrm{Ob}(\mathcal{W})$, let $\mathcal{X}(L_0,L_1)$ denote the set of time-one flow lines of $X$ from $L_0$ to $L_1$, which as usual denote the images of the immersions. To define a graded cochain complex, we should assume that:
\begin{assumption}\label{nondegchords}
All chords in $\mathcal{X}(L_0,L_1)$ are nondegenerate.
\end{assumption}
Note that this assumption includes the case that $L_0 = L_1$. Let $\deg(x)$ denote the Maslov index of the chord $x\in\mathcal{X}(L_0,L_1)$.

\begin{definition}
Let $Z\coloneqq(-\infty,\infty)\times[0,1]$ with coordinates $(s,t)$. Given $x_0,x_1\in\mathcal{X}(L_0,L_1)$, define the spaces $\widehat{\mathcal{M}}(x_0,x_1)$ to be the set of all maps $u\colon Z\to M$ which converge exponentially to $x_0$ at the negative end and $x_1$ at the positive end, satisfying the boundary conditions
\begin{align*}
u(s,0) &\in L_0 \\
u(s,1) &\in L_1
\end{align*}
and which satisfy Floer's equation
\[ (du-X\otimes dt)^{0,1} = 0, \]
with respect to a generic almost complex structure of contact type.
\end{definition}

Similarly to~\cite[Lemma 2.3]{AbouzaidGen}, we have the following regularity statement for \textit{uncompactified} moduli spaces of strips.
\begin{lemma}
For a generic almost complex structure $J$ satisfying the conclusion of Lemma~\ref{teardropregularity}, the moduli space $\widehat{\mathcal{M}}(x_0,x_1)$ is a smooth manifold of dimension $\deg(x_0)-\deg(x_1)$. Whenever $\deg(x_0)-\deg(x_1)>0$, the natural $\mathbb{R}$-action induced by translations in the $s$-direction is smooth and free. \qed
\end{lemma}
\begin{proof}
This follows from the argument of~\cite[Prop. 25]{Perrier}. In more detail, let $\phi^1_H$ denote the time-$1$ flow of $X$, and note that elements of $\widehat{\mathcal{M}}(x_0,x_1)$ correspond to pseudoholomorphic strips with a boundary component on $L_0$ and another boundary component on $(\phi^1_H)^{-1}(L_1)$. Any such pseudoholomorphic strip can be thought of as a pseudoholomorphic disk with corners on the Lagrangian immersion $L_0\cup(\phi^1_H)^{-1}(L_1)$, and since this strip has boundary components on \textit{different} Lagrangian submanifolds, it is not a double cover of a teardrop. Applying~\cite[Cor. 1]{Perrier} completes the proof. Although the results of op. cit. are stated for compact Lagrangian submanifolds, they apply in our setting for similar reasons as discussed before Lemma~\ref{teardropregularity}. More precisely, Perrier deduces~\cite[Cor. 1]{Perrier} from a factorization of $J$-holomorphic polygons into simple $J$-holomorphic polygons in relative homology~\cite[Theorem 1]{Perrier}. The latter result applies to any almost complex structure, and thus we can choose $J$ generically among the subspace of contact structures which are of contact type at infinity.
\end{proof}
\begin{remark}
Strictly speaking, achieving regularity for teardrops with domain-independent almost complex structures is slightly more than we will need. In the gluing arguments for teardrops below, however, we will make use of domain-dependent almost complex structures which agree with a fixed almost complex structure on the boundary. This will make it possible to choose perturbation data inductively, while still allowing for the attaching of pseudoholomorphic teardrops. We can still achieve regularity for generic such choices of domain-dependent almost complex structures, by modifying the standard transversality arguments appearing in~\cite[\S{9}]{SeidelBook} (see~\cite[\S{4}]{PalmerWoodward} as well).
\end{remark}
\begin{definition}
Define $\mathcal{M}(x_0,x_1)$ to be $\widehat{\mathcal{M}}(x_0,x_1)/\mathbb{R}$ whenever $\deg(x_0)-\deg(x_1)>0$, and set $\mathcal{M}(x_0,x_1) = \emptyset$ if not.
\end{definition}

Having fixed perturbation data for inhomogeneous strips, we can now explain how to choose perturbation data for disks with more boundary punctures consistently with the choices made thus far. In the following definition, let $Z_{\pm}$ denote the positive and negative half-strips of $Z$, respectively. To discuss disks with corners on a Lagrangian immersion, we use some terminology from~\cite{AlstonBao} for boundary punctures. We will divide the boundary punctures on a disk into two types. Type 1 punctures will be asymptotic to Hamiltonian chords, and Type 2 punctures will be asymptotic to switching components of the Lagrangian immersion.
\begin{definition}\label{perturbationdatum}
Let $d\geq1$ and $m\geq0$ be integers such that $d+m\geq2$, and let $S$ be a stable disk with one negative strip-like end, $d$ positive strip-like ends which are said to be of Type 1, and $m$ positive strip-like ends which are said to be of Type 2. Let $\zeta_0,\zeta_1,\ldots,\zeta_d$ denote the Type 1 marked points and $\eta_1,\ldots,\eta_m$ denote the Type 2 marked points. We assume that both sets of marked points are ordered cyclically on the boundary of $S$. A \textit{perturbation datum} for $S$ consists of:
\begin{itemize}
\item[(i)] Strip-like ends $\epsilon_i^1\colon Z_{+}\to S$ and $\epsilon_j^{2}\colon Z_{+}\to S$ for $i = 1,\ldots,d$ and $j=1,\ldots,m$, and $\epsilon_0^{1}\colon Z_{-}\to S$ near the corresponding boundary punctures.

\item[(ii)] A time-shifting function $\rho_S\colon\partial S\to[1,\infty)$ which takes the value $1$ away from the Type 1 marked points, and is constant of value $w_{i,S}$ near the Type 1 marked points for $i = 0,1,\ldots,d$.

\item[(iii)] A one-form $\alpha_S$ which vanishes along $\partial S$, and domain-dependent Hamiltonian $H_S\colon S\to\mathcal{H}(M)$ defining a Hamiltonian vector field $X_S$ such that $X_S\otimes\alpha_S$ pulls back to $X_{H/w_{i,S}\circ\phi^{w_{i,S}}}\otimes dt$ in each strip-like end near the Type 1 marked points.

\item[(iv)] A domain-dependent almost-complex structure which agrees with $J$ along the boundary, and pulls back to $J_t$ in the strip-like ends near the Type 1 marked points, and pulls back to $J$ in the ends near the Type 2 marked points.
\end{itemize}
\end{definition}
\begin{figure}
\begin{tikzpicture}[scale  = 2]
\draw (0,0) circle (1.0);

\draw[fill = white] (0,1) circle (0.25ex);
\draw[fill = white] (1,0) circle (0.25ex);
\draw[fill = white] (0,-1) circle (0.25ex);
\draw[fill = black] (-1,0) circle (0.25ex);
\draw[fill = black] (0.707,-0.707) circle (0.25ex);
\draw[fill = white] (-0.707,-0.707) circle (0.25ex);

\node[anchor = south] at (0,1) {$\zeta_0$};
\node[anchor = west] at (1,0) {$\zeta_3$};
\node[anchor = north west] at (0.707,-0.707) {$\eta_2$};
\node[anchor =  north] at (0,-1) {$\zeta_2$};
\node[anchor = east] at (-1,0) {$\eta_1$};
\node[anchor = north east] at (-0.707,-0.707) {$\zeta_1$};
\end{tikzpicture}
\caption{A stable disk with $3$ positive punctures of Type 1 and $2$ negative punctures of Type 2.}
\end{figure}
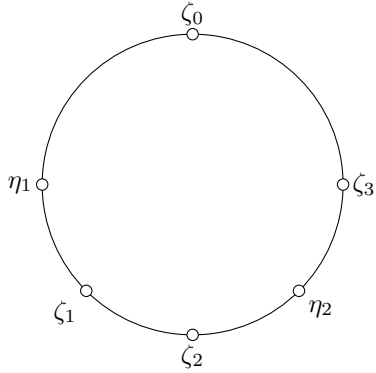

Our definition of the wrapped $A_{\infty}$-operations will count disks whose domains only have Type 1 marked points. Let $(\widetilde{L}_0,\nabla_0),\ldots,(\widetilde{L}_d,\nabla_d)$, where $d\geq2$, be a sequence of Lagrangian branes in $\mathrm{Ob}(\mathcal{W})$. It could be the case that $\widetilde{L}_i = \widetilde{L}_j$ for some $1\leq i\neq j\leq d$, but by Assumption~\ref{nondegchords}, all Hamiltonian chords between these two Lagrangian immersions are nondegenerate. Consider a sequence of Hamiltonian chords $x_i\in\mathcal{X}(L_{i-1},L_i)$ for all $i = 1,\ldots,d$, and a chord $x_0\in\mathcal{X}(L_0,L_d)$.

\begin{definition}
Define $\mathcal{M}_{d+1}(x_0,x_1,\ldots,x_d)$ to be the moduli space consisting of all maps $u\colon S\to M$, where 
\begin{itemize}
\item[(i)] $S$ is a stable disk with $d$ positive ends of Type 1, one negative end, and no ends of Type 2, and which is equipped with a choice of perturbation data.
\item[(ii)] $u\colon S\to M$ satisfying the boundary conditions
\begin{align*}
\begin{cases}
u(z)\in\phi^{\rho_S(z)}L_i & \text{for} \; z \; \text{between} \; \zeta_i \text{ and }\zeta_{i+1}\\
u(z)\in\phi^{\rho_S(z)}L_d & \text{for} \; z \; \text{between} \; \zeta_d \text{ and }\zeta_{0}\\
\lim_{s\to\pm\infty} u\circ\epsilon^1_i(s,\cdot) = \phi^{\rho_S(z)}x_i
\end{cases}
\end{align*}
and
\item[(iii)] $u$ satisfies the inhomogeneous Cauchy--Riemann equation
\[ (du-X_S\otimes\alpha_S)^{0,1}.\]
\end{itemize}
\end{definition}
Similarly to~\cite[Lemma 4.4]{AbouzaidGen}, we can choose perturbation data consistently. More precisely, we say that two perturbation data $(\rho_S^i,\alpha_S^i,H_S^i,J_S^i)$ on $S$ (where $i = 1,2$) are \textit{conformally consistent} if there is some constant $C>1$ such that
\begin{itemize}
\item[(i)] $\rho_S^2 = C\rho_S^1$
\item[(ii)] $\alpha_S^2 = C\alpha_S^1$
\item[(iii)]
\[ H_S^2 = \frac{1}{C^2}H_S^1\circ\psi^C \]
\item[(iv)] $J_S^2 = (\psi^C)^*J_S^1$
\end{itemize}
where $\psi^C$ is the time $\log(C)$ Liouville flow.
\begin{lemma}
There is a universal and conformally consistent, in the sense of~\cite{AbouzaidGen}, choice of perturbation data. This choice of data has the property that all of the moduli spaces $\mathcal{M}(x_0,x_1,\ldots,x_d)$ defined above are smooth oriented manifolds of the expected dimension.
\end{lemma}
\begin{proof}
The only salient difference between our situation and that of~\cite{AbouzaidGen} is that we need to achieve regularity using domain-dependent almost complex structures on $S$ which are constant along the boundary $\partial S$. We sketch the argument for the existence of regular perturbation data alluded to in~\cite[Lemma 4.4]{AbouzaidGen}, which itself follows~\cite[(9k)]{SeidelBook}, as well as the required modifications.

In~\cite{SeidelBook}, perturbation data are associated to a universal family $\mathcal{S}^{d+1}\to\mathcal{R}^{d+1}$ of $(d+1)$-pointed disks (see~\cite[(9.6)]{SeidelBook} for a definition of this family). One can make a universal choice of strip-like ends for this family~\cite[(9g)]{SeidelBook}. A \textit{universal choice of perturbation data} is the choice, for every $d\geq 2$ and every $(d+1)$-tuple $(L_0,\ldots,L_d)$ of exact Lagrangian branes of perturbation data on the family $\mathcal{S}^{d+1}\to\mathcal{R}^{d+1}$ which restrict to the given choice of perturbation data over the strip-like ends. One can show (cf.~\cite[Lemma 4.3]{AbouzaidGen} or~\cite[Lemma 9.5]{SeidelBook}) that universal choices of conformally consistent perturbation data exist. We remark that we can take these universal choices of perturbation data to satisfy Definition~\ref{perturbationdatum}(iv), with no changes to the arguments.

One now proves the existence of regular universal choices of conformally consistent perturbation data by induction on the number of (Type 1) marked points. So suppose that we have already chosen perturbation data for all $d\leq d_0$, and consider some Lagrangian labels $(L_0,\ldots,L_{d_0})$. We will construct an infinitesimal deformation of the perturbation datum $(H_{L_0,\ldots,L_{d_0}},J_{L_0,\ldots,L_{d_0}})$ associated to this tuple in such a way that the resulting moduli spaces are regular. By construction, these deformations only change $(H_{L_0,\ldots,L_{d_0}},J_{L_0,\ldots,L_{d_0}})$ in some open subset $\Omega\subset\mathcal{S}^{d_0+1}$. For each $r\in\mathcal{R}^{d_0+1}$, the intersections $\Omega_r = \Omega\cap\mathcal{S}^{d_0+1}_r$ must be nonempty and must satisfy
\begin{description}
\item[(i)] $\Omega_r$ is disjoint from the strip-like ends; and

\item[(ii)] for certain $r\in\mathcal{R}^{d_0+1}$, the intersection $\Omega_r$ is disjoint from the strip-like ends of $\mathcal{S}^{d_0+1}_r$
\end{description}
as detailed in~\cite[p. 128]{SeidelBook}. The infinitesimal perturbation data considered vanish outside $\Omega$, and hence the perturbation data restricted to each fiber vanish outside of $\Omega_r$. Finally, note that we can choose $\Omega$ so that each $\Omega_r$ is disjoint from the boundary of $\mathcal{S}^{d_0+1}_r$. From here, the argument of~\cite{SeidelBook} applies through with no changes.
\end{proof}
The $A_{\infty}$-operations will count elements of $\mathcal{M}_{d+1}(x_0,x_1,\ldots,x_d)$ as in~\cite{AbouzaidGen}. To show that these operations are well-defined and satisfy the $A_{\infty}$-relations, we need to discuss the Gromov compactifications of these moduli spaces. Elements of the boundary of these moduli spaces involve disks with corners on the switching components of immersed Lagrangian submanifolds, as described in the following definition.
\begin{definition}
Let $S$ be a disk with $d\geq1$ positive ends of Type 1 and $m\geq0$ positive ends of Type 2. A disk with corners on a sequence of Lagrangian immersions $\iota_i\colon\widetilde{L}_i\to M$ for $i = 0,1,\ldots,d$ and with inputs $x_i\in\mathcal{X}(L_i,L_{i+1})$ and output $x_0\in\mathcal{X}(L_0,L_d)$ is a tuple $(S,u,\widetilde{\partial u},\gamma)$, where
\begin{itemize}
\item[(i)] $u\colon S\to M$ is continuous.
\item[(ii)] $u(z)\in \phi^{\rho_S(z)}L_i$ for all $z$ between $\zeta_i$ and $\zeta_{i+1}$ and $u(z)\in\phi^{\rho_S(z)}L_d$ for all $z$ between $\zeta_d$ and $\zeta_0$.
\item[(iii)] $\widetilde{\partial u}\colon\partial S\to\coprod_{i=0}^d\widetilde{L}_i$ is a continuous map whose composition with the Lagrangian immersion $\coprod_{i=0}^d\iota_i$ coincides with the restriction of $u$ to $\partial S$.
\item[(iv)] The restriction of $u$ to the interior of $S$ satisfies the inhomogeneous Cauchy--Riemann equation $(du-X_S\otimes\alpha_S)^{0,1}$, which is defined with respect to the perturbation data already chosen.
\item[(v)] $\gamma$ is a function whose domain is a subset $I\subset\lbrace 1,\ldots,m\rbrace$, and whose value at $j$ is a switching component $L_{i,\gamma(j)}$ of the immersion $\widetilde{L}_i\to M$, where the marked point $\eta_j$ is between $\zeta_i$ and $\zeta_{i+1}$ (or it is a switching component of $L_d$ if $\eta_j$ is between $\zeta_d$ and $\zeta_0$).
\item[(vi)] Near the Type 1 ends, $u$ converges exponentially to $\phi^{\rho_S(z)}x_i(\cdot)$, where  for $i = 0,1,\ldots,d$.
\item[(vii)] Near the Type 2 ends, we have that
\[ \left(\lim_{z\to\eta_j^{-}}\widetilde{\partial u}(z),\lim_{z\to\eta_j^{+}}\widetilde{\partial u}(z)\right)\in L_{i,\gamma(j)}\]
if $j\in I$, and if $j\not\in I$ then $u$ extends smoothly over the puncture.
\end{itemize}
We write $\mathcal{M}(x_0,\ldots,x_d,\gamma)$ for the moduli space of all such stable disks. In the special case where $d=1$ and $m=0$, this definition reduces to the definition of inhomogenous strips from before, otherwise all such disks are stable.
\end{definition}
By~\cite[Prop.  7.5 and Prop. 8.5]{AlstonBao}, these moduli spaces are smooth manifolds of the expected dimension
\begin{align}
\deg x_0 - \sum_{i=1}^d\deg x_i-\sum_{j=1}^m\deg\gamma(j)+d+m-2
\end{align}
for generic choices of perturbation data. They are also oriented, by the arguments of~\cite{FOOO, Fukaya}, since by rescaling by Hamiltonian isotopies, we can treat inhomogeneous polygons as pseudoholomorphic polygons with corners on cleanly intersecting Lagrangian submanifolds. Following~\cite{AlstonBao}, we can think of elements in the compactification of $\mathcal{M}(x_0,x_1,\ldots,x_d)$ in terms of trees whose vertices are labeled by elements of the moduli spaces described above with teardrops glued to each Type 2 end.
\begin{lemma}[{\cite[Prop. 7.6]{AlstonBao}}]
The moduli spaces $\mathcal{M}(x_0,x_1,\ldots,x_d)$ admit Gromov compactifications denoted $\overline{\mathcal{M}}(x_0,x_1,\ldots,x_d)$. Elements of the compactified moduli space can be written as
\[ (T,\lbrace u_i,F_i,v_1^i,\ldots,v_{m_i}^i\rbrace_{i\in\mathrm{Vert}(T)}) \]
where 
\begin{itemize}
\item[(i)] $T$ is a tree with vertex set $\mathrm{Vert}(T)$ and a distinguished root
\item[(ii)] $u_i\in\mathcal{M}(x_0^i,\ldots,x_{d_i}^i;\gamma_i)$
\item[(iii)] $F_i\subset\lbrace 1,\ldots,d_i\rbrace$ for nonroot vertices, and $F_r\subset\lbrace0,\ldots,d_r\rbrace$ for the root vertex $r\in\mathrm{Vert}(T)$. The sets $F_i$ label the chords which are among the original chords $x_0,x_1,\ldots,x_d$.
\item[(iv)] $v_1,\ldots,v_{m_i}$ are elements of moduli spaces $\mathcal{M}_1(L,\gamma_i(j))$.
\end{itemize}
These data should satisfy additional constraints as detailed in~\cite[Prop. 7.6]{AlstonBao} which say that the components $u_i$ and $v_1^i,\ldots,v_{m_i}^i$ can be glued along their domains.
\qed
\end{lemma}
This lemma is stated for Lagrangian submanifolds with transverse double points, which itself is a modification of the compactness argument in~\cite{AkahoJoyce}, but the proof given in~\cite{AlstonBao} extends to the clean self-intersection case as well~\cite[Prop. 8.5]{AlstonBao}. We remark that the proofs of Prop. 8.5 and its analogues for Lagrangian immersions with only transverse double points in~\cite{AlstonBao} do not use the positivity Condition 1.2 therein.

\begin{definition}
For each pair of unobstructed Lagrangian branes $(\widetilde{L}_0,\nabla_0)$ and $(\widetilde{L}_1,\nabla_1)$ which are cylindrical at infinity, we define the wrapped Floer cochain complex to be
\[ CW^*(L_0,L_1;H;J_t)\coloneqq\bigoplus_{x\in\mathcal{X}(L_0,L_1)}\mathbb{K}\langle x\rangle\]
as a graded vector space, where the grading of the $\mathbb{K}\langle x\rangle$ summand is $\deg(x)$.
\end{definition}
Let $\beta\in H_2(M,L_0\cup L_1)$, and let $\mathcal{M}(x_0,x_1;\beta)$ denote the elements of $\mathcal{M}(x_0,x_1)$ in this relative homology class. We define the differential
\begin{align*} \mathfrak{m}_1\colon CW^*(L_0,L_1) &\to CW^*(L_0,L_1) \\
\mathfrak{m}_1(x_0) &= (-1)^{\deg(x_0)}\sum_{\substack{\deg(x_0)-\deg(x_1) = 1, \\ u\in\mathcal{M}(x_0,x_1)}}\hol(\partial u)\mathrm{sgn}(u)x_1
\end{align*}
where the holonomy $\hol(\partial u)\in\mathbb{K}^*$ is taken with respect to the flat connections on $L_0$ and $L_1$ and the sign $\mathrm{sgn}(u)\in\lbrace\pm1\rbrace$ is determined by the orientation of the moduli space. Similarly, we define the higher compositions
\begin{align*}
\mathfrak{m}_d\colon CW^*(L_{d-1},L_d)\otimes\cdots\otimes CW^*(L_0,L_1) \to CW^*(L_0,L_d)
\end{align*}
by
\begin{align*}
\mathfrak{m}_d(x_d\otimes\cdots\otimes x_1) = (-1)^{\sum i\deg(x_i)}\sum_{\substack{\deg(x_0) = 2-d+\sum\deg(x_i) \\ u\in\mathcal{M}(x_0,\ldots,x_d)}}\hol(\partial u)\mathrm{sgn}(u)x_0,
\end{align*}
where $\hol(\partial u)$ and $\mathrm{sgn}(u)$ are defined as before. Note that there is a canonical isomorphism
\[ CW^*(L_i,L_j;H,J_t)\cong CW^*\left(\phi^{\rho}L_i,\phi^{\rho}L_j;\frac{H}{\rho}\circ\phi^{\rho},(\phi^{\rho})^*J_t\right) \]
which we use to ensure that $\mathfrak{m}_d$ has the correct domain and codomain. For more details on this point, see~\cite{AbouzaidGen}.

Showing that these operations are well-defined amounts to showing that the moduli spaces of disks with no Type 2 punctures are compact.
\begin{lemma}
The moduli spaces $\mathcal{M}_{d+1}(x_0,\ldots,x_d)$ are compact when $\deg x_0 = \sum\deg x_i+2-d$. Consequently, the operations $\mathfrak{m}_d$ are well-defined for all $d\geq1$.
\end{lemma}
\begin{proof}
\textit{A priori}, an element in the compactification of $\mathcal{M}_{d+1}(x_0,x_1,\ldots,x_d)$ is described by a treed disk
\begin{align*}
(T,\lbrace u_i,F_i,v_1^i,\ldots,v_{m_i}^i\rbrace_{i\in\mathrm{Vert}(T)})
\end{align*}
where, by regularity, each of the $u_i$'s must belong to a moduli space of nonnegative dimension. This implies that
\begin{align}\label{0dregineq}
\deg(x_{i,0})-\sum_{j=1}^{d_i}\deg(x_{i,j})-\sum_{j=1}^{m_i}\deg(\gamma_i(j))+d_i+m_i-2\geq0 \,.
\end{align}
We claim that $\deg(\gamma_i(j))\geq2$ for all $j = 1,\ldots,m_i$. To see this, observe that if $L_{\gamma_i(j)}$ is $0$-dimensional, then, the fact that $\deg\gamma_i(j)\geq2$ follows from the regularity of $v^i_j$, which must lie in a moduli space of nonnegative dimension, which itself follows from Lemma~\ref{teardropregularity}. If $L_{\gamma_i(j)}$ is positive-dimensional, this is the content of Definition~\ref{monotonicity}. We can now rewrite~\eqref{0dregineq} to obtain
\[ \deg(x_{i,0})-\sum_{j=1}^{d_i}\deg(x_{i,j}) = \sum_{j=1}^{m_i}\deg\gamma_i(j)-d_i-m_i+2 \geq m_i-d_i+2 \,.\]

By summing over the verties of $T$, it follows that for any element in the boundary of the Gromov compactification $\overline{\mathcal{M}}(x_0,\ldots,x_d)$, we have the inequality
\[ 2-d = \deg(x_0)-\sum_{j=1}^d\deg(x_j)\geq \sum_{i=1}^{|\mathrm{Vert}(T)|}m_i+2|\mathrm{Vert}(T)|-\sum_{i=1}^{|\mathrm{Vert}(T)|}d_i.\]
Using that $|\mathrm{Vert}(T)|-1+d = \sum_{i=1}^{|\mathrm{Vert}(T)|}d_i$, this becomes
\[ 2-d\geq\sum_{i=1}^{|\mathrm{Vert}(T)|} m_i+|\mathrm{Vert}(T)|+1-d\]
or equivalently
\[ 1\geq\sum_{i=1}^{|\mathrm{Vert}(T)|} m_i+|\mathrm{Vert}(T)|.\]
Since $|\mathrm{Vert}(T)|$ is always at least $1$, it follows that $m_i = 0$ for all $i = 1,\ldots,|\mathrm{Vert}(T)|$. In other words, any element in the compactification of $\mathcal{M}_{d+1}(x_0,\ldots,x_d)$ is represnted by a inhomogeneous disk with no Type 2 marked points.
\end{proof}
To show that these operations satisfy the $A_{\infty}$-relations, we will use the compactness of the moduli spaces of virtual dimension $1$.
\begin{proposition}
The operations $\lbrace\mathfrak{m}_d\rbrace_{d\geq1}$ defined above satisfy the $A_{\infty}$-relations
\[ \sum_{\substack{d_1+d_2 = d+1 \\ 0\leq k <d_1}}(-1)^{\star_k}\mathfrak{m}_{d_1}(x_d,\ldots,x_{k+d_2+1},\mathfrak{m}_{d_2}(x_{k+d_2},\ldots,x_{k+1}),\ldots,x_k,\ldots,x_1) = 0 \]
where the sign is determined by setting $\star_k = \sum_{i=1}^{k=1}(\deg(x_i)+1)$.
\end{proposition}
\begin{proof}
As usual, we need to describe the boundaries of the $1$-dimensional moduli spaces $\overline{\mathcal{M}}(x_0,\ldots,x_d)$. The argument is similar to the argument of the previous lemma. Suppose we are given an element
\[ (T,\lbrace u_i,F_i,v_1^i,\ldots,v_{m_i}^i\rbrace_{i\in\mathrm{Vert}(T)}) \]
in the boundary of $\overline{\mathcal{M}}(x_0,\ldots,x_d)$. Each of the $u_i$'s must belong to a moduli space of nonnegative dimension, whence
\[ \deg(x_{i,0})-\sum_{j=1}^{d_i}\deg(x_{i,j})-\sum_{j=1}^{m_i}\deg(\gamma_i(j))+d_i+m_i-2\geq0.\]
Just as in the proof of the previous lemma, we can rewrite this (using the fact that regularity and Assumption~\ref{restrictionassumption} imply that $\deg\gamma_i(j)\geq 2$) as
\[ \deg(x_{i,0})-\sum_{j=1}^{d_i}\deg(x_{i,j}) = \sum_{j=1}^{m_i}\deg\gamma_i(j)-d_i-m_i+2 \geq m_i-d_i+2 \,.\]
Since we are considering $1$-dimensional moduli spaces, we have that $x_0,\ldots,x_d$ satisfy $\deg(x_0)-\sum\deg(x_j)-1= 2-d$. Summing over the vertices of $T$ and rearranging as above, we obtain the inequality\
\[ 2\geq\sum_{i=1}^{|\mathrm{Vert}(T)|}m_i+|\mathrm{Vert}(T)|.\]
Since $\mathrm{Vert}(T)\geq1$, we have that $\sum_i m_i \leq1$. This means that either $|\mathrm{Vert}(T)| = 1$ and $m_1 = 1$, or that $|\mathrm{Vert}(T)| = 2$. The latter case corresponds to boundary strata consisting of broken strips. 

The former case corresponds to boundary strata consisting of an inhomogeneous disk $u$ with a corner on a switching component $L_{\gamma}$, and a pseudoholomorphic teardrop attached. A gluing argument now shows that any teardrop $v$ with a corner on $L_{\gamma}$ gives rise to a boundary stratum of the form
\[ (T,\lbrace u,F,v\rbrace)\]
where $T$ is the tree with one vertex and $F = \lbrace x_0,\ldots,x_d\rbrace$. The gluing arguments in this case reduce to standard gluing arguments for strips, as explained on~\cite[p. 43]{PalmerWoodward}. Since $\mathfrak{m}_0(\widetilde{L},\nabla)$ vanishes, it follows that the algebraic count, weighted by holonomy, of all strata of this form vanishes.
\end{proof}

\begin{definition}
Given the choice of object $\mathrm{Ob}(\mathcal{W})$ from before, we define the wrapped Fukaya category $\mathcal{W}(M)$ to be the $A_{\infty}$-category with this set of objects and composition maps $\lbrace\mathfrak{m}_d\rbrace_{d=1}^{\infty}$.
\end{definition}

\section{Background on SYZ mirror symmetry}\label{backgroundSYZ}
In this section, we will review the setup of SYZ mirror symmetry for Lagrangian torus fibrations without singularities, mainly for the purpose of fixing notation. We will then collect some definitions from tropical geometry and review the construction of the Lagrangian pair-of-pants following~\cite{Matessi}, ending with a computation of its Floer-theoretic support.
\subsection{Lagrangian torus fibrations}
Let $Q_{\mathbb{Z}}\cong\mathbb{Z}^n$ be an integer lattice, and define $Q = Q_{\mathbb{R}}\coloneqq Q_{\mathbb{Z}}\otimes_{\mathbb{Z}}\mathbb{R}$. Then there is a local system $T^*_{\mathbb{Z}}Q$ of integral $1$-forms. Given this data, there is an associated symplectic manifold $X\coloneqq T^*Q/T^*_{\mathbb{Z}}Q$ which admits a Lagrangian $T^n$-fibration $\pi\colon X\to Q$ induced by the bundle projection $T^*Q\to Q$. Notice that we can identify $X$ with $T^*T^n$, and that under this identification $\pi\colon T^*T^n\to Q$ is the projection map to the cotangent fiber. Let $\theta_i\coloneqq dq_i$, for $i = 1,\ldots,n$, denote the dual coordinates on the cotangent fiber of $T^*Q$. These descend to coordinates on the zero-section of $T^*T^n$, after identifying it with $T^*Q/T^*_{\mathbb{Z}}Q$, which we also refer to as $\lbrace\theta_i\colon i = 1,\ldots,n\rbrace$.

If $W\subset Q$ is an integral affine subspace, then there is a Lagrangian submanifold $L_W\coloneqq N^*W/N^*_{\mathbb{Z}}W\subset X$ called the periodized conormal bundle with the property that $\pi(L_W) = W$. If $\dim W = k$, then $L_W$ is a $T^{n-k}$-bundle over $W$. In particular, the Lagrangian lift of a point $q\in Q$ is just $L_q = \pi^{-1}(q)$.

There is a dual complex manifold $\check{X}\coloneqq TQ/T_{\mathbb{Z}}Q$, were $T_{\mathbb{Z}}Q$ is the local system of integral tangent vectors. The complex structure arises by identifying $TQ$ with $\mathbb{C}^n$ via
\begin{align*}
TQ &\to \mathbb{C}^n \\
\left(x,\sum_{j=1}^n y_j\frac{\partial}{\partial x_j}\right) &\mapsto (z_j)_{j=1,\ldots,n} = (x_j+iy_j)_{j=1,\ldots,n}
\end{align*}
The complex structure on $TQ$ descends to one on $\check{X}$ which respect to which the fibers $\check{\pi}(q)$ are totally real tori. Notice that $\check{X}$ comes with a natural map $\check{\pi}\colon\check{X}\to Q$ induced by the tangent bundle projection. Under the identification $TQ\cong\mathbb{C}^n$, we see that $X$ is identified with $(\mathbb{C}^*)^n$, and the map $\check{\pi}$ coincides with $\Log\colon(\mathbb{C}^*)^n\to\mathbb{R}^n$, which is given by $\Log(z_1,\ldots,z_n) = (\log|z_1|,\ldots,\log|z_n)$.

To an integral affine subspace $W\subset Q$, there is an associated complex submanifold $C_W\coloneqq TC/T_{\mathbb{Z}}C\subset\check{X}$ such that $\check{\pi}(C_W) = W$. This construction associates, to a point $q\in Q$, a point $\check{X}$, namely the zero vector in $T_q Q$. This is not the only complex submanifold of $\check{X}$ mapped to $q$ under $\check{\pi}$: any other point in $\check{\pi}^{-1}(q)$ has this property.

On the other hand, $q\in Q$ only has one Lagrangian lift in $X$. To obtain a family of $A$-model objects corresponding to the points of $\check{\pi}^{-1}(q)$, we can equip $L_q$ wit local systems. There is a bijection between pairs $(L_q,\nabla)$, where $\nabla$ is a $U(1)$-local system on $L_q$, and points $z\in\check{X}$. In particular, $\check{\pi}^{-1}(q)$ is identified with $\Hom(\pi_1(L_q),U(1))$, so we can think of $\check{\pi}^{-1}(q)$ as the dual torus to $L_q$. Similarly, equipping $L_W\subset X$ with $U(1)$-local systems corresponds to taking different almost complex lifts of $W$ in $\check{X}$.

We will need to consider exact Lagrangians equipped with local systems over a field $\mathbb{K}$ of arbitrary characteristic. With respect to the standard choice of primitive, the zero section $T^n\subset T^*T^n$ is exact. The space of $GL(1)$-local systems on $T^n$ is $\Hom(\pi_1(T^n),\mathbb{K}^*)\cong(\mathbb{K}^*)^n$, which can be thought of as a copy of the mirror space $(\mathbb{K}^*)^n$. We can make this more precise as follows.

Consider the wrapped Fukaya category $\mathcal{W}(T^*T^n)$ whose objects consist of exact Lagrangians which are cylindrical at infinity equipped with $GL(1)$-local systems. In this setting, we can appeal to homological mirror symmetry.
\begin{theorem}\label{hms}
There is an equivalence of derived categories
\[ \mathcal{F}\colon D^{\pi}\mathcal{W}(T^*T^n)\to D^b\Coh((\mathbb{K}^*)^n)\]
where the wrapped Fukaya category has coefficients in $\mathbb{K}$. \qed
\end{theorem}
To prove this, one checks that $\mathcal{W}(T^*T^n)$ is generated by a cotangent fiber $L_Q$, and that there are algebra isomorphisms 
\[ HW^*(L_Q,L_Q)\cong\mathbb{K}[t_1,\ldots,t_n,t_1^{-1},\ldots,t_n^{-1}]\cong\Hom_{D^b\Coh}(\mathcal{O}_{(\mathbb{K}^*)^n},\mathcal{O}_{(\mathbb{K}^*)^n}).\]
Because $\mathcal{O}_{(\mathbb{K}^*)^n}$ generates the $dg$-enhancement of $D^b\Coh(\mathbb{K}^*)^n)$, the theorem follows. Under this equivalence, objects $(T^n,\nabla)$ of $\mathcal{W}(T^*T^n)$, where $\nabla$ is a $GL(1)$-local system on $T^n$, correspond to skyscraper sheaves over points in $(\mathbb{K}^*)^n$. 
\begin{definition}\label{pointlocalsystems}
Let $p\in(\mathbb{K}^*)^n$. Define $\nabla_p\in\Hom(\pi_1(T^n),\mathbb{K}^*)$ to be the rank one $\mathbb{K}$-local system on the zero-section $T^n\subset T^*T^n$ for which the Lagrangian brane $(T^n,\nabla_p)$ corresponds, under Theorem~\ref{hms}, to the skyscraper sheaf $\mathcal{O}_p$.
\end{definition}
Given this, we can attempt to match objects of $D^{\pi}\mathcal{W}(T^*T^n)$ supported on $L_W$ to their images under the mirror functor. We remark that $L_W$ is an exact Lagrangian submanifold of $T^*T^n$ if and only if $W\subset\mathbb{R}^n$ is a \textit{linear} subspace.

If $W\subset Q$ is the $k$-dimensional linear subspace spanned by $\lbrace q_{n-k+1},\ldots,q_n\rbrace$, then we can write $L_W$ as the product of $W$ with the isotropic subtorus $T^{n-k}$ of $T^*Q/T^*_{\mathbb{Z}}Q$ spanned by $\lbrace\theta_1,\ldots,\theta_{n-k}\rbrace$. Said differently, $L_W$ is the conormal bundle $N^*T^{n-k}$ of this subtorus. We can determine the support of the mirror sheaf to $L_W$ following the discussion in ~\cite[Ex. 2.4.3]{HicksReal}.
\begin{lemma}\label{conormalsupport}
Let $\nabla$ be a $GL(1)$ local system on $N^*T^{n-k}$ with holonomy $\alpha_j$ about the circle in the $\theta_j$-direction for $1\leq j\leq n-k$. Then the mirror sheaf $\mathcal{F}(N^*T^{n-k},\nabla)$ on $(\mathbb{K}^*)^n$ is supported on the $k$-dimensional subtorus $\lbrace z_i = \alpha_i\vert 1\leq i\leq n-k\rbrace$.
\end{lemma}
\begin{proof}
Our strategy will be to find all local systems $\nabla_p\in\Hom(\pi_1(T^n),\mathbb{K}^*)$ for which
\[ HW^*((N^*T^{n-k},\nabla),(T^n,\nabla_p))\neq0.\]
Since $\mathcal{F}$ takes Floer cohmology algebras to Ext-algebras, this characterizes the support of the mirror sheaf.

For a small positive real constant $a$, consider the Hamiltonian
\[ H = \sum_{i=1}^{n-k} a\cos(\pi\theta_i) \]
on $T^*T^n$, and let $\phi$ denote its time-one Hamiltonioan flow. Then the intersection of $\phi(N^*T^{n-k})$ with $T^n$ is
\[ \phi(N^*T^{n-k})\cap T^n = \lbrace(\epsilon_1,\ldots,\epsilon_{n-k},0,\ldots,0)\mid\epsilon_i\in\lbrace0,1\rbrace\rbrace \]
and the index of each intersection point is $\sum_{i=1}^{n-k}\epsilon_i$. We index these intersection points $x_I$ by subsets $I\subset\lbrace 1,\ldots,n-k\rbrace$. We also write $x_I <_1 x_J$ whenever $J = I\cup\lbrace i\rbrace$ for some $i\in\lbrace 1,\ldots, n-k\rbrace$.

As a vector space, $CW^*(\phi(N^*T^{n-k}),T^n)$ coincides with the Morse complex $CM^*(T^{n-k})$ defined using the Morse function $H$ restricted to $T^{n-k}$. In particular, it is the exterior algebra of an $(n-k)$-dimensional vector space. The holomorphic strips connecting $x_I$ to $x_J$ correspond to Morse flowlines between critical points of $H$. When $x_I<_1 x_J$, there are exactly two holomorphic strips, which we denote by $u_{IJ}^{+}$ and $u_{IJ}^{-}$, by the sign they appear with in the Floer differential.

Let $\nabla_p$ be a local system on $T^n$ with holonomy $z_i$ in the $\theta_i$-direction. Suppose that $I<_1 J$ differ only in the $i$th factor. We can normalize $\nabla$ and $\nabla_p$ such that the intersection of $u_{IJ}^+$ with $N^*T^{n-k}$ is a path of holonomy $1$ and its intersection with $T^n$ is a path of holonomy $z_i$. Hence $u_{IJ}^{-}$ should intersect $N^*T^{n-k}$ in a path with holonomy $\alpha_i^{-1}$ and $T^n$ in a path of holonomy $1$. Therefore the Floer differential is characterized as follows.
\[ 
\langle d(x_I),x_J\rangle = \begin{cases} 
z_i - \alpha_i & I<_1 J \\
0 & I\not<_1 J
\end{cases}
\]
We now see that the Floer homology group will vanish unless $z_i = \alpha_i$ for all $i\in\lbrace 1,\ldots,n-k\rbrace$.
\end{proof}
Notice that when $\mathbb{K} = \mathbb{C}$ and $\alpha_j = 1$ for all $1\leq j\leq n-k$, the subvariety appearing in Lemma~\ref{conormalsupport} coincides with complex lift $C_W$ of $W\subset Q$. Equipping $L_W = N^*T^{n-k}$ with a nontrivial local system then corresponds to taking a translate of $C_W$ under the action of $(\mathbb{C}^*)^n$ on itself by multiplication. We will introduce the following notation for local systems on $L_W$.
\begin{definition}\label{curvelocalsystem}
Fix a subset $J\subset\lbrace 1,\ldots,n\rbrace$, and let $W\subset Q$ denote the subspace spanned by $\lbrace q_i\colon i \not\in J\rbrace$. Setting $C\coloneqq\lbrace z_j = \alpha_j\colon j\in J\rbrace$, let $\nabla_C$ denote the local system on $L_W^*$ whose holonomy in the $\theta_j$-direction is $\alpha_j$.
\end{definition}

\subsection{The Lagrangian pair of pants}
To achieve a geometric description of mirror symmetry in full generality, one should consider tropical subvarieties of $Q \cong\mathbb{R}^n$ rather than just affine subspaces.

A tropical curve $W\subset Q$ is a collection of $1$-dimensional rational convex polyhedral domains $\lbrace W_s\subset Q\rbrace$ and weights $\lbrace w_s\in\mathbb{Z}_{>0}\rbrace$ which are required to satisfy the following conditions.
\begin{description}
\item[(i)] The intersection $W_s\cap W_t$ is either empty or a boundary point of both $W_s$ and $W_t$.
\item[(ii)] At each boundary point $v\in W_s$, let $u_s\in T_{\mathbb{Z},v}Q$ denote the primitive integral vector tangent to $W_s$ at $v$. We require that
\[ \sum_{\lbrace s\mid v\in W_s\rbrace} w_s u_s = 0.\]
\end{description}

\begin{example} Consider the following rays in $\mathbb{R}^2$.
\begin{align*}
W_1 &= \lbrace(-t,0)\mid t\in\mathbb{R}_{\geq0}\rbrace \\
W_2 &= \lbrace(0,-t)\mid t\in\mathbb{R}_{\geq0}\rbrace \\
W_3 &= \lbrace(t,t)\mid t\in\mathbb{R}_{\geq0}\rbrace
\end{align*}
Their union $W$ is a tropical curve which we call the tropical pair of pants.
\end{example}
\begin{figure}
\begin{tikzpicture}
\begin{scope}[xshift = -30]
\draw[thick] (0,-2) -- (0,0) -- (-2,0);
\draw[thick] (0,0) -- (1.41,1.41);

\node[anchor = north west] at (1.44,1.44) {$(1,1)$};
\node[anchor = west] at (0,-2) {$(0,-1)$};
\node[anchor = south] at (-2,0) {$(-1,0)$};

\end{scope}

\begin{scope}[xshift = 60, yshift = -60, scale = 0.5]
\draw[thick] (0,0) -- (6,0) -- (6,6) -- (0,6) -- (0,0);
\filldraw[ultra thick, gray!45] (0,0) -- (3,0) -- (0,3);
\filldraw[ultra thick, gray!45] (6,6) -- (6,3) -- (3,6);

\node[] at (1,1) {$\Delta_+$};
\node[] at (5,5) {$\Delta_{-}$};
\end{scope}

\end{tikzpicture}
\caption{The tropical pair of pants (left) and its coamoeba (right).}\label{tropicalpants}
\end{figure}
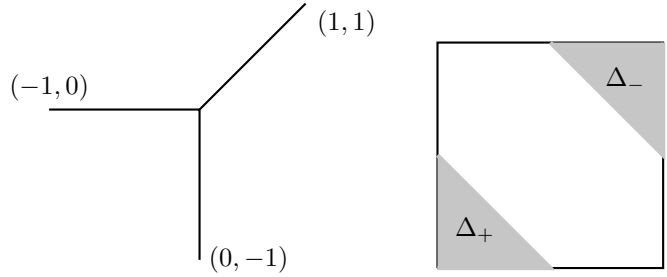
On the $B$-side, this tropical curve arises from the standard pair of pants 
\[C = \lbrace 1+z_1+z_2 = 0\rbrace\]
in $(\mathbb{C}^*)^2$ as the limit of the amoebae $\Log_t(C)\subset\mathbb{R}^2$ as $t\to\infty$. There is a corresponding Lagrangian submanifold in $T^*T^2$. In the following, let $\pi_{SYZ}\colon T^*T^2\to Q$ denote projection to the cotangent fiber.
\begin{theorem}[\cite{Matessi, HicksUnobstructed, Mikhalkin}]
For any $\epsilon>0$, there is a Lagrangian pair of pants $\Lpants\subset T^*T^2$ which agrees with the conormal lifts of the cones $W_i$ when restricted to the subset $\pi^{-1}_{SYZ}(\mathbb{R}^2\setminus B_{\epsilon}(0))$.
\end{theorem}
A Lagrangian lift of this tropical curve can be constructed by applying a hyperK{\"a}hler rotation to $C$~\cite{Mikhalkin}, or as a mapping cone in the Fukaya category corresponding to a mapping cone of vector bundles in $D^b\mathrm{Coh}((\mathbb{C}^*)^2)$ defining $\mathcal{O}_C$~\cite{HicksUnobstructed}. We will briefly review Matessi's~\cite{Matessi} construction of the Lagrangian pair of pants, since this is the method of construction we will generalize later, and because the embedding $\Lpants\hookrightarrow T^*T^2$ can be made particularly explicit in this setup. This construction starts by considering the expected image of the projection $\Lpants\to T^2$ to the SYZ fiber.
\begin{definition}\label{coamoebapants}
Identifying $T^2 \cong \mathbb{R}^2/\mathbb{Z}^2$ with the usual quotient of the unit cube $[0,1]^2$. Consider the simplices $\overline{\Delta}_{\pm}\subset[0,1]^2$ defined as
\begin{align*}
\overline{\Delta}_{+} = \left\lbrace(\theta_1,\theta_2)\in[0,1]^2\mid\theta_1+\theta_2\leq\frac{1}{2}\right\rbrace \\
\overline{\Delta}_{-} = \left\lbrace(\theta_1,\theta_2)\in[0,1]^2\mid\theta_1+\theta_2\geq\frac{3}{2}\right\rbrace
\end{align*}
Let $\Delta_{\pm}$ denote the union of the $0$-cells and $2$-cells of $\overline{\Delta}_{\pm}$. Then the coamoeba $\Delta$ of $W$ is the union of the images of $\Delta_{\pm}$ in $T^2$.
\end{definition}
The Lagrangian $\Lpants$ can be thought of as the closure of the graph of an exact $1$-form on the interior of $\Delta$. To make this more precise, we will replace the coamoeba with a smooth pair of pants before attempting to define a Lagrangian embedding.

The real blowup $\widetilde{\Delta}$ of $\Delta$ is the space obtained by blowing up $\Delta$ at each of its vertices. More precisely, if $\theta_0$ is a vertex of $\Delta$ and $U_{\theta_0}\subset\Delta$ is a small neighborhood of $\theta_0$, we can form
\[\widetilde{U}_{\theta_0} = \lbrace(\theta,\ell)\in U_{\theta_0}\times\mathbb{RP}^{1}\mid \theta-\theta_0\in\ell\rbrace,\]
which comes with a natural projection map $\widetilde{U_{\theta_0}}\to U_{\theta_0}$. We construct $\widetilde{\Delta}^n$ by gluing $\widetilde{U_{\theta_0}}$ to $\mathrm{Int}(\Delta^n)$, which is just the complement of the vertices in $\Delta^n$. There is an obvious projection $\pi\colon\widetilde{\Delta}^n\to\Delta^n$.

Consider the function $g:\Delta\to\mathbb{R}$ defined by
\begin{align}\label{coamoebaprimitivepants}
g(\theta) = \begin{cases}
-\sqrt{\cos(\theta_1+\theta_2)\sin(\theta_1)\sin(\theta_2)}, \; \theta\in\Delta_+ \\
\sqrt{\cos(\theta_1+\theta_2)\sin(\theta_1)\cdots\sin(\theta_2)}, \; \theta\in \Delta_{-}.
\end{cases}
\end{align}
The following Lemma is proven by Matessi~\cite{Matessi}.
\begin{lemma}
The function $dg = \left(\frac{\partial g}{\partial\theta_1},\frac{\partial g}{\partial\theta_2}\right)$ extends to a smooth map $\widetilde{dg}\colon\widetilde{\Delta}\to\mathbb{R}^2$.
\end{lemma}
Matessi's Lagrangian embedding is $\Phi\colon\widetilde{\Delta}\to  T^*T^2 = T^2\times\mathbb{R}^2$ defined by $\Phi(\theta) = (\pi(\theta),\widetilde{dg}(\theta))$. 

Since $\Phi(\widetilde{\Delta})$ is exact and approaches the conormals to the top-dimensional cones of $W$ asymptotically, one can deform it via Hamiltonian isotopy to lie in tropical position. This means that it coincides with the conormal lifts to the top-dimensional cones of $W_n$ away from the $(n-2)$-skeleton of $W_n$. Since the vertex of $W_n$ is placed at the origin, this means that $W_n$ becomes an object of $\mathcal{W}(T^*T^n)$ when equipped with a $GL(1)$-local system. We can now compute the support of the mirror sheaf using the same strategy as in Lemma~\ref{conormalsupport}.

Observe that by construction $\Lpants$ is homotopy equivalent to $\widetilde{\Delta}$. Hence we can identify a set of generators for $H_1(\Lpants)$ with the generators $[\theta_1],\ldots,[\theta_n]$ of $H_1(T^n)$. Th compute Floer homology, one should also equip $\Lpants$ with a spin structure, but we leave this unspecified for now.
\begin{lemma}\label{pantssupport}
Suppose that $\Lpants$ is equipped with a choice of spin structure. Let $\nabla$ be a $GL(1)$-local system on $\Lpants$ with holonomy $\rho_i$ along $[\theta_i]$. If $\nabla_z\in\Hom(\pi_1(T^2),\mathbb{K}^*)$ has holonomy $z_i$ along $[\theta_i]$, then $HW^*((\Lpants,\nabla),(T^2,\nabla_z))$ is nonzero, in which case it is isomorphic to $H^*(S^1)$ as a graded vector space, if and only if 
\begin{equation}\label{pantsdiff}
\pm\rho_1^{-1}z_1\pm\rho_2^{-1}z_2\pm1 = 0.
\end{equation}
\end{lemma}
\begin{proof}
We identify $T^*T^2$ with $(\mathbb{C}^*)^2$ symplectically via $(q_i,\theta_i)\mapsto z_i\coloneqq(\log|q_i|,\theta_i)$, and we write $r_i\coloneqq\log|q_i|$. From the formula for $g$ it is easy to see that $\Lpants$ intersects $T^2$ transversely in two points corresponding to the barycenters of $\Delta_+$ and $\Delta_{-}$. Call these two intersection points $x_+$ and $x_{-}$. There are $3$ obvious holomorphic strips connecting $x_{+}$ to $x_{-}$, which are obtained as portions of holomorphic cylinders lying over lines in $Q$ through the origin in the directions of $-r_1-r_2$, $-r_1+2r_2$, and $2r_1-r_2$. These holomorphic strips intersect $T^2$ in the three geodesic paths connecting $x_+$ to $x_{-}$.
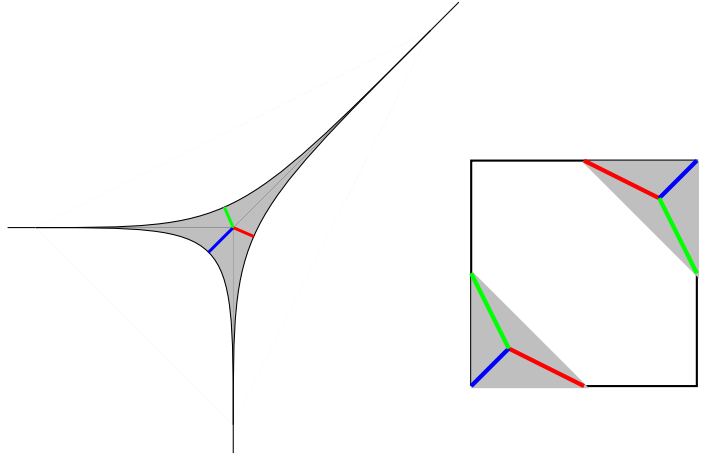
\begin{figure}
\begin{tikzpicture}
\begin{scope}[xshift = -30, scale = 0.75]
 \draw (0,0) -- (-4,0);
 \draw (0,0) -- (0,-4);
 \draw (0,0) -- (4,4);

\newcommand{\pathA}{(-3.5,0) .. controls (0,0) .. (0,-3.5)}
\newcommand{\pathB}{(-3.5,0) .. controls (0,0) .. (3.5,3.5)}
\newcommand{\pathC}{(0,-3.5) .. controls (0,0) .. (3.5,3.5)}
 
\fill[gray!50] \pathA -- (-3.5,0) -- (0,0) -- (0,-3.5) -- cycle;
\fill[white] \pathA;
\draw[white] (-3.5,0) -- (0,-3.5);

\fill[gray!50] \pathB -- (-3.5,0) -- (0,0) -- (3.5,3.5) -- cycle;
\fill[white] \pathB;
\draw[white] (-3.5,0) -- (3.5,3.5);

\fill[gray!50] \pathC -- (0,-3.5) -- (0,0) -- (3.5,3.5) -- cycle;
\fill[white] \pathC;
\draw[white] (0,-3.5) -- (3.5,3.5);

 \draw \pathA;
 \draw \pathB;
 \draw \pathC;
 
\draw[blue, line width = 0.4 mm] (0,0) -- (-0.4375,-0.4375);
\draw[red, line width = 0.4 mm] (0,0) -- (0.36338,-0.157689);
\draw[green, line width = 0.4 mm] (0,0) -- (-0.157689,0.36338);
\end{scope}

\begin{scope}[xshift = 60, yshift = -60, scale = 0.5]
\draw[thick] (0,0) -- (6,0) -- (6,6) -- (0,6) -- (0,0);
\filldraw[ultra thick, gray!50] (0,0) -- (3,0) -- (0,3);
\filldraw[ultra thick, gray!50] (6,6) -- (6,3) -- (3,6);

\node[] at (1,1) {};
\node[] at (5,5) {};

\draw[ultra thick, red] (1,1) -- (3,0);
\draw[ultra thick, red] (5,5) -- (3,6);
\draw[ultra thick, blue] (1,1) -- (0,0);
\draw[ultra thick, blue] (5,5) -- (6,6);
\draw[ultra thick, green] (1,1) -- (0,3);
\draw[ultra thick, green] (5,5) -- (6,3);
\end{scope}
\end{tikzpicture}
\caption{The boundary of the three holomorphic strips contributing to the Floer differential on $CW^*(\Lpants,T^2)$, projected to the amoeba (left), and the coamoeba (right).}\label{strips}
\end{figure}

We claim that the holomorphic strips described above are the only ones which contribute to the Floer differential. The Lagrangian submanifold $\Lpants$ constructed by Matessi, is exact, since it is obtained as the graph of an exact $1$-form on the coamoeba (cf.~\cite[\S 5.4]{Matessi}). The $0$-section $T^2$ is also an exact Lagrangian submanifold of $T^*T^2$. There is a Reimannian metric on $T^*T^2$ induced from the symplectic form and almost complex structure, and this pulls back to Riemannian metrics on $\Lpants$ and $T^2$. The metric on $T^2$ obtained this way is just the standard flat metric.

Any holomorphic strip $u\colon(-\infty,\infty)\times[0,1]\to T^*T^2$ with an input at $x_{+}$ and an output at $x_{-}$ has two boundary components: an arc $u(s,0)\subset T^2$ on $T^2$ and an arc on $u(s,1)\subset\Lpants$ on $\Lpants$. By the exactness of $T^2$ and $\Lpants$, it follows from Stokes' theorem that the area of any such strip is given by the sum of the lengths of these arcs. It also follows from exactness and Stokes' theorem that the area is determined by the values of the primitives for $T^2$ and $\Lpants$ at $x_{\pm}$. More precisely, if the primitives are denoted $f_{T^2}$ and $f_{\Lpants}$, then
\[ \int u^*\omega = (f_{\Lpants}(x_{+})-f_{T^2}(x_{+}))-(f_{\Lpants}(x_{-})-f_{T^2}(x_{-})) \,.\]
The paths on $T^2$ and $\Lpants$ depicted in Figure~\ref{strips} are the shortest paths connecting $x_{+}$ and $x_{-}$ in the respective metrics, meaning that any holomorphic strip joining $x_{+}$ to $x_{-}$ has the same boundary as one of the holomorphic strips already identified. By the maximum principle, there cannot be any other holomorphic strips.

Thus we need only compute the holonomy of the local systems along the boundaries of these strips. We can normalize the local systems on $\Lpants$ and $T^2$ such that parallel transport along the blue arc in Figure~\ref{strips} is multiplication by $1\in\mathbb{K}^*$. In that case, we must have that parallel transport along the orange and green arcs act as multiplication by $z_1$ and $\rho_1$ and by $z_2$ and $\rho_2$, respectively. Thus, without loss of generality, the Floer differential is
\[ x_+\mapsto\pm(\rho_1^{-1}z_1\pm\rho_2^{-1}z_2\pm1)x_{-}, \]
which completes the proof.
\end{proof}
\begin{remark}
For different choices of spin structure on $L_{pants}$, the differential in this chain complex can be any of $\pm\rho_1^{-1}z_1\pm\cdots\pm\rho_n^{-1}z_n\pm 1$ up to an overall sign.
\end{remark}
\begin{remark}
A similar computation in two dimensions using non-exact Lagrangian tori is carried out in~\cite{HicksReal}. The Floer differential we computed above also appears in~\cite{HicksDimer}.
\end{remark}
\section{Singular Lagrangian lift of a 4-valent vertex}\label{singularconstruction}
This section is devoted to the construction of a Lagrangian lift of a certain genus $0$ tropical curve $V\subset\mathbb{R}^3$. Specifically, $V$ is defined to be the union of the following four one-dimensional cones
\begin{align*}
V_1 = \lbrace(-t,0,0)\mid t\in\mathbb{R}_{\geq0}\rbrace, \\
V_2 = \lbrace(0,-t,0)\mid t\in\mathbb{R}_{\geq0}\rbrace, \\
V_3 = \lbrace(0,0,-t)\mid t\in\mathbb{R}_{\geq0}\rbrace, \\
V_4 = \lbrace(t,t,t)\mid t\in\mathbb{R}_{\geq0}\rbrace.
\end{align*}
This has one $4$-valent vertex at the origin. We will prove that
\begin{theorem}\label{singularlift}
There is a family of Hamiltonian isotopic singular Lagrangian submanifolds $\Lsing^{\epsilon}\subset T^*T^3$ for sufficiently small $\epsilon>0$, each with one singular point. The image of $\Lsing^{\epsilon}$ under projection to the cotangent fiber is $\epsilon$-close to $V$ in the Hausdorff metric. The link of the singular point of $\Lsing^{\epsilon}$ is Legendrian isotopic to the link of the Harvey--Lawson cone. For each $\Lsing^{\epsilon}$, there is a small open ball $B\subset\mathbb{R}^3$ centered at the origin such that
\[ \Lsing^{\epsilon}\mid_{T^*T^3\setminus\pi_{SYZ}^{-1}(B)} = \bigcup_{j=1}^4 L_{V_j}\mid_{T^*T^3\setminus\pi_{SYZ}^{-1}(B)}. \]
Away from the cone points, $\Lsing^{\epsilon}$ is diffeomorphic to the minimally twisted five-component chain link complement in $S^3$, drawn in Figure~\ref{intro-fig}.
\end{theorem}
\subsection{The minimally twisted five-component chain link}\label{chainlinkdescription} We begin by collecting some topological facts about the minimally-twisted five component chain link complement, hereafter denoted $L'$. The link complement $L'$ is a hyperbolic $3$-manifold, which one can see by exhibiting an ideal triangulation for it, following Dunfield--Thurston~\cite{DT}. Fix the labeling of the cusps of $M_5$ shown in Figure~\ref{orbifold}. The labeling we choose does not matter, since the group of hyperbolic isometries of $L'$ acts transitively on the cusps of $L'$.
\begin{figure}
\begin{tikzpicture}
\begin{scope}[xshift = -100]
\node[] at (3,0) {1};
\node[] at (0.93,2.86) {0};
\node[] at (-2.43,1.77) {4};
\node[] at (-2.43,-1.77) {3};
\node[] at (0.93,-2.86) {2};

\node[] at (0,-3.5) {$L'$};

\draw[gray, ultra thick] (0,0) circle (2.3);
\draw(2,0)[line width = 0.25 mm] ellipse (0.5 and 1.5);
\draw[rotate = 72, line width = 0.25 mm] (2,0) ellipse (0.5 and 1.5);
\draw[rotate = 144, line width = 0.25 mm] (2,0) ellipse (0.5 and 1.5);
\draw[rotate = 216, line width = 0.25 mm] (2,0) ellipse (0.5 and 1.5);
\draw[rotate = 288, line width = 0.25 mm] (2,0) ellipse (0.5 and 1.5);

\draw (2.5,0)[white, double = black, ultra thick] arc (0:100:0.5 and 1.5);
\draw (2.5,0)[white, double= black, ultra thick] arc (0:-100:0.5 and 1.5);
\draw[white, double = black, ultra thick, rotate = 72] (2.5,0) arc (0:100:0.5 and 1.5);
\draw[white, double = black, ultra thick, rotate = 144] (1.5,0) arc (180:240:0.5 and 1.5);
\draw[white, double = black, ultra thick, rotate = 72] (1.5,0) arc (180:240:0.5 and 1.5);

\draw[white, double = black, ultra thick, rotate = 144] (1.5,0) arc (180:120:0.5 and 1.5);

\draw[white, double = black, ultra thick, rotate = 216] (2.5,0) arc (0:-100:0.5 and 1.5);

\draw[white, double = black, ultra thick, rotate = 216] (2.5,0) arc (0:100:0.5 and 1.5);

\draw[white, double = black, ultra thick, rotate = 288] (1.5,0) arc (180:240:0.5 and 1.5);
\draw[white, double = black, ultra thick, rotate = 288] (1.5,0) arc (180:120:0.5 and 1.5);
\end{scope}

\begin{scope}[xshift = 75]
\node[] at (0,-3.5) {$N$};

 \coordinate (a) at (3,0,0);
 \coordinate (b) at (-1.5,2.598,0);
 \coordinate (c) at (-1.5,-2.598,0);
 \coordinate (d) at (0,0,2);
 \coordinate (e) at (0,0,0);
 
\draw[gray] (a) -- (b) -- (c) -- (a);
\draw[gray] (a) -- (d);
\draw[gray] (a) -- (e);
\draw[gray] (b) -- (d);
\draw[gray] (b) -- (e);
\draw[gray] (c) -- (d);
\draw[gray] (c) -- (e);
\draw[gray] (d) -- (e);

\draw[white, double = gray, ultra thick] (3,0,0) -- (0,0,2);

\filldraw[black] (a) circle  (0.1);
\filldraw[black] (b) circle  (0.1);
\filldraw[black] (c) circle  (0.1);
\filldraw[black] (d) circle  (0.1);
\filldraw[black] (e) circle  (0.1);

\node[anchor = west] at (a) {2};
\node[anchor = south east] at (b) {1};
\node[anchor = east] at (c) {4};
\node[anchor = east] at (d) {3};
\node[anchor = south west] at (e) {0};
\end{scope}

\end{tikzpicture}
\caption{The circle of symmetry (left) and its image in the quotient orbifold (right).}\label{orbifold}
\end{figure}
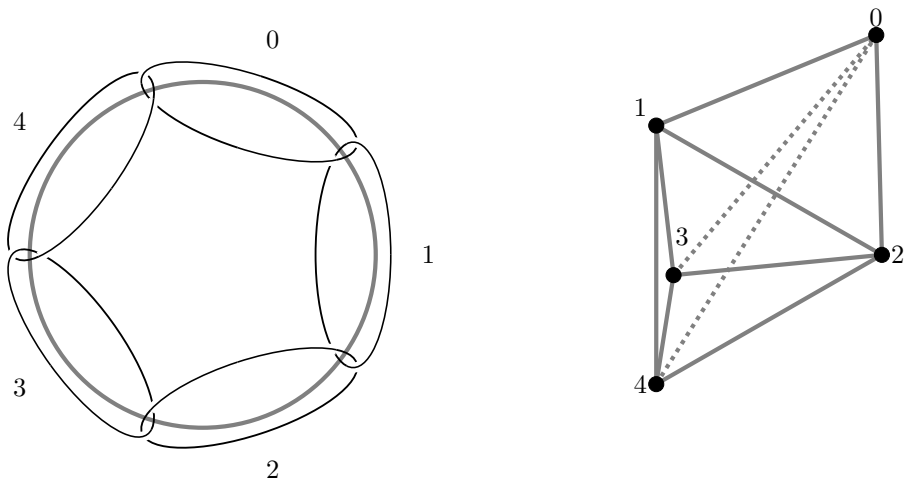

Observe that there is an involution of $L'$ given by rotation about the circle depicted in Figure~\ref{orbifold}. Note that this circle intersects the chain-link in $S^3$, so in $L'$ the fixed locus is a union of $10$ arcs between the cusps. Inspecting the diagram shows that any two cusps of $L'$ are joined by one of these arcs. The quotient $N$ of $L'$ by this involution is an orbifold whose underlying $3$-manifold is $S^3$ with five points removed, with each puncture in $S^3$ corresponding to the hyperelliptic quotient of a cusp of $L'$. The orbifold locus consists of an unknotted arc in $S^3$ connecting each pair of punctures, and all orbifold points are cone points of order $2$.

Viewing $N$ as the boundary $3$-sphere of a standard $4$-simplex carrying the induced triangulation with its vertices deleted gives us an ideal triangulation of $N$. The union of edges of the triangulation is the orbifold locus of $N$. It is shown in~\cite{DT} that these can be thought of as regular ideal tetrahedra in hyperbolic $3$-space $\mathbb{H}^3$, so this gives $N$ the structure of a hyperbolic orbifold decomposed into five regular ideal tetrahedra. We can lift this to an ideal triangulation of $L'$ consisting of ten regular ideal tetrahedra.

Since we will build $\Lsing^{\epsilon}$ using a coamoeba, it will be convenient to recast this ideal triangulation as a decomposition of $L'$ into two ideal cubes, which will correspond to the two pieces of the coamoeba. From our description of the ideal triangulation of $L'$, two of the ideal tetrahedra, which we call $T_{0,\pm}$, in the triangulation will have ideal vertices at the cusps labeled $1$ through $4$. These are the black tetrahedra in Figure~\ref{triangulation}. At each face of $T_{0,+}$, the adjacent tetrahedron will have one ideal vertex at the $0$th cusp. The union of these five ideal tetrahedra in $L'$ is an ideal cube. Similar remarks apply to $T_{0,-}$, implying that the union of the other five ideal tetrahedra is also an ideal cube.
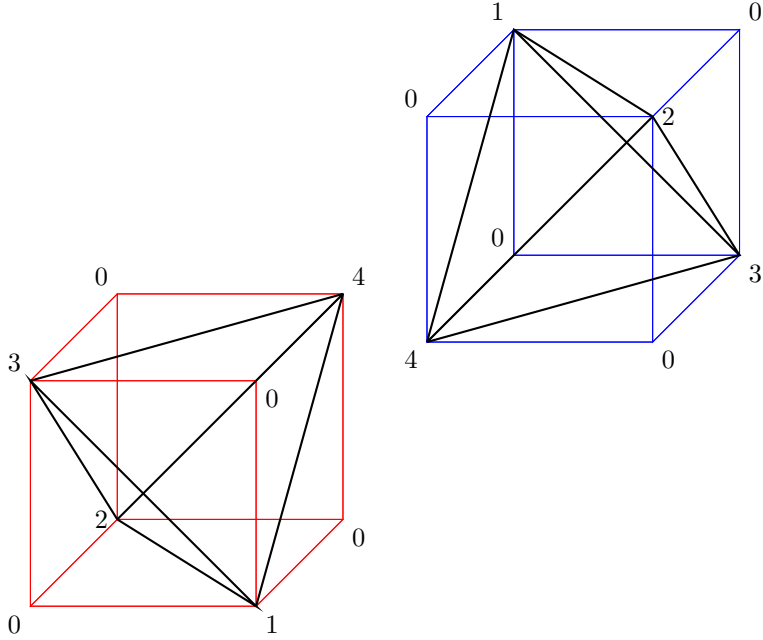
\begin{figure}
\begin{tikzpicture}
\begin{scope}[xshift =  75, yshift = 50, scale = 3]
\draw[blue] (0,0,0) -- (1,0,0) -- (1,1,0) -- (0,1,0) -- (0,0,0);
\draw[blue] (0,0,0) -- (0,1,0) -- (0,1,1) -- (0,0,1) -- (0,0,0);
\draw[blue] (0,0,0) -- (0,0,1) -- (1,0,1) -- (1,0,0) -- (0,0,0);
\draw[blue] (1,0,0) -- (1,1,0) -- (1,1,1) -- (1,0,1) -- (1,0,0);
\draw[blue] (0,1,0) -- (1,1,0) -- (1,1,1) -- (0,1,1) -- (0,1,0);
\draw[blue] (0,0,1) -- (1,0,1) -- (1,1,1) -- (0,1,1) -- (0,0,1);

\draw[thick, black] (0,0,1) -- (0,1,0) -- (1,0,0) -- (0,0,1);
\draw[thick, black] (0,1,0) -- (1,1,1);
\draw[thick, black] (1,0,0) -- (1,1,1);
\draw[white, double = black, ultra thick] (0,0,1) -- (1,1,1);

\node[anchor = west] at (1,1,1) {$2$};
\node[anchor = north east] at (0,0,1) {$4$};
\node[anchor = south east] at (0,1,0) {$1$};
\node[anchor = north west] at (1,0,0) {$3$};

\node[anchor = south east] at (0,0,0) {$0$};
\node[anchor = south west] at (1,1,0) {$0$};
\node[anchor = north west] at (1,0,1) {$0$};
\node[anchor = south east] at (0,1,1) {$0$};
\end{scope}

\begin{scope}[xshift = -75, yshift = -50, scale = 3]
\draw[red] (0,0,0) -- (1,0,0) -- (1,1,0) -- (0,1,0) -- (0,0,0);
\draw[red] (0,0,0) -- (0,1,0) -- (0,1,1) -- (0,0,1) -- (0,0,0);
\draw[red] (0,0,0) -- (0,0,1) -- (1,0,1) -- (1,0,0) -- (0,0,0);
\draw[red] (1,0,0) -- (1,1,0) -- (1,1,1) -- (1,0,1) -- (1,0,0);
\draw[red] (0,1,0) -- (1,1,0) -- (1,1,1) -- (0,1,1) -- (0,1,0);
\draw[red] (0,0,1) -- (1,0,1) -- (1,1,1) -- (0,1,1) -- (0,0,1);

\draw[thick, black] (0,0,0) -- (1,1,0);
\draw[thick, black] (0,0,0) -- (1,0,1);
\draw[white, double = black, ultra thick] (1,0,1) -- (0,1,1);
\draw[thick, black] (0,1,1) -- (0,0,0);
\draw[thick, black] (1,0,1) -- (1,1,0);
\draw[thick, black] (0,1,1) -- (1,1,0);
\node[anchor = east] at (0,0,0) {$2$};
\node[anchor = south west] at (1,1,0) {$4$};
\node[anchor = north west] at (1,0,1) {$1$};
\node[anchor = south east] at (0,1,1) {$3$};

\node[anchor = north west] at (1,1,1) {$0$};
\node[anchor = north east] at (0,0,1) {$0$};
\node[anchor = south east] at (0,1,0) {$0$};
\node[anchor = north west] at (1,0,0) {$0$};
\end{scope}
\end{tikzpicture}
\caption{The two ideal cubes in the decomposition of $L'$, decomposed into five ideal tetrahedra each. The edges of the tetrahedra $T_{0,\pm}$ are the black lines.}\label{triangulation}
\end{figure}
The next lemma summarizes our discussion thus far.
\begin{lemma}
$L'$ admits a decomposition glued into two ideal cubes glued together as shown in Figure~\ref{triangulation}, where opposite faces identified by reflections along the black lines, so that vertices with the same label are identified. \qed
\end{lemma}
The edges of the tetrahedra $T_{0,\pm}$ are the black lines in Figure~\ref{triangulation}.

It will be convenient to describe the spin structures on $L'$ that we will use in terms of the ideal triangulation, following Benedetti and Petronio~\cite{BP}, as we will now quickly review. Their description of ideal triangulations makes use of so-called special spines, which are certain $2$-dimensional complexes dual to a (possibly ideal) triangulation of a $3$-manifold. The special spine is the union $P$ of compact $2$-dimensional polyhedra dual to each maximal cell in the ideal triangulation. The singular set $S(P)$ of $P$ is a $4$-valent graph whose edges are dual to the $2$-dimensional faces of the triangulation on $L'$. 

To specify a spin structure, we will first need to smooth $P$ to obtain a new spine $Y$ which has a well-defined normal direction. A choice of smoothing is determined by the following combinatorial data associated to a triangulation. A branching on an oriented tetrahedron $T$ is an orientation of its edges so that none of its faces are cycles. In particular, each vertex of $T$ should have a different number of incoming edges. Each face of $T$ is given the orientation that restricts to the prevailing orientation of its edges.
\begin{lemma}\label{spinchoice}
The triangulation of $L'$ admits a branching.
\end{lemma}
\begin{proof}
Fix a branching of $T_{0,{\pm}}$ for which the $i$th vertex, where $i = 1,2,3,4$, has $i-1$ incoming edges. Any other tetrahedron in the ideal triangulation has a $0$th vertex, and we declare that all edges adjacent to this vertex are incoming. With these choices, it is easy to see that all edges have been given an orientation, and that no face in this triangulation is a cycle.
\end{proof}
As we will see momentarily, this choice of branching determines a spin structure, and when we compute Floer homology, we will always assume that $L'$ is equipped with this spin structure.
\begin{remark}
For an arbitrary $3$-manifold, not every triangulation admits a branching, so in general one requires so-called weak branchings~\cite{BP} to describe spin structures.
\end{remark}
Given a branching of the triangulation, one can orient the $2$-dimensional faces of $P$ so that they are positively transverse to the dual edges of the triangulation.
\begin{lemma}\label{spine}
The spine $P$ can be smoothed in such a way that it gains a well-defined normal vector field $\nu$. Let $Y$ denote the smoothed spine. \qed
\end{lemma}
This lemma tells us that there is a canonically defined nonvanishing vector field $\nu$ along $Y$. We can choose another vector field $\eta$ defined along $Y$ which is tangent to $Y$, the \textit{descending vector field} of~\cite[\S{2.1}]{BP}. These two vector fields are always linearly independent, so they specify a spin structure on $M$. 

\subsection{Construction} As a starting point for the construction of $\Lsing^{\epsilon}$, we define a coamoeba dual to $V$. Identify $T^3 = \mathbb{R}^3/\mathbb{Z}^3$ with a quotient of $[0,1]^3$.
\begin{definition}
Let $\overline{\Delta}_{+}\subset[0,1]^3$ denote the $3$-simplex with vertices at $(0,0,0)$, $(1,0,0)$, $(0,1,0)$, and $(0,0,1)$. Let $\overline{\Delta}_{-}\subset[0,1]^3$ denote the $3$-simplex with vertices at $(0,1,1)$, $(1,0,1)$, $(1,1,0)$, and $(1,1,1)$. Define $\Delta_{\pm}$ to be the union of the $1$-cells and $3$-cells in $\overline{\Delta}_{\pm}$, and define the coamoeba to be $\Delta\coloneqq\Delta_+\cup\Delta_{-}$.
\end{definition}
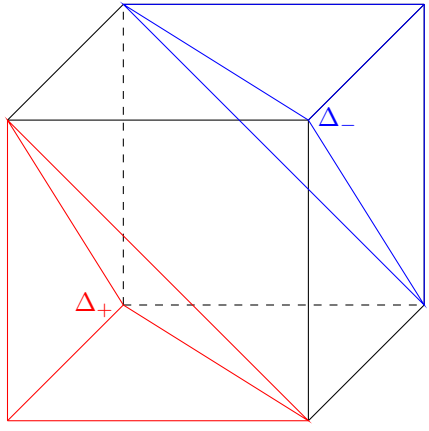
\begin{figure}
\begin{tikzpicture}[scale =  4]
 \draw[dashed] (0,0,0) -- (1,0,0);
 \draw[dashed] (0,0,0) -- (0,1,0);
 \draw (1,0,0) -- (1,1,0);
 \draw (1,0,0) -- (1,0,1);
 \draw (0,1,0) -- (1,1,0);
 \draw (0,1,0) -- (0,1,1);
 \draw (0,1,1) -- (1,1,1);
 \draw (1,0,1) -- (1,1,1);
 \draw (1,1,0) -- (1,1,1);

 \draw[red] (0,0,0) -- (1,0,1) -- (0,1,1) -- cycle;
 \draw[red] (0,0,0) -- (0,0,1);
 \draw[red] (1,0,1) -- (0,0,1);
 \draw[red] (0,1,1) -- (0,0,1);
 
 \draw[blue] (1,1,1) -- (1,0,0) -- (0,1,0) -- cycle;
 \draw[blue] (1,1,1) -- (1,1,0);
 \draw[blue] (1,0,0) -- (1,1,0);
 \draw[blue] (0,1,0) -- (1,1,0);

 \node[anchor = west] at (1,1,1) {\color{blue} $\Delta_{-}$};
 \node[anchor = east] at (0,0,0) {\color{red} $\Delta_{+}$};
\end{tikzpicture}
\caption{The coamoeba $\Delta\subset[0,1]^3$ associated to the $4$-valent tropical vertex.}\label{Coamoeba}
\end{figure}
The images of these simplices in $T^3$ intersect each other along their edges, which are either of the form $E_{k}\coloneqq\lbrace\theta_i = \theta_j = 0\rbrace\subset[0,1]^3$, where $\lbrace i,j,k\rbrace = \lbrace 1,2,3\rbrace$ or $E_{ij}\coloneqq\lbrace \theta_i+\theta_j = 0, \text{ and } \theta_k = 0\rbrace\subset[0,1]^3$, where $\lbrace i,j,k\rbrace = \lbrace 1,2,3\rbrace$. All vertices of $\overline{\Delta}_{+}$ and $\overline{\Delta}_{-}$ project to the same point in $T^3$. Define $E_{k}^{\circ}$ and $E_{ij}^{\circ}$ to be the $1$-cells of $\Delta$ (which do not include the vertices).

It will be helpful to understand the symmetries of $\Delta$ for our construction. Let $\theta_0,\ldots,\theta_3$ denote the vertices of $\Delta_+\subset[0,1]^3$ as in Figure~\ref{triangulation}. For $k = 1,2,3$, there is a unique affine automorphism of $\mathbb{R}^3$ which fixes $\Delta_+$, switches $\theta_0$ with $\theta_k$, and leaves the other vertices of $\Delta_+$ fixed. This descends to an affine automorphism of $T^3$ which we call $R_k$. Let $G$ be the group generated by $\lbrace R_1,R_2,R_3\rbrace$. Note that these automorphisms also fix $\Delta_{-}$, so $G$ acts on $\Delta$. The obvious action of $G$ on $\lbrace0,1,2,3\rbrace$ corresponds to its action on the edges and faces of $\Delta$.

As in the construction of the Lagrangian pair of pants, we need to find an appropriate smoothing of $\Delta$. Choose neighborhoods $U_k$ of $E_k^{\circ}$ and $U_{ij}$ of $E_{ij}^{\circ}$ in $\Delta$ which are diffeomorphic to subsets of $\mathbb{R}^2$ crossed with open intervals. We can form the real blowups $\widetilde{U}_k$ and $\widetilde{U}_{ij}$ by blowing up along the interiors of the edges $E_k^{\circ}$ and $E_{ij}^{\circ}$, respectively. 

To explain this more precisely, fix a neighborhood $V$ of a vertex in the $2$-dimensional coamoeba for a pair of pants. We can identify the neighborhoods $U_{ij}$ and $U_k$  with the product $(0,1)\times V$. In these coordinates, $\widetilde{U}_{ij}$ and $\widetilde{U}_k$ are identified with a product of the real blowup of $V$ at the origin with an open interval.

The real blowups along the edges come with natural maps $\widetilde{U}_k\to U_k$ and $\widetilde{U}_{ij}\to U_{ij}$.
\begin{definition}
The real blowup $\widetilde{\Delta}$ is the space obtained from $\mathrm{Int}(\Delta)$ by gluing $\widetilde{U}_k$ and $\widetilde{U}_{ij}$ to $\mathrm{Int}(\Delta)$ using the projection maps. Let $\pi\colon\widetilde{\Delta}\to\Delta$ denote the blowup map.
\end{definition}
We emphasize that $\Delta$ does not include the vertices of the two tetrahedra, and in the construction of the real blowup, we have not glued anything in at the vertices. This means that $\widetilde{\Delta}$ is smooth.

\begin{lemma}\label{topology}
The real blowup $\widetilde{\Delta}$ is a smooth $3$-manifold diffeomorphic to $L'$.
\end{lemma}
\begin{proof} Identify $\Int\Delta_{+}$ and $\Int\Delta_{-}$ with the interiors of solid cubes in such a way that the edges of $\Delta_{+}$ and $\Delta_{-}$ correspond to faces of the cube. Consequently, vertices of $\Delta_{\pm}$ will correspond to the vertices of the cubes labeled $0$, and the faces of $\Delta_{\pm}$ will correspond to the vertices of the cubes labeled $1$, $2$, $3$, and $4$. Under this identification, the real blowup will correspond to gluing the two cubes together by reflections along the black lines in Figure~\ref{triangulation}.
\end{proof}
Define the function $g\colon\Delta\to\mathbb{R}$ by
\begin{align}\label{coamoebaprimitive}
g(\theta_1,\theta_2,\theta_3) = \begin{cases}
-\sqrt{\sin(\pi\theta_1+\pi\theta_2+\pi\theta_3)\sin(\pi\theta_1)\sin(\pi\theta_2)\sin(\pi\theta_3)}, \; \theta\in\Delta_{+} \\
\sqrt{\sin(\pi\theta_1+\pi\theta_2+\pi\theta_3)\sin(\pi\theta_1)\sin(\pi\theta_2)\sin(\pi\theta_3)}, \; \theta\in\Delta_{-}.
\end{cases}
\end{align}
Observe that $g$ vanishes on the boundary of $\overline{\Delta}$, and that it is invariant under the involution $\theta\mapsto-\theta$ on $T^3$.
\begin{lemma}
The differential $dg$ restricted to $\Int\Delta$ extends to a smooth function $\widetilde{dg}\colon\widetilde{\Delta}\to\mathbb{R}$.
\end{lemma}
\begin{proof}
One easily computes the components of the differential $dg$ on $\Delta_{\pm}$ to be
\begin{align}\label{coamoeba1form}
\frac{\partial g}{\partial\theta_1} &= \mp\frac{\pi\sin(\pi\theta_2)\sin(\pi\theta_3)\sin(2\pi\theta_1+\pi\theta_2+\pi\theta_3)}{2\sqrt{\sin(\pi\theta_1)\sin(\pi\theta_2)\sin(\pi\theta_3)\sin(\pi\theta_1+\pi\theta_2+\pi\theta_3)}} \\
\frac{\partial g}{\partial\theta_2} &= \mp\frac{\pi\sin(\pi\theta_1)\sin(\pi\theta_3)\sin(\pi\theta_1+2\pi\theta_2+\pi\theta_3)}{2\sqrt{\sin(\pi\theta_1)\sin(\pi\theta_2)\sin(\pi\theta_3)\sin(\pi\theta_1+\pi\theta_2+\pi\theta_3)}} \nonumber \\
\frac{\partial g}{\partial\theta_3} &= \mp\frac{\pi\sin(\pi\theta_1)\sin(\pi\theta_2)\sin(\pi\theta_1+\pi\theta_2+2\pi\theta_3)}{2\sqrt{\sin(\pi\theta_1)\sin(\pi\theta_2)\sin(\pi\theta_3)\sin(\pi\theta_1+\pi\theta_2+\pi\theta_3)}}. \nonumber
\end{align}
First we will show that $dg$ extends over the edge $E_3^{\circ}$. In the neighborhood $\widetilde{U_3}$, the real blowup has coordinates $(\theta_1,t,\theta_3)$, where $t\in\mathbb{R}$, and the map $\widetilde{U_3}\to U_3$ is $(\theta_1,t,\theta_3)\mapsto(t\theta_1,t,\theta_3)$. Rewriting $dg$ in these coordinates, we get
\begin{align*}
\frac{\partial g}{\partial\theta_1} &= \mp\frac{\pi\sin(\pi t)\sin(\pi\theta_3)\sin(2\pi t\theta_1+\pi t+\pi\theta_3)}{2\sqrt{\sin(\pi t\theta_1)\sin(\pi t)\sin(\pi\theta_3)\sin(\pi t\theta_1+\pi t+\pi\theta_3)}} \\
\frac{\partial g}{\partial\theta_2} &= \mp\frac{\pi\sin(\pi t\theta_1)\sin(\pi\theta_3)\sin(\pi t\theta_1+2\pi t+\pi\theta_3)}{2\sqrt{\sin(\pi t\theta_1)\sin(\pi t)\sin(\pi\theta_3)\sin(\pi t\theta_1+\pi t+\pi\theta_3)}} \\
\frac{\partial g}{\partial\theta_3} &= \mp\frac{\pi\sin(\pi t\theta_1)\sin(\pi t)\sin(\pi t\theta_1+\pi t+2\pi\theta_3)}{2\sqrt{\sin(\pi t\theta_1)\sin(\pi t)\sin(\pi\theta_3)\sin(\pi t\theta_1+\pi t+\pi\theta_3)}}.
\end{align*}
Notice that nonzero values of $t$ correspond to points $(t\theta_1,t,\theta_3)$ in the interior of $\Delta$, so checking whether these functions extend smoothly across the blowup amounts to checking whether or not they extend smoothly to $t=0$. To check whether $\frac{\partial g}{\partial\theta_1}$ extends to the blowup, we rewrite it as
\begin{align*}
\frac{\partial g}{\partial\theta_1} &= \mp\frac{\pi\sin(\pi t)\sin(\pi\theta_3)\sin(2\pi t\theta_1+\pi t+\pi\theta_3)}{2\sqrt{\sin(\pi t\theta_1)\sin(\pi t)\sin(\pi\theta_3)\sin(\pi t\theta_1+\pi t+\pi\theta_3)}} \\
&= \mp\epsilon\frac{\pi\frac{\sin(\pi t)}{t}\sin(\pi\theta_3)\sin(2\pi t\theta_1+\pi t+\pi\theta_3)}{2\sqrt{\frac{\sin(\pi t\theta_1)}{t}\frac{\sin(\pi t)}{t}\sin(\pi\theta_3)\sin(\pi t\theta_1+\pi t+\pi\theta_3)}},
\end{align*}
where $\epsilon$ is the sign $\frac{t}{\sqrt{t^2}}$, which is $(-1)$ when $t<0$ and $1$ when $t>0$. Because the value of $\theta_3$ will always be nonzero, since we are considering the blowup along $E_3^{\circ}$. Since $\theta_1$ is close to $0$ or $1$, depending on whether the point lies in $\Delta_+$ or $\Delta_{-}$, all factors in these expressions involving $\theta_3$ are never $0$. Thus it is clear that $\frac{\partial g}{\partial\theta_1}$ extends smoothly across $t=0$, with the extension given by
\begin{align*}
\widetilde{dg}_1 = -\frac{\pi\frac{\sin(\pi t)}{t}\sin(\pi\theta_3)\sin(2\pi t\theta_1+\pi t+\pi\theta_3)}{2\sqrt{\frac{\sin(\pi t\theta_1)}{t}\frac{\sin(\pi t)}{t}\sin(\pi\theta_3)\sin(\pi t\theta_1+\pi t+\pi\theta_3)}}.
\end{align*}

The sign $-1$ comes from comparing the sign in the definition of $g$ and the sign $\frac{t}{\sqrt{t^2}}$. A similar computation shows that $\frac{\partial g}{\partial\theta_2}$ extends smoothly to the blowup, and the extension has the formula
\begin{align*}
\widetilde{dg}_2 = -\frac{\pi\frac{\sin(\pi t\theta_1)}{t}\sin(\pi\theta_3)\sin(\pi t\theta_1+2\pi t+\pi\theta_3)}{2\sqrt{\frac{\sin(\pi t\theta_1)}{t}\frac{\sin(\pi t)}{t}\sin(\pi\theta_3)\sin(\pi t\theta_1+\pi t+\pi\theta_3)}}.
\end{align*}

Treating $\frac{\partial g}{\partial\theta_3}$ in the same way, we find that the third component of the extension is
\begin{align*}
\widetilde{dg}_3 = -t\frac{\pi\frac{\sin(\pi t\theta_1)}{t}\frac{\sin(\pi t)}{t}\sin(\pi t\theta_1+\pi t+2\pi\theta_3)}{2\sqrt{\frac{\sin(\pi t\theta_1)}{t}\frac{\sin(\pi t)}{t}\sin(\pi\theta_3)\sin(\pi t\theta_1+\pi t+\pi\theta_3)}}.
\end{align*}
We note that this vanishes when $t=0$, so points in $E_3$ are mapped to the hyperplane $\lbrace q_3 = 0\rbrace$ in $\mathbb{R}^3$. Since $g$ is symmetric under the group $G$ of symmetries of $\Delta$, it follows that $dg$ extends over the blowups along all other edges of $\Delta$.
\end{proof}
Given a finite subset $E  =\lbrace e_1,\ldots,e_k\rbrace\subset\mathbb{R}^3$, let
\[\mathrm{Cone}(E)\coloneqq\left\lbrace\sum_{j=1}^k t_j e_j\mid t_j\in\mathbb{R}_{\geq0}\right\rbrace\]
denote the cone over $E$. The following is a consequence of the formula for $\widetilde{dg}$.

\begin{lemma}\label{coneimages}
Let $\lbrace i,j,k\rbrace = \lbrace 1,2,3\rbrace$. The map $\widetilde{dg}\colon\widetilde{\Delta}\to\mathbb{R}^3$ maps points lying over the edge $E_i$ in $\Delta$ to points in $\mathrm{Cone}\lbrace u_j,u_k\rbrace$. It maps points lying over the edge $E_{jk}$ to points in $\mathrm{Cone}\lbrace u_0,u_i\rbrace$. \qed
\end{lemma}

We will use this to study the front projection of the Legendrian link of the cone point.
\begin{definition}\label{closuredef}
Define $\Lsing$ to be the closure of the embedding
\begin{align*}
\Phi\colon\widetilde{\Delta}&\to T^*T^3 = T^3\times\mathbb{R}^3 \\
\theta &\mapsto(\pi(\theta),\widetilde{dg}(\theta))
\end{align*}
\end{definition}
Since $\lim_{\theta\to0} dg(\theta) = 0$, this is just the union $\Lsing = \Phi(\widetilde{\Delta})\cup\lbrace0\rbrace\subset T^*T^3$. Clearly $\Lsing$ is smooth everywhere except at $0\in T^*T^3$, so for the rest of this subsection, we will study the behavior of $\Lsing$ near this point. By the construction of the diffeomorphism between $\widetilde{\Delta}$ and the minimally twisted five component chain link complement, we know that the intersection of $\Lsing$ with a small sphere centered $0\in T^*T^3$ is a (Legendrian) $2$-torus. We now prove that this link is the Legendrian link of the Harvey--Lawson cone.
\begin{definition}
The Harvey--Lawson cone $C_{HL}$ is defined~\cite{HL} to be the subset
\[ C_{HL} \coloneqq \lbrace(z_1,z_2,z_3)\in\mathbb{C}^3\mid |z_1| = |z_2| = |z_3|, \; z_1z_2z_3\in\mathbb{R}_{\geq0}\rbrace. \]
Let $\Lambda_{HL} \coloneqq C_{HL}\cap S^5(\epsilon)$ denote the link of $C_{HL}$.
\end{definition}
The link $\Lambda_{HL}$ is the image of a map $S^1\times S^1 \to S^5(\epsilon)\subset\mathbb{C}^3$ defined to be
\begin{equation}\label{parametrization}
(s,t) \mapsto\left(\sqrt{\frac{\epsilon}{3}}e^{is},\sqrt{\frac{\epsilon}{3}}e^{it},\sqrt{\frac{\epsilon}{3}}e^{-is-it}\right).
\end{equation}
We write the coordinates on $\mathbb{C}^3$ in terms of real and imaginary parts as $z = x+iy$, where $z = (z_1,z_2,z_3)$. Then if we define $n\coloneqq\frac{1}{|x|} x\in S^2$ and $f\coloneqq n\cdot y$, the front projection of $\Lambda_{HL}$ is
\begin{equation}\label{frontformula}
(s,t)\mapsto e^f n\in\mathbb{R}^3\setminus\lbrace 0\rbrace\cong S^2\times\mathbb{R}.
\end{equation}
The caustic set of this projection is the edge graph of a tetrahedron embedded on $S^2$, where the vertices are at the points $\left(\sqrt{\frac{\epsilon}{3}},\sqrt{\frac{\epsilon}{3}},\sqrt{\frac{\epsilon}{3}}\right)$, $\left(-\sqrt{\frac{\epsilon}{3}},\sqrt{\frac{\epsilon}{3}},\sqrt{\frac{\epsilon}{3}}\right)$, $\left(-\sqrt{\frac{\epsilon}{3}},-\sqrt{\frac{\epsilon}{3}},\sqrt{\frac{\epsilon}{3}}\right)$, and $\left(-\sqrt{\frac{\epsilon}{3}},-\sqrt{\frac{\epsilon}{3}},-\sqrt{\frac{\epsilon}{3}}\right)$.

It will be helpful to think of $T^*T^3$ as $(\mathbb{C}^*)^3$ using the identification $(q_i,\theta_i) = (\log(r_i),\theta_i)$, where $q_i$ is the coordinate on the cotangent fiber of $T^*T^3$. Hence, if $U$ is a small neighborhood of $0\in T^*T^3$, it is identified with a small neighborhood of $0\in\mathbb{C}^3$ under the exponential map $\mathbb{C}^3\to(\mathbb{C}^*)^3$. Under these identifications, the projection $U\to\mathbb{R}^3$ to the cotangent fiber can be thought of as a restriction of the projection $\mathbb{C}^3\to\mathbb{R}^3$ to the real axis.
\begin{lemma}\label{legiso}
For a sufficiently small sphere $S^5(\epsilon)\subset\mathbb{C}^3$, the intersection $\Lambda_{\epsilon}\coloneqq S^5(\epsilon)\cap \Lsing$ is Legendrian isotopic to $\Lambda_{HL}$.
\end{lemma}
\begin{figure}
\begin{tikzpicture}
\begin{scope}
  \draw[thick, dotted] (0,0,0) -- (-4,0,0) node[anchor = north west]{$(-1,0,0)$};
  \draw[thick, dotted] (0,0,0) -- (0,-4,0) node[anchor = south west]{$(0,0,-1)$};
  \draw[thick, dotted] (0,0,0) -- (0,0,-7) node[anchor = north west]{$(1,1,1)$};
  \draw[thick, dotted] (0,0,0) -- (-5,-3,-3) node[anchor = north west]{$(0,-1,0)$};

  \draw[fill = gray, fill opacity = 0.6] (-2,0,0) -- (0,-2,0) -- (0,0,-3.5) -- (-2,0,0);
  \draw[fill = gray, fill opacity = 0.4] (-2,0,0) -- (-2.5,-1.5,-1.5) -- (0,-2,0) -- (-2,0,0);
  \draw[fill = gray, fill opacity = 0.1] (-2.5,-1.5,-1.5) -- (0,0,-3.5) -- (0,-2,0) -- (-2.5,-1.5,-1.5);
\end{scope}
\end{tikzpicture}

\caption{The image in $Q$ of a link of the cone point of $\Lsing$. The dotted lines are the edges tropical curve $V$.}\label{front}
\end{figure}
\begin{proof} We will begin by describing the image of $\Lambda_{\epsilon}$ in $T^*T^3$ under $\widetilde{dg}$, which, by Definition~\ref{closuredef}, is the composition of $\Phi\colon\widetilde{\Delta}\to T^*T^3$ with projection onto the cotangent fiber. Since the cone point of $\Lsing$ is mapped to $0\in\mathbb{R}^3$ under this map, it follows that a sufficiently small neighborhood of the cone point will be mapped to an open ball in $\mathbb{R}^3$ centered at the origin. Without loss of generality, we can assume that $\Lambda_{\epsilon}$ is mapped to the boundary of this open ball, which we denote by $S^2(\epsilon)$. We can think of this ball as a smoothing of the tetrahedron in Figure with vertices on the legs of $V$, shown in~\ref{front}. In particular, there is a tetrahedral graph embedded in $S^2(\epsilon)$. The vertices lying on the tropical curve $V$, and the edges lie on the cones spanned by pairs of edges of $V$, as described in Lemma~\ref{coneimages}.

We claim that $\widetilde{dg}$ is $2$-to-$1$ over all points of $S^2(\epsilon)$ not contained in the tetrahedral graph. Indeed, Lemma~\ref{coneimages} implies that the preimage of any such point must lie in (the preimage of) $\Int\Delta$. Since $\widetilde{dg}$ is invariant under the involution on $T^3$, it follows that there are two preimages of any such point, which lie in the same orbit under the involution. Inspecting the formula for $\widetilde{dg}$ shows that it is $1$-to-$1$ over the edges of the graph.

Now observe that the matrix
\begin{align}\label{matrixbase}
\frac{1}{4}\begin{pmatrix} -1 & 1 & 1 \\ 1 & -1 & 1 \\ 1 & 1 & -1 \end{pmatrix}
\end{align}
maps the caustic set of~\eqref{frontformula} to the edge graph on $S^2(\epsilon)$, as depicted in Figure~\ref{front}. Comparing with the formula in~\eqref{frontformula}, and using the identification of a neighborhood of $0$ in $\mathbb{C}^3$ with a neighborhood of $0\in T^*T^3$, shows that the front projection of $\Lambda_{\epsilon}$ has caustic set given by the tetrahedral graph. Said differently, $\Lambda_{\epsilon}$ is the Legendrian surface constructed from the weave associated to this tetrahedral graph, as in~\cite{CZ}. It follows that $\Lambda_{HL}$ and $\Lambda_{\epsilon}$ are Legendrian isotopic. Concretely, this isotopy is obtained by lifting a path of matrices in $GL(3,\mathbb{R})$ joining the identity to the matrix of~\eqref{matrixbase} to a path of symplectic automorphisms of $T^*\mathbb{R}^3$.
\end{proof}
This implies the following description of $L$ near the singular point.
\begin{corollary}\label{conenbhd}
There is a small ball $B\subset\mathbb{C}^3$ centered at the origin such that $\Lsing\cap(B\setminus\lbrace0\rbrace)$ is diffeomorphic to $T^2\times(0,\epsilon)$. There is a Hamiltonian isotopy of $B$ taking the Lagrangian cone $L\cap B$ to the Harvey--Lawson cone.
\end{corollary}
Extending this Hamiltonian isotopy by the identity outside of a slightly larger neighborhood of the singular point lets us identify the cone point of $\Lsing$ with the Harvey--Lawson cone.

We can similarly deform $\Lsing$ over the legs of $V$. Since $\Lsing$ is exact in the appropriate sense, it can be put in tropical position by a Hamiltonian isotopy.
\begin{lemma} Let $\pi\colon T^*T^3\to\mathbb{R}^3$ denote the projection. For any $\epsilon>0$, there is a Hamiltonian isotopy which takes $\Lsing$ to a (singular) Lagrangian submanifold $\Lsing^{\epsilon}$ for which $\pi_{SYZ}(\Lsing^{\epsilon})$ is within Hausdorff distance $\epsilon$ of $V$. Moreover, $\Lsing^{\epsilon}$ agrees with the conormal lifts of the $1$-dimensional faces $V_i$ of $V$ outside of an arbitrarily small ball $B$ centered at the origin in $\mathbb{R}^3$, i.e.
\begin{align}\label{conormalregion}
\Lsing^{\epsilon}\mid_{T^*T^3\setminus\pi^{-1}_{SYZ}(B)} = \bigcup_{i=0}^3 L_{V_i}\mid_{T^*T^3\setminus\pi_{SYZ}^{-1}(B)} \, .
\end{align}
\qed
\end{lemma}
The argument in the proof of~\cite{HicksDimer} applies in our situation with no changes, since $L\mid_{T^*T^3\setminus\pi^{-1}(B_0(\epsilon))}$ is smooth. Therefore we can construct a Lagrangian lift which agrees with the periodized conormals to the legs of $V$ outside of a ball in $\mathbb{R}^3$ centered at the origin. The statements about Hausdorff distance and the size of $B$ can be proven by rescaling $\widetilde{dg}$ by an arbitrarily small real constant. From now on, we will use $\Lsing$ to refer interchangeably to the closure of $\Phi(M)$ constructed using $\lambda\widetilde{dg}$ for any $\lambda>0$, or to any of the Lagrangian submanifolds $\Lsing^{\epsilon}$ in tropical position.

\subsection{First homology}\label{firsthomology} In this section, we will discuss $H_1(L')$ and describe the induced map $H_1(L')\to H_1(T^*T^3) = H_1(T^3)$ in preparation for the proof of Theorem~\ref{main}. Since $L'$ is a link complement, we know that $H_1(L')$ is generated by the meridians $m_0,\ldots,m_4$ to the components of the minimally twisted five-component chain link. It will also be helpful for us to consider the longitudes of the link components, denoted by $\ell_0,\ldots,\ell_4$. We choose orientations on the $m_i$'s and $\ell_i$'s which are shown in Figure~\ref{loops}. Examining this figure has the following immediate consequence.
\begin{figure}
\tikzset{
 redarrows/.style={postaction={decorate},decoration={markings,mark=at position 0 with {\arrow[draw=red]{<}}},
           }}
 \tikzset{
  redarrows2/.style={postaction={decorate},decoration={markings,mark=at position -0.4 with {\arrow[draw=red]{<}}},
  			}}
\tikzset{
 bluearrows/.style={postaction={decorate},
 decoration={markings,mark=at position 0.2 with {\arrow[draw=blue]{>}}},
           }}
\begin{tikzpicture}[scale = 1.5]
\draw[gray, ultra thick] (0,0) circle (2.3);
\draw(2,0)[line width = 0.25 mm] ellipse (0.5 and 1.5);

\draw(2,0)[red, line width = 0.25 mm, redarrows2] ellipse (0.7 and 1.7);
\draw[rotate = 72, line width = 0.25 mm] (2,0) ellipse (0.5 and 1.5);
\draw(2,0)[rotate = 72, red, line width = 0.25 mm, redarrows] ellipse (0.7 and 1.7);

\draw[rotate = 144, line width = 0.25 mm] (2,0) ellipse (0.5 and 1.5);
\draw(2,0)[rotate = 144, red, line width = 0.25 mm, redarrows] ellipse (0.7 and 1.7);

\draw[rotate = 216, line width = 0.25 mm] (2,0) ellipse (0.5 and 1.5);
\draw(2,0)[rotate = 216, red, line width = 0.25 mm, redarrows2] ellipse (0.7 and 1.7);

\draw[rotate = 288, line width = 0.25 mm] (2,0) ellipse (0.5 and 1.5);
\draw(2,0)[rotate = 288, red, line width = 0.25 mm, redarrows] ellipse (0.7 and 1.7);

\draw (2.5,0)[white, double = black, ultra thick] arc (0:100:0.5 and 1.5);
\draw (2.7,0)[white, double = red, ultra thick] arc (0:100:0.7 and 1.7);

\draw (2.5,0)[white, double= black, ultra thick] arc (0:-100:0.5 and 1.5);
\draw (2.7,0)[white, double= red, ultra thick] arc (0:-100:0.7 and 1.7);

\draw[white, double = black, ultra thick, rotate = 72] (2.5,0) arc (0:100:0.5 and 1.5);
\draw[white, double = red, ultra thick, rotate = 72] (2.7,0) arc (0:100:0.7 and 1.7);

\draw[white, double = black, ultra thick, rotate = 144] (1.5,0) arc (180:240:0.5 and 1.5);
\draw[white, double = red, ultra thick, rotate = 144] (1.3,0) arc (180:240:0.7 and 1.7);

\draw[white, double = black, ultra thick, rotate = 72] (1.5,0) arc (180:240:0.5 and 1.5);
\draw[white, double = red, ultra thick, rotate = 72] (1.3,0) arc (180:240:0.7 and 1.7);

\draw[white, double = black, ultra thick, rotate = 144] (1.5,0) arc (180:120:0.5 and 1.5);
\draw[white, double = red, ultra thick, rotate = 144] (1.3,0) arc (180:120:0.7 and 1.7);

\draw[white, double = black, ultra thick, rotate = 216] (2.5,0) arc (0:-100:0.5 and 1.5);
\draw[white, double = red, ultra thick, rotate = 216] (2.7,0) arc (0:-100:0.7 and 1.7);

\draw[white, double = black, ultra thick, rotate = 216] (2.5,0) arc (0:100:0.5 and 1.5);
\draw[white, double = red, ultra thick, rotate = 216] (2.7,0) arc (0:100:0.7 and 1.7);

\draw[white, double = black, ultra thick, rotate = 288] (1.5,0) arc (180:240:0.5 and 1.5);
\draw[white, double = red, ultra thick, rotate = 288] (1.3,0) arc (180:240:0.7 and 1.7);

\draw[white, double = black, ultra thick, rotate = 288] (1.5,0) arc (180:120:0.5 and 1.5);
\draw[white, double = red, ultra thick, rotate = 288] (1.3,0) arc (180:120:0.7 and 1.7);

\draw[blue, thick, bluearrows] (1.5,0) ellipse (0.4 and 0.2);
\draw[white, double = blue, ultra thick] (1.1,0) arc (180:360:0.4 and 0.2);

\draw[white, double = black, thick] (1.5,0) arc (180:160:0.5 and 1.5);
\draw[white, double = red, thick] (1.3,0) arc (180:160:0.7 and 1.7);

\draw[blue, thick, rotate = 72, bluearrows] (1.5,0) ellipse (0.4 and 0.2);
\draw[white, double = blue, ultra thick, rotate = 72] (1.1,0) arc (180:360:0.4 and 0.2);

\draw[white, double = black, thick, rotate = 72] (1.5,0) arc (180:160:0.5 and 1.5);
\draw[white, double = red, thick, rotate = 72] (1.3,0) arc (180:160:0.7 and 1.7);

\draw[blue, thick, rotate = 144, bluearrows] (1.5,0) ellipse (0.4 and 0.2);
\draw[white, double = blue, ultra thick, rotate = 144] (1.1,0) arc (180:360:0.4 and 0.2);

\draw[white, double = black, thick, rotate = 144] (1.5,0) arc (180:160:0.5 and 1.5);
\draw[white, double = red, thick, rotate = 144] (1.3,0) arc (180:160:0.7 and 1.7);

\draw[blue, thick, rotate = 216, bluearrows] (1.5,0) ellipse (0.4 and 0.2);
\draw[white, double = blue, ultra thick, rotate = 216] (1.1,0) arc (180:360:0.4 and 0.2);

\draw[white, double = black, thick, rotate = 216] (1.5,0) arc (180:160:0.5 and 1.5);
\draw[white, double = red, thick, rotate = 216] (1.3,0) arc (180:160:0.7 and 1.7);

\draw[blue, thick, rotate = 288, bluearrows] (1.5,0) ellipse (0.4 and 0.2);
\draw[white, double = blue, ultra thick, rotate = 288] (1.1,0) arc (180:360:0.4 and 0.2);

\draw[white, double = black, thick, rotate = 288] (1.5,0) arc (180:160:0.5 and 1.5);
\draw[white, double = red, thick, rotate = 288] (1.3,0) arc (180:160:0.7 and 1.7);

\node[] at (3,0) {\color{red}$\ell_1$};
\node[] at (0.93,2.86) {\color{red}$\ell_0$};
\node[] at (-2.43,1.77) {\color{red}$\ell_4$};
\node[] at (-2.43,-1.77) {\color{red}$\ell_3$};
\node[] at (0.93,-2.86) {\color{red}$\ell_2$};

\node[] at (0.9,0) {\color{blue}$m_1$};
\node[] at (0.28,0.86) {\color{blue}$m_0$};
\node[] at (-0.73,0.53) {\color{blue}$m_4$};
\node[] at (-0.63,-0.53) {\color{blue}$m_3$};
\node[] at (0.28,-0.86) {\color{blue}$m_2$};

\end{tikzpicture}
\caption{$m_i$'s (blue) and $\ell_i$'s (red).}\label{loops}
\end{figure}
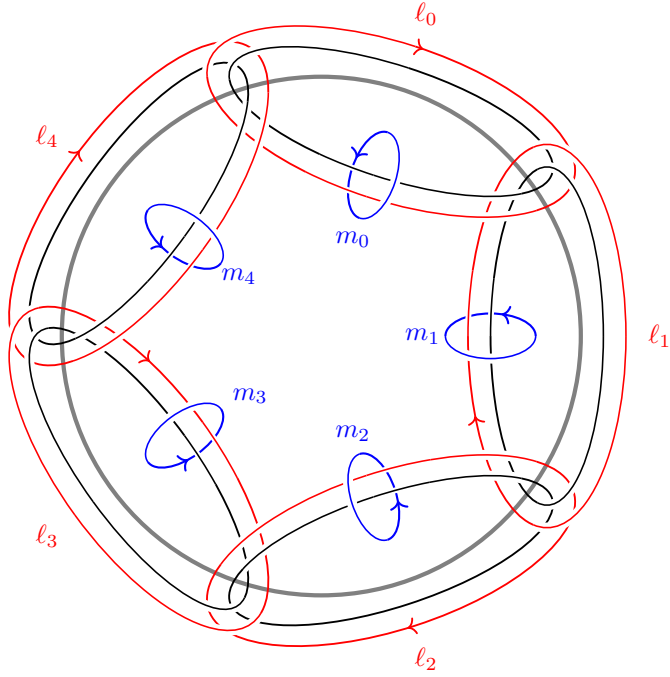
\begin{lemma}\label{relations}
The longitudes can be expressed in terms of the meridians in $H_1(L')$ by the following formulas.
\begin{align*}
\ell_0 &= -m_1-m_4, \\
\ell_1 &=-m_0+m_2, \\
\ell_2 &= m_1-m_3, \\
\ell_3 &= -m_2+m_4, \\
\ell_4 &= -m_0+m_3.
\end{align*}
\qed
\end{lemma}
Let $\lbrace e_1,e_2,e_3\rbrace\subset H_1(T^3)$ denote the set of generators which lift to the coordinate axes in the universal cover $\mathbb{R}^3\to T^3$.
\begin{lemma}\label{inducedmap}
The map $H_1(L')\to H_1(T^3)$ is determined by the following values (up to automorphisms of $H_1(T^3)$ changing the signs of $e_1$, $e_2$, and $e_3$).
\begin{align*}
m_0 &\mapsto 0, \\
m_1 &\mapsto e_2-e_3, \\
m_2 &\mapsto e_3, \\
m_3 &\mapsto -e_1+e_2, \\
m_4 &\mapsto -e_2+e_3.
\end{align*}
Consequently, the values of this map on the longitudes are as follows.
\begin{align*}
\ell_0 &\mapsto 0, \\
\ell_1 &\mapsto e_3, \\
\ell_2 &\mapsto e_1-e_3, \\
\ell_3 &\mapsto -e_2, \\
\ell_4 &\mapsto -e_1+e_2.
\end{align*}
\end{lemma}
\begin{proof}
It will be more convenient to determine the values on the longitudes first. Recall that in the ideal triangulation of $L'$ described in subsection~\ref{chainlinkdescription}, there were two ideal tetrahedra $T_{0,\pm}$ with no vertices on the zeroth cusp. Informally, we can think of these as the dual tetrahedra to the components $\Delta_{\pm}$ of the coamoeba.

For $i = 1,2,3,4$, the longitude $\ell_i$ restricts to a small arc on $T_{0,+}$ near the $i$th vertex which connects two edges adjacent to the $i$th vertex. Recall that in the ideal cubulation of $L'$ from subsection~\ref{chainlinkdescription}, each ideal cube was written as the union of five ideal tetrahedra. We can think of these five ideal tetrahedra as representatives of the ideal tetrahedra in the quotient orbifold $N$ of $L'$ by the obvious $\mathbb{Z}/2$-action on $L'$. Note that in terms of the coamoeba, this $\mathbb{Z}/2$-action comes from the involution $\theta\mapsto-\theta$ on $T^3$.

\begin{figure}

\begin{tikzpicture}[scale = 1.5]
\draw(2,0)[line width = 0.25 mm] ellipse (0.5 and 1.5);

\draw(2,0)[red, line width = 0.25 mm] ellipse (0.7 and 1.7);
\draw[rotate = 72, line width = 0.25 mm] (2,0) ellipse (0.5 and 1.5);
\draw(2,0)[rotate = 72, red, line width = 0.25 mm] ellipse (0.7 and 1.7);

\draw[rotate = 144, line width = 0.25 mm] (2,0) ellipse (0.5 and 1.5);
\draw(2,0)[rotate = 144, red, line width = 0.25 mm] ellipse (0.7 and 1.7);

\draw[rotate = 216, line width = 0.25 mm] (2,0) ellipse (0.5 and 1.5);
\draw(2,0)[rotate = 216, red, line width = 0.25 mm] ellipse (0.7 and 1.7);

\draw[rotate = 288, line width = 0.25 mm] (2,0) ellipse (0.5 and 1.5);
\draw(2,0)[rotate = 288, red, line width = 0.25 mm] ellipse (0.7 and 1.7);

\draw (2.5,0)[white, double = black, ultra thick] arc (0:100:0.5 and 1.5);
\draw (2.7,0)[white, double = red, ultra thick] arc (0:100:0.7 and 1.7);

\draw (2.5,0)[white, double= black, ultra thick] arc (0:-100:0.5 and 1.5);
\draw (2.7,0)[white, double= red, ultra thick] arc (0:-100:0.7 and 1.7);

\draw[white, double = black, ultra thick, rotate = 72] (2.5,0) arc (0:100:0.5 and 1.5);
\draw[white, double = red, ultra thick, rotate = 72] (2.7,0) arc (0:100:0.7 and 1.7);

\draw[white, double = black, ultra thick, rotate = 144] (1.5,0) arc (180:240:0.5 and 1.5);
\draw[white, double = red, ultra thick, rotate = 144] (1.3,0) arc (180:240:0.7 and 1.7);

\draw[white, double = black, ultra thick, rotate = 72] (1.5,0) arc (180:240:0.5 and 1.5);
\draw[white, double = red, ultra thick, rotate = 72] (1.3,0) arc (180:240:0.7 and 1.7);

\draw[white, double = black, ultra thick, rotate = 144] (1.5,0) arc (180:120:0.5 and 1.5);
\draw[white, double = red, ultra thick, rotate = 144] (1.3,0) arc (180:120:0.7 and 1.7);

\draw[white, double = black, ultra thick, rotate = 216] (2.5,0) arc (0:-100:0.5 and 1.5);
\draw[white, double = red, ultra thick, rotate = 216] (2.7,0) arc (0:-100:0.7 and 1.7);

\draw[white, double = black, ultra thick, rotate = 216] (2.5,0) arc (0:100:0.5 and 1.5);
\draw[white, double = red, ultra thick, rotate = 216] (2.7,0) arc (0:100:0.7 and 1.7);

\draw[white, double = black, ultra thick, rotate = 288] (1.5,0) arc (180:240:0.5 and 1.5);
\draw[white, double = red, ultra thick, rotate = 288] (1.3,0) arc (180:240:0.7 and 1.7);

\draw[white, double = black, ultra thick, rotate = 288] (1.5,0) arc (180:120:0.5 and 1.5);
\draw[white, double = red, ultra thick, rotate = 288] (1.3,0) arc (180:120:0.7 and 1.7);

\draw[blue, thick] (1.5,0) ellipse (0.4 and 0.2);
\draw[white, double = blue, ultra thick] (1.1,0) arc (180:360:0.4 and 0.2);

\draw[white, double = black, thick] (1.5,0) arc (180:160:0.5 and 1.5);
\draw[white, double = red, thick] (1.3,0) arc (180:160:0.7 and 1.7);

\draw[blue, thick, rotate = 72] (1.5,0) ellipse (0.4 and 0.2);
\draw[white, double = blue, ultra thick, rotate = 72] (1.1,0) arc (180:360:0.4 and 0.2);

\draw[white, double = black, thick, rotate = 72] (1.5,0) arc (180:160:0.5 and 1.5);
\draw[white, double = red, thick, rotate = 72] (1.3,0) arc (180:160:0.7 and 1.7);

\draw[blue, thick, rotate = 144] (1.5,0) ellipse (0.4 and 0.2);
\draw[white, double = blue, ultra thick, rotate = 144] (1.1,0) arc (180:360:0.4 and 0.2);

\draw[white, double = black, thick, rotate = 144] (1.5,0) arc (180:160:0.5 and 1.5);
\draw[white, double = red, thick, rotate = 144] (1.3,0) arc (180:160:0.7 and 1.7);

\draw[blue, thick, rotate = 216] (1.5,0) ellipse (0.4 and 0.2);
\draw[white, double = blue, ultra thick, rotate = 216] (1.1,0) arc (180:360:0.4 and 0.2);

\draw[white, double = black, thick, rotate = 216] (1.5,0) arc (180:160:0.5 and 1.5);
\draw[white, double = red, thick, rotate = 216] (1.3,0) arc (180:160:0.7 and 1.7);

\draw[blue, thick, rotate = 288] (1.5,0) ellipse (0.4 and 0.2);
\draw[white, double = blue, ultra thick, rotate = 288] (1.1,0) arc (180:360:0.4 and 0.2);

\draw[white, double = black, thick, rotate = 288] (1.5,0) arc (180:160:0.5 and 1.5);
\draw[white, double = red, thick, rotate = 288] (1.3,0) arc (180:160:0.7 and 1.7);

\node[] at (3,0) {\color{red}$\ell_1$};
\node[] at (0.93,2.86) {\color{red}$\ell_0$};
\node[] at (-2.43,1.77) {\color{red}$\ell_4$};
\node[] at (-2.43,-1.77) {\color{red}$\ell_3$};
\node[] at (0.93,-2.86) {\color{red}$\ell_2$};

\node[] at (0.9,0) {\color{blue}$m_1$};
\node[] at (0.28,0.86) {\color{blue}$m_0$};
\node[] at (-0.73,0.53) {\color{blue}$m_4$};
\node[] at (-0.63,-0.53) {\color{blue}$m_3$};
\node[] at (0.28,-0.86) {\color{blue}$m_2$};

\filldraw[white] (0,0) circle (2.3);

\draw[gray, ultra thick] (0,0) circle (2.3);

\end{tikzpicture}
\caption{The images of the longitudes in the quotient oribfold.}\label{loopsquotient}
\end{figure}
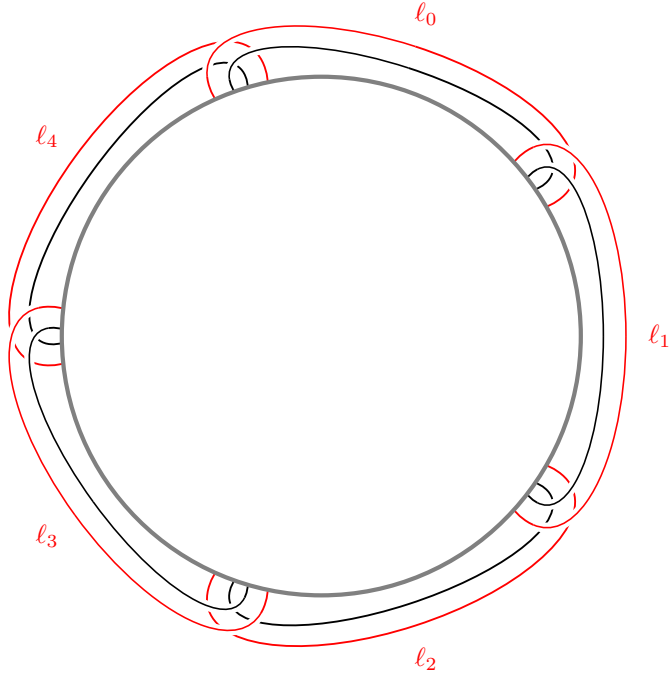

From Figure~\ref{loopsquotient}, which shows the images of the $\ell_i$'s under the orbifold quotient, we can determine the arcs on $T_{0,+}$ to which each longitude restricts. These are also drawn on the ideal cube in Figure~\ref{loopscube}. Thus the values of the induced map on homology of $\ell_i$ for $i = 1,2,3,4$ are as given in the statement of the lemma, assuming that the arcs are oriented as in Figure~\ref{loops}. We also have $\ell_0\mapsto0$ because the zeroth cusp of $\widetilde{L}$ is the link of the cone point in $\Lsing$. The values of the map on the meridians follow from Lemma~\ref{relations}.

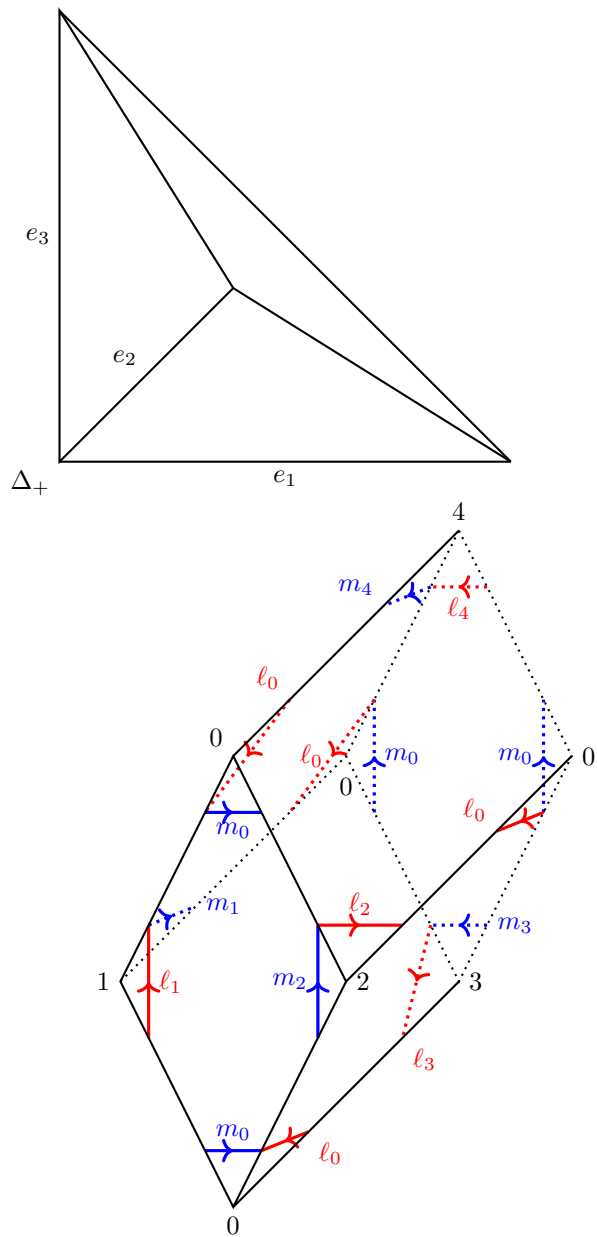
\begin{figure}

\begin{tikzpicture}
\begin{scope}[scale = 1.5]
\draw[thick] (0,0,0) -- (0,0,4) -- (0,4,4) -- (0,0,0);
\draw[thick] (0,0,0) -- (4,0,4);
\draw[thick] (0,0,4) -- (4,0,4);
\draw[thick] (0,4,4) -- (4,0,4);

\node[anchor = north east] at (0,0,4) {$\Delta_+$};

\node[anchor = east] at (0,2,4) {$e_3$};
\node[anchor = south east] at (0,0,2) {$e_2$};
\node[anchor = north] at (2,0,4) {$e_1$};
\end{scope}

\begin{scope}[scale = 1.5, yshift = -175, very thick,decoration={
    markings,
    mark=at position 0.5 with {\arrow{>}}}
    ]
    
\node[anchor = south east] at (0,2) {$0$};
\node[anchor = north] at (0,-2) {$0$};
\node[anchor = west] at (3,2) {$0$};
\node[anchor = north] at (1,1.9) {$0$};

\node[anchor = west] at (1,0) {$2$};
\node[anchor = east] at (-1,0) {$1$};
\node[anchor = south] at (2,4) {$4$};
\node[anchor = west] at (2,0) {$3$};

\draw[blue, postaction = {decorate}] (-1/4,3/2) -- (1/4,3/2);
\draw[blue, postaction = {decorate}] (-1/4,-3/2) -- (1/4,-3/2);
\draw[blue, postaction = {decorate}, dotted] (5/4,3/2) -- (5/4,5/2);
\draw[blue, postaction = {decorate}, dotted] (11/4,3/2) -- (11/4,5/2);

\node[anchor = north] at (0,3/2) {\color{blue} $m_0$};
\node[anchor = south] at (0,-3/2) {\color{blue} $m_0$};
\node[anchor = west] at (5/4,2) {\color{blue} $m_0$};
\node[anchor = east] at (11/4,2) {\color{blue} $m_0$};

\draw[blue, postaction = {decorate}, dotted] (-3/4,1/2) -- (-1/3,2/3); 
\draw[blue, postaction = {decorate}] (3/4,-1/2) -- (3/4,1/2); 
\draw[blue, postaction = {decorate}, dotted] (9/4,1/2) -- (7/4,1/2); 
\draw[blue, postaction = {decorate}, dotted] (7/4,7/2) -- (4/3,10/3); 

\node[anchor = west] at (-1/3,2/3) {\color{blue}$m_1$};
\node[anchor = east] at (3/4,0) {\color{blue}$m_2$};
\node[anchor = west] at (9/4,1/2) {\color{blue}$m_3$};
\node[anchor = south east] at (4/3,10/3) {\color{blue}$m_4$};

\draw[red, postaction = {decorate}, dotted] (1/2,5/2) -- (-1/4,3/2);
\draw[red, postaction = {decorate}] (2/3,-4/3) -- (1/4,-3/2);
\draw[red, postaction = {decorate}] (11/4,3/2) -- (7/3,4/3);
\draw[red, postaction = {decorate}, dotted]  (5/4,5/2) -- (1/2,3/2);

\node[anchor = south east] at (1/2,5/2) {\color{red} $\ell_0$};
\node[anchor = north west] at (2/3,-4/3) {\color{red} $\ell_0$};
\node[anchor = south east] at (7/3,4/3) {\color{red}$\ell_0$};
\node[anchor = west] at (1/2,2) {\color{red}$\ell_0$};

\draw[red, postaction = {decorate}] (-3/4,-1/2) -- (-3/4,1/2); 
\draw[red, postaction = {decorate}] (3/4,1/2) -- (3/2,1/2); 
\draw[red, postaction = {decorate}, dotted] (7/4,1/2) -- (3/2,-1/2); 
\draw[red, postaction = {decorate}, dotted] (9/4,7/2) -- (7/4,7/2); 

\node[anchor = west] at (-3/4,0) {\color{red}$\ell_1$};
\node[anchor = south] at (9/8,1/2) {\color{red}$\ell_2$};
\node[anchor = north west] at (3/2,-1/2) {\color{red}$\ell_3$};
\node[anchor = north] at (2,7/2) {\color{red}$\ell_4$};

\draw[thick] (0,2) -- (1,0) -- (0,-2) -- (-1,0) -- cycle;

\draw[thick, dotted] (2,4) -- (3,2) -- (2,0) -- (1,2) -- cycle;

\draw[thick] (0,2) -- (2,4);
\draw[thick] (1,0) -- (3,2);
\draw[thick] (0,-2) -- (2,0);
\draw[thick, dotted] (-1,0) -- (1,2);

\end{scope}
\end{tikzpicture}
\caption{$\Delta_+$ (top), and the corresponding ideal cube with the restrictions of $m_i$ and $\ell_i$ with orientations (bottom).}
\label{loopscube}
\end{figure}
\end{proof}
Although $m_0$ and $\ell_0$ are mapped to nullhomotopic curves in $T^*T^3$, we can still characterize their images inside  neighborhood of the cone point of $\Lsing$. Corollary~\ref{conenbhd} gives a neighborhood $B_0$ of the cone point of $\Lsing$ and a Hamiltonian isotopy carrying the link of the cone point to the link $\Lambda_{HL}$ to the link of the Harvey--Lawson cone. We will determine the images, in $\Lambda_{HL}\subset\mathbb{C}^3$, of the curves $m_0$ and $\ell_0$ under this isotopy. Consider the curves $\gamma_x(s)$ and $\gamma_z(t)$ on $\Lambda_{HL}$, which are homologous to the following two curves in $\mathbb{C}^3$:
\begin{align}\label{linkcircles}
\gamma_x(s)\coloneqq \left(\frac{1}{\sqrt{3}}\epsilon e^{is},\frac{1}{\sqrt{3}}\epsilon e^{-is},\frac{1}{\sqrt{3}}\epsilon \right) \, ,\qquad & s\in[0,2\pi]\, \\
\gamma_z(t)\coloneqq \left(\frac{1}{\sqrt{3}}\epsilon,\frac{1}{\sqrt{3}}\epsilon e^{-it},\frac{1}{\sqrt{3}}\epsilon e^{it} \right) \, , \qquad & t\in[0,2\pi] \, . \nonumber
\end{align}
\begin{remark}\label{startpoint}
Both of the curves above start at the point $\left(\frac{1}{\sqrt{3}}\epsilon,\frac{1}{\sqrt{3}}\epsilon,\frac{1}{\sqrt{3}}\epsilon\right)$ when $s = t  = 0$, but we can assume, up to a translation of $\Lambda_{HL}$, that $\gamma_x(s)$ and $\gamma_z(t)$ start at the point $\left(-\frac{1}{\sqrt{3}}\epsilon,\frac{1}{\sqrt{3}}\epsilon,-\frac{1}{\sqrt{3}}\epsilon\right)$, instead. With this convention, $\gamma_x(s)$ and $\gamma_z(t)$ project to arcs in $\mathbb{R}^3$ starting at this point and terminating on $\left(\frac{1}{\sqrt{3}}\epsilon,-\frac{1}{\sqrt{3}}\epsilon,-\frac{1}{\sqrt{3}}\epsilon\right)$ and $\left(-\frac{1}{\sqrt{3}}\epsilon,-\frac{1}{\sqrt{3}}\epsilon,\frac{1}{\sqrt{3}}\epsilon\right)$, respectively.
\end{remark} 
Let $\phi$ denote the Hamiltonian symplectomorphism of Corollary~\ref{conenbhd}. More precisely, $\phi$ acts via multiplication by~\eqref{matrixbase} on the base of $T^*\mathbb{R}^3$ and via multiplication by the matrix
\begin{align}\label{matrixfiber}
\frac{1}{2}\begin{pmatrix} 0 & 1 & 1 \\ 1 & 0 & 1 \\ 1 & 1 & 0\end{pmatrix} \, , 
\end{align}
which is the inverse transpose of~\eqref{matrixbase}, on the fibers of $T^*\mathbb{R}^3$. Recall that we have identified $T^*\mathbb{R}^3\cong\mathbb{C}^3$ by identifying the base $\mathbb{R}^3$ with the real axis in $\mathbb{C}^3$ and the cotangent fiber over $0$ with the imaginary axis. Now observe that the vectors
\begin{align*}
\begin{pmatrix} 1 \\ -1 \\ 0 \end{pmatrix} \text{ and }
\begin{pmatrix} 0 \\ -1 \\ 1 \end{pmatrix}
\end{align*}
are both eigenvectors for~\eqref{matrixfiber} with eigenvalue $-\frac{1}{2}$. This implies that the inverse of the Hamiltonian symplectomorphism of $T^*\mathbb{R}^3$ induced by~\eqref{matrixbase} scales the tangent vectors to the circles~\eqref{linkcircles} at $s = t = 0$ by a negative real number. This is because the tangent vectors to these circles at this point both point in the fiber direction in $T^*\mathbb{R}^3$, i.e. the imaginary direction of $\mathbb{C}^3$. Thus the circles $\phi^{-1}\circ\gamma_x(s)$ and $\phi^{-1}\circ\gamma_z(t)$ have the same center as $\gamma_x(s)$ and $\gamma_z(t)$, but different radii and reversed orientation. Notice in particular that the starting points of $\gamma_x(s)$ and $\gamma_z(t)$, where $s = t = 0$, are not preserved. The starting point of $\gamma_x(s)$ lies on $\mathbb{R}_{>0}\cdot(-1,0,0)$, and the starting point of $\gamma_z(t)$ lies on $\mathbb{R}_{>0}\cdot(0,0,-1)$ (neither of which corresponds to the original starting point, which would be mapped to a point on $\mathbb{R}_{>0}\cdot(0,-1,0)$, where these curves will intersect when $s = t = \pi$).

The tangent vectors to $\phi^{-1}\circ\gamma_x(s)$ and $\phi^{-1}\circ\gamma_z(t)$ at the point $s = t  = 0$ point in the fiber direction in $T^*\mathbb{R}^3$, and are given by the vectors $-e_1+e_2$ and $e_2-e_3$, respectively, where we have identified these classes in $H_1(T^3)$ with the tangent vectors to their geodesic representatives.

The homology classes in $H_1(\Lambda)$ represented by these circles can now be determined from Lemma~\ref{inducedmap} as follows. 
\begin{lemma}\label{linkh1gens}
The circles $\phi^{-1}\circ\gamma_x(s)$ and $\phi^{-1}\circ\gamma_z(t)$ are represented in $H_1(\Lambda)$ by $m_0$ and $-\ell_0$, respectively.
\end{lemma}
\begin{proof}
The starting point $\phi^{-1}\circ\gamma_x(0)$ lies on the the leg of $V$ given by the ray $\mathbb{R}_{>0}(-1,0,0)$, and the tangent vector to the curve at this point is given by $-e_1+e_2$. It follows that at the point $\phi^{-1}\circ\gamma_x(\pi)$, which lies on $\mathbb{R}_{>0}(0,-1,0)$, the tangent vector to this circle points in the direction $e_1-e_2$. Similarly, at the point $\phi^{-1}\circ\gamma_z(\pi)$, the tangent vector points in the direction $-e_2+e_3$.

The classes $m_2$ and $\ell_1$ in $H_1(L')$ are both mapped to $e_3$ by Lemma~\ref{inducedmap}. The class $m_2$ lies on the second cusp of $L'$, and is thus obtained from a class on the subtorus of $T^3$ spanned by $\lbrace e_2,e_3\rbrace$ under inclusion into $L'$. Similarly, the class $\ell_1$ lies on the first cusp of $L'$, so it comes from a class in the subtorus spanned by $\lbrace e_1,e_3\rbrace$. Taking a small translate of $e_3$ in the subtorus spanned by $\lbrace e_2,e_3\rbrace$ and translating it in the direction $e_1-e_2$ gives a small translate of $e_3$ in the subtorus spanned by $\lbrace e_1,e_3\rbrace$. In $H_1(L')$, this corresponds to adding a class in $H_1(\Lambda)$ to the initial cycle, which is represented by $\ell_1$. Since $\ell_1 = -m_0+m_2$ by Lemma~\ref{relations}, and the resulting cycle is $m_2$, it follows that the class added is $m_0$, so $\phi^{-1}\circ\gamma_x(s)$ represents this class.

Similarly, $m_2+\ell_2$ and $-m_3-\ell_3$ in $H_1(L')$ are both mapped to $e_1$, so a similar application of Lemm~\ref{inducedmap}, where we translate a copy of $e_1$ in the subtorus spanned by $e_1,e_2$ in the direction $-e_2+e_3$ to obtain a curve in the subtorus spanned by $\lbrace e_1,e_3\rbrace$, shows that $\phi^{-1}\circ\gamma_z(t)$ is represented by the class $(m_2+\ell_2)-(-m_3-\ell_3) = -\ell_0\in H_1(\Lambda)$.
\end{proof}

\section{Weinstein neighborhoods and smoothings}\label{nbdhsmthsect}
In this section, we will prove Theorem~\ref{smoothings} from the introduction. In \S\ref{nbhdsect}, we will prove, adapting the techniques of~\cite{Joyce}, a Weinstein neighborhood theorem for conical at infinity, non-compact Lagrangian submanifolds with conical singularities in Liouville manifolds. In \S\ref{smthsect}, we will use the neighborhood of $\Lsing$ obtained this way to give the proof of Theorem~\ref{smoothings}. 

\subsection{Lagrangian neighborhood theorems}\label{nbhdsect}
In this subsection, we will prove a version of the Weinstein neighborhood theorem for Lagrangian submanifolds in Liouville manifolds which have conical singular points and which are cylindrical at infinity. Although the results that we use are similar to those in~\cite{Joyce}, we will state them here for the reader's convenience.

The starting point for the proof of the Lagrangian neighborhood theorem we will use is the following statement from Lagrangian foliations, which is stated in~\cite{Joyce} and extracted from~\cite{Weinstein}. 
\begin{theorem}[Joyce~\cite{Joyce}]\label{foliation} Let $(M,\omega)$ be a symplectic manifold and $N\subset M$ a half-dimensional embedded submanifold. Let $\lbrace L_x\mid x\in N\rbrace$ be a smooth family of embedded, noncompact Lagrangian submanifolds in $M$ for which $x\in L_x$ and $T_x L_x\cap T_x N = \lbrace0\rbrace$ for all $x\in N$. Then there is an open neighborhood $U$ of the zero section $N$ in $T^*N$ such that the fibers of the restricted bundle projection $\pi\colon U\to N$ are connected along with a unique embedding $\Phi\colon U\to M$ with $\Phi(\pi^{-1}(x))\subset L_x$ for each $x\in N$ with the property that $\Phi$ restricts to the identity on $N$ and that $\Phi^*(\omega) = \hat{\omega}+\pi^*(\omega\mid_N)$, where $\hat{\omega}$ is the canonical symplectic structure on $T^*N$.
\end{theorem}
One also has a neighborhood theorem for Lagrangian cones in $C\subset\mathbb{C}^n$, by which we mean singular Lagrangian submanifolds in $\mathbb{C}^n$ with the property that $tC = C$ for any $t\in\mathbb{R}_{>0}$. This can either be proven using Theorem~\ref{foliation}, as in~\cite{Joyce}, or using the Legendrian neighborhood theorem. Such neighborhoods can be made dilation invariant with respect to a certain $\mathbb{R}_{>0}$-action on $T^*(\Sigma\times\mathbb{R}_{>0})$ which we will now describe.

Let $C$ be a Lagrangian cone in $\mathbb{C}^n$ with isolated singularity at $0$ and $\Sigma = C\cap S^{2n-1}$. Define $\iota\colon\Sigma\times\mathbb{R}_{>0}\to\mathbb{C}^n$ by the formula $\iota(\sigma,s) = s\sigma$. For $\sigma\in\Sigma$, $\tau\in T^*_{\sigma}\Sigma$, $s\in\mathbb{R}_{>0}$ and $u\in\mathbb{R}$, let $(\sigma,s,\tau,u)$ denote the point $\tau+uds\in T^*_{(\sigma,r)}(\Sigma\times\mathbb{R}_{>0})$.  Then there is an action of $\mathbb{R}_{\geq0}$ on $T^*(\Sigma\times(0,\infty))$ given by
\[ t\colon(\sigma,s,\tau,u)\mapsto(\sigma,ts,t^2\tau,tu) \]
for $t\in\mathbb{R}_{>0}$. If we let $\hat{\omega}$ denote the canonical symplectic form on $T^*(\Sigma\times\mathbb{R}_{>0})$, then this action has the property that $t^*(\hat{\omega})  t^2\hat{\omega}$.
\begin{lemma}[Joyce~\cite{Joyce}]\label{Weinsteinconical}
There is an open neighborhood $U_C$ of $\Sigma\times\mathbb{R}_{>0}$ in $T^*(\Sigma\times\mathbb{R}_{>0})$ invariant under the action of $\mathbb{R}_{>0}$ and an embedding $\Phi_C\colon U_C\to\mathbb{C}^n$ which restricts to $\iota$ over the zero section, for which the pullback of the usual symplectic form on $\mathbb{C}^n$ is $\hat{\omega}$, and such that $\Phi_C\circ t = t\Phi_C$.
\end{lemma}

Now we will move on to discuss a global version of this theorem. Let $M$ be a Liouville manifold, and let $L\subset M$ be a Lagrangian which is cylindrical at infinity and embedded, except at a discrete set of conical singular points, say $x_1,\ldots,x_m\in L$, all of which are contained in $M^{in}$ away from $\partial M$. Let $\Lambda_1,\ldots,\Lambda_{\ell}$ denote the connected components of the Legendrian $\partial M\cap L$. In particular the Lagrangian submanifold $L_{M^{in}}\coloneqq L\cap M^{in}$ is compact with Legendrian boundary.

Suppose that there are Darboux charts $\Psi_1,\ldots,\Psi_m$ near $x_1,\ldots,x_m$ such that $\Psi_i\colon B_{r}(0)\to M$ is defined on a ball centered at the origin in $\mathbb{C}^n$ and takes the value $\Psi_i(0) = x_i$. Denote by $\psi_i$ the linear isomorphism $d_0\Psi_i\colon\mathbb{C}^n\to T_{x_i}M$.

Further assume that $\Psi_i^{-1}(L)$ is a (dilation-invariant) Legendrian cone in $B_{\epsilon}(0)\subset\mathbb{C}^n$ with Legendrian link denoted $\Sigma_i$. We will let $\iota_i\colon\Sigma_i(0,\epsilon)\to B_{\epsilon}(0)$ denote the map $(\sigma,s)\mapsto s\sigma$.

Define the punctured Lagrangian submanifold $L'\coloneq L\setminus\lbrace x_1,\ldots,x_m\rbrace$. Then we can write $\Psi^{-1}_i(L')$ as the image under $\Phi_{\Sigma_i}$ of the graph of the zero section in $T^*(\Sigma_i\times(0,\epsilon'))$ for some $\epsilon'\in(0,\epsilon]$. Now define $\phi\colon\Sigma_i\times(0,\epsilon')\to B_{\epsilon}(0)$ by $\phi_i(\sigma,s) =  \Phi_{\Sigma_i}(\sigma,s)$. Then $\Psi_i\circ\phi_i$ maps $\Sigma_i\times(0,\epsilon')\to L'$. Let $S_i$ denote the image of $\Psi_i\circ\phi_i$, and define $K$ to be
\[ K\coloneqq L'\setminus\left(S\cup\bigcup_{j=1}^{\ell}(\Lambda_i\times[1,\infty))\right).\]
\begin{lemma}\label{Weinsteinnbhd}
There is an open tubular neighborhood $U_{L'}\subset T^*L'$ of the zero section and a symplectic embedding $\Phi_{L'}\colon U_{L'}\to M$ which restricts to the inclusion $L'\hookrightarrow L$ over the zero section. Moreover, this embedding satisfies
\begin{equation}\label{localdefn}
\Phi_{L'}\circ(d\Psi_i\circ\phi_i) = \Psi_i\circ\Phi_{\Sigma_i}
\end{equation}
for all points $(\sigma,s,\tau,u)\in T^*(\Sigma_i\times(0,\epsilon'))$ in the open neighborhood from Lemma~\ref{Weinsteinconical}.
\end{lemma}

\begin{proof}
We take the equation~\ref{localdefn} to be a definition for $\Phi_{L'}$ over the subset $S$, and so this determines what $U_{L'}$ should be near each of the cone points. Similarly we can define $\Phi_{L'}$ and $U_{L'}$ over the cylindrical ends using the Legendrian neighborhood theorem, meaning that we need only extend $\Phi_{L'}$ over $K$.

To that end, define $L_x = \Phi_{L'}(T_x^*L'\cap U_{L'})$ for all points $x$ in $S$ or in the cylindrical ends.  We see that $L_x$ is an open Lagrangian ball which intersects $L'$ transversely at $x$, and this family depends smoothly on $x$. We can then extend this to a family parametrized by $x\in L'$, since the set where the family isn't already defined is the compact set $K$. Applying Theorem~\ref{foliation} to $\lbrace L_x\mid x\in L'\rbrace$ yields an open neighborhood $U$ of $L'$ in $T^*L'$ and a symplectic embedding $\Phi\colon U\to T^*T^3$ which restricts to the identity on $L'$. By the local uniqueness of Theorem~\ref{foliation} we see that $\Phi_{C_{HL}}$ and $\Phi_{L'}$ coincide where they are both defined.
\end{proof}

\subsection{Desingularizations and lifts of smooth tropical curves}\label{smthsect}
The Harvey--Lawson cone $C_{HL}$ has three asymptotically conical smoothings, given in coordinates by
\[ C_{HL}^i  = \lbrace|z_i|^2-\delta^2 = |z_j|^2 = |z_k|^2,  \; z_1z_2z_3\in\mathbb{R}_{\geq0}\rbrace\subset\mathbb{C}^3\]
for all $\lbrace i,j,k\rbrace = \lbrace 1,2,3\rbrace$. Notice that $C_{HL}^i$ bounds a holomorphic disk $D^i = \lbrace |z_i|^2 \leq \delta^2 \text{ and }z_j = z_k = 0\rbrace$.

We will construct three distinct smoothings of $\Lsing$ by perturbing $L'$ by a Lagrangian isotopy in a Weinstein neighborhood furnished by Lemma~\ref{Weinsteinnbhd}. The isotopy in question is given by the graph of a closed $1$-form on $L'$. We must then check that the resulting non-properly embedded Lagrangian submanifold can be compactified by gluing in one of the smoothings $C_{HL}^i(\epsilon)$. For this purpose, we will start by describing $C_{HL}^i\setminus\partial D^i$ as the graph of a $1$-form on $C_{HL}\setminus 0$.

Using the long exact sequence in cohomology, we have an isomorphism
\[ H^1(C_{HL}^i;\mathbb{R})\cong H^2(\mathbb{C}^3,C_{HL}^i;\mathbb{R}) \, .\]
Let $Y(C_{HL}^i)$ denote the class in $H^1(\Lambda_{HL})$ obtained from $[\omega]\in H^2(\mathbb{C}^3,C_{HL}^i)$, the class of the symplectic form on $\mathbb{C}^3$, under the composition of the isomorphism above with the pullback $H^1(C_{HL}^i;\mathbb{R})\to H^1(\Lambda_{HL})$, where have identified $\Lambda_{HL}$ with $\partial C_{HL}^i$.

The values of $Y(C_{HL}^i)$ are described in~\cite[Example 6.9 and \S{10}]{Joyce2}. Under the identification
\begin{align*}
(e^{is},e^{it})\mapsto\left(\frac{1}{\sqrt{3}} e^{is},\frac{1}{\sqrt{3}} e^{-is-it},\frac{1}{\sqrt{3}} e^{it} \right)
\end{align*}
we have that
\begin{align}\label{coneform}
Y(C_{HL}^1) = (\pi\delta,0) \, \quad Y(C_{HL}^2) = (0,\pi\delta) \, \quad Y(C_{HL}^3) = (-\pi\delta,-\pi\delta)
\end{align}
where $H^1(\Lambda_{HL};\mathbb{R})\cong\mathbb{R}^2$. We can thus write $C_{HL}^i\setminus\partial D^i$ as the graph of a $1$-form on $C_{HL}\setminus 0$ which represents the cohomology class $Y(C_{HL}^i)$ in a neighborhood as in~\ref{Weinsteinconical}.

To prove Theorem~\ref{smoothings}, we must now find deformations of $\Lsing\setminus 0$ corresponding to these asymptotically conical fillings. Recall that $\Lsing$ agrees with the periodized conormal bundles to the legs of $V$ away from a compact subset of $T^*T^3$ containing the cone point. Consider the tropical curve $V_3$ from the introduction. Away from a compact subset of $\mathbb{R}^3$ containing the origin, deforming $V$ to $V_3$ shifts the negative $q_1$-axis by $\epsilon$ in the $-q_2$-direction. Similarly, it shifts the negative $q_2$-axis by $\epsilon$ in the $-q_1$-direction, and the other two legs of $V$ by $\epsilon$ in the $(q_1+q_2)$-direction. These deformations correspond to Lagrangian isotopies of the periodized conormals to the legs of $V$, each of which can be described as the graph of a closed $1$-form on $\mathbb{R}\times T^2$. 

With this understood, we define a $1$-form $\alpha_3$ on an open subset of $L'$ given by a union of collar neighborhoods of the first, second, third, and fourth cusps. We assume that these collar neighborhoods are identified with (portions of) the periodized conormals to the legs of $V$ under the embedding $L'\to T^*T^3$. For example, by Lemma~\ref{inducedmap}, we have that $m_1\mapsto e_2-e_3$ and $\ell_1\mapsto e_3$. The $1$-form $\alpha_3$, restricted to a neighborhood of the first cusp, gives a $1$-form on $T^2\times\mathbb{R}_{>0}$, where we can think of the classes $m_1,\ell_1\in H_1(T^2\times\mathbb{R}_{>0};\mathbb{R})$ as generators. Viewing this $1$-form as an element of the dual space to $H_1(T^2\times\mathbb{R}_{>0};\mathbb{R})$, we see, by Lemma~\ref{inducedmap}, that it should satisfy $\alpha_3(m_1) = -\epsilon$ and $\alpha_3(\ell_1) = 0$, and consequently $\alpha_3(m_1+\ell_1) = -\epsilon$. Since $m_1+\ell_1$ is mapped to $e_2$ and $\ell_1$ is mapped to $e_3$, the $1$-form described this way restricts to the $1$-form on the periodized conormal whose graph gives the expected deformation corresponding to the deformation from $V$ to $V_3$. 

To prove Theorem~\ref{smoothings}, then, we now need only observe that there is no cohomological obstruction to extending these $1$-forms to \textit{closed} $1$-forms defined on $L'$ in such a way that they give one of forms~\eqref{coneform} near the cone point.
\begin{proof}[Proof of Theorem~\ref{smoothings}]
We begin by describing the forms $\alpha_1$, $\alpha_2$, and $\alpha_3$. By the same considerations as above, we can describe these forms in a neighborhood of the periodized conormals to the legs of $V$ by describing their values on the classes $\lbrace m_j,\ell_j\rbrace$ for $j = 1,2,3,4$. Applying Lemma~\ref{inducedmap}, this gives us the following values.
\begin{align*}
\begin{matrix}
\alpha_1(m_1) = 0, & \alpha_1(m_2) = -\epsilon, & \alpha_1(m_3) = 0, & \alpha_1(m_4) = -\epsilon \\
\alpha_1(\ell_1) = \epsilon  & \alpha_1(\ell_2) = \epsilon & \alpha_1(\ell_3) = -\epsilon & \alpha_1(\ell_4) = 0 \\
 & & & \\
\alpha_2(m_1) = \epsilon & \alpha_2(m_2) = \epsilon & \alpha_2(m_3) = \epsilon & \alpha_2(m_4) = \epsilon \\
\alpha_2(\ell_1) = -\epsilon & \alpha_2(\ell_2) = 0 & \alpha_2(\ell_3) = 0 & \alpha_2(\ell_4) = -\epsilon \\
& & & \\
\alpha_3(m_1) = -\epsilon & \alpha_3(m_2) = 0 & \alpha_3(m_3) = 0 & \alpha_3(m_4) = -\epsilon \\
\alpha_3(\ell_1) = 0 & \alpha_3(\ell_2) = -\epsilon & \alpha_3(\ell_3) = -\epsilon & \alpha_3(\ell_4) = 0 \, .
\end{matrix}
\end{align*}
These determine the values of $\alpha_1$, $\alpha_2$, and $\alpha_3$ on $m_0$ and $\ell_0$ since, by Lemma~\ref{relations}, we have that $m_0 = -\ell_1+m_2$ and $\ell_0 = -m_1-m_4$. These give us that
\begin{align*}
\begin{matrix}
\alpha_1(m_0) = 0 & \alpha_1(\ell_0) = 2\epsilon \\
\alpha_2(m_0) = 2\epsilon & \alpha_2(\ell_0) = -2\epsilon \\
\alpha_3(m_0) = -2\epsilon & \alpha_3(\ell_0) = 0 \, .
\end{matrix}
\end{align*}
Comparing this with Lemma~\ref{linkh1gens}, it follows that we can take $\alpha_3$ to agree with $Y(C_{HL}^1)$ near the cone point of $\Lsing$, for appropriately chosen values of $\epsilon$ and $\delta$. Note that the identification $T^*\mathbb{R}^3\cong\mathbb{C}^3$ we used above was \textit{anti}-symplectic, carrying the real locus in $\mathbb{C}^3$ to the base of the cotangent bundle\footnote{This is consistent with the conventions of~\cite[\S{3.2}]{TZ}, where singular Lagrangian fillings are described as branched double covers of the base of the cotangent bundle of a $3$-ball.}, which explains the minus signs in the values of the $\alpha_i$'s at $m_0$ and $\ell_0$.

Similarly, we can take $\alpha_2$ to agree with $Y(C_{HL}^3)$ and $\alpha_3$ to agree with $Y(C_{HL}^2)$ near the cone point. We can thus glue the corresponding fillings of $\Lambda_{HL}$ to the graph of these $1$-forms inside $T^*L'$. Notice that this corresponds to taking Dehn fillings of $L$ that contract the curves $m_0$, $m_0+\ell_0$, and $\ell_0$, respectively. These give the $\infty$-, $1$-, and $0$-fillings of $L'$ along the zeroth cusp, and these are precisely the fillings that yield the graph manifold~\eqref{graphmfd} by the main result of~\cite[\S{1.4}]{MPR}.
\end{proof}
\begin{remark}
We are currently unable to show that the smoothings of Theorem~\ref{smoothings} are Hamiltonian isotopic to tropical Lagrangian lifts of these smooth curves, in the sense of~\cite{HicksReal}. The main difficulty lies in constructing a Hamiltonian isotopy that makes these Lagrangian submanifolds coincide with the periodized conormal to (a portion of) the finite edge of $V_i$. Notably, these smoothings bound holomorphic disks lying over the finite edge, as one would expect from~\cite[Fig. 9]{HicksReal}, meaning that the standard argument used to put exact Lagrangian submanifolds in tropical position do not apply.
\end{remark}

\section{An Immersed Tropical Lagrangian}\label{IL}
In this section, we will construct a cleanly immersed Lagrangian submanifold $\Limm$ whose Floer theory we will later see behaves as one would expect the Floer theory of the direct sum of $\Lsing$ with itself to behave, if it were to be defined. We will also discuss gradings and unobstrucedness for this immersed Lagrangian.

\subsection{Construction} We construct the immersed Lagrangian submanifold of Theorem~\ref{main} using a `doubling' trick, similar to~\cite{AbouzaidSylvan}. We begin by describing a `double' of the Harvey--Lawson cone. This is a Lagrangian immersion in $\mathbb{C}^3$ which, morally, should be thought of as representing two copies of the Harvey--Lawson cone.

Consider the map $w\colon\mathbb{C}^3\to\mathbb{C}$ given by $q(x,y,z) = xyz$. Under this map, the Harvey--Lawson cone $C_{HL}$ is sent to the positive real axis in $\mathbb{C}$. Identify $\mathbb{C}^3\setminus\lbrace xyz = 1\rbrace$ with 
\[ X\coloneqq\lbrace(x,y,z,v)\in\mathbb{C}^3\times\mathbb{C}^*\mid xyz = 1+v\rbrace.\]
Note that this is one of the local models considered in~\cite{AbouzaidSylvan}. Let $D\coloneqq w^{-1}(0)\subset X$ be the normal crossing divisor.
\begin{lemma}\label{productcoords}
The maps $(x,z,w)$ define coordinates on $X\setminus D$ such that the following diagram commutes:
\[\begin{tikzcd}
X\setminus D\arrow{r}{\cong}\arrow{d} & \mathbb{C}^2\times(\mathbb{C}\setminus\lbrace 0,1\rbrace)\arrow{d} \\
X\arrow{r} & \mathbb{C}^2\times\mathbb{C}
\end{tikzcd}\]
\end{lemma}
This enables us to define a local model for the immersed Lagrangian submanifold as a product. Let $L_1$ denote the immersed arc in $\mathbb{C}\setminus\lbrace 0,1\rbrace$ shown in Figure~\ref{immersedarc}. 
\begin{definition} Define the Lagrangian immersion $L_{loc}\subset X$ to be $T^2\times L_1$, where the $T^2$-factor is a Clifford torus $T^2\subset\mathbb{C}^2$.
\end{definition}
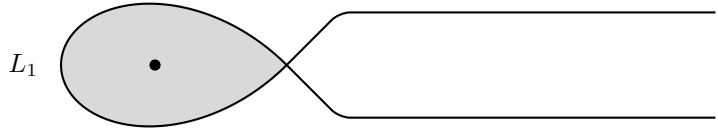
\begin{figure}
\begin{tikzpicture}\draw[smooth, thick, fill = gray!30] 
  plot[domain=135:225,samples=200] (\x:{3*cos(2*\x)});
  
  \draw[rounded corners, thick] (0,0) -- (0.7,0.7) -- (5.7,0.7);
  \draw[rounded corners, thick] (0,0) -- (0.7,-0.7) -- (5.7,-0.7);
  
  \node[] at (-3.5,0) {$L_1$};
  \node[circle, fill, inner sep = 1.5 pt] at (-1.75,0) {};
 
\end{tikzpicture}
\caption{The immersed Lagrangian submanifold $L_1\subset\mathbb{C}$, bounding a holomorphic teardrop through the origin.}\label{immersedarc}
\end{figure}
We can think of this Lagrangian immersion as lying inside the neighborhood $\Psi(B_{\epsilon'})$ constructed during the proof of Lemma~\ref{Weinsteinnbhd}. To construct $\Limm$, we will glue two copies of $L'$, the minimally-twisted five-component chain link complement, to $L_{loc}$. The two copies of $L'$ in question are obtained by perturbing the smooth part of $\Lsing$ in the Weinstein neighborhood constructed in Lemma~\ref{Weinsteinnbhd}. More precisely, they will be obtained as the graphs $\Gamma(\pm dh)$ inside the Weinstein neighborhood of $\Lsing$ constructed in Lemma~\ref{Weinsteinnbhd}, which contains $T^*L'$ as an open subdomain. This choice of perturbation will guarantee that $\Limm$ is exact, but we will also need to specify the behavior of $h$ near the cusps of $L'$ to ensure that $\Gamma(dh)$ and $\Gamma(-dh)$ can be smoothly glued to $L_{loc}$.

\begin{lemma}\label{morsefn}
There is a Morse function $h\colon L'\to[0,\infty)$ with $10$ index $0$ critical points, $20$ index $1$-critical points, and $10$ index $2$ critical points. Moreover, near each cusp of $L'$, we can arrange for $h$ to coincide with the function $h\colon T^2\times(0,\infty)\to(0,\infty)$ given by projection to the radial coordinate in suitable coordinates.
\end{lemma}
\begin{proof} This Morse function is constructed from the ideal triangulation of $L'$, with the numbers of critical points correspond to the number of $3$-cells, $2$-cells, and $1$-cells, respectively, in the triangulation. 
\end{proof}
Recall that we can identify a neighborhood of the point $(0,0)\in T^*T^3$ with an open ball via $\Phi\colon\mathbb{C}^3\cong\mathbb{R}^6\to T^*T^3$, where the second map is the universal cover. Let $U\subset T^*T^3$ denote the image of this ball. In this Darboux chart, the Liouville form on $T^*T^3$ pulls back to the Liouville form on $\mathbb{C}^3$. We can also assume that $U$ is the Darboux ball used to construct the Weinstein neighborhood of $\Lsing$ constructed in Lemma~\ref{Weinsteinnbhd}.

Now consider the graphs $\Gamma(dh)$ and $\Gamma(-dh)$ taken inside the cotangent bundle $T^*L'$. Since we have Hamiltonian isotoped $\Lsing$ near $(0,0)\in T^*T^3$ to agree with the Harvey--Lawson cone, it follows that the intersection $(\Gamma(dh)\cup\Gamma(-dh))\cap U$ pulled back to $\mathbb{C}^3$ and projected to $\mathbb{C}$ under $q\colon\mathbb{C}^3\to\mathbb{C}$ is the union of two open arcs
\[ q\circ\Phi^{-1}((\Gamma(dh)\cup\Gamma(-dh))\cap U) = \lbrace re^{i\alpha}\colon r\in(0,\epsilon)\rbrace\cup\lbrace re^{-i\alpha}\colon r\in(0,\epsilon)\rbrace, \]
for a small (in absolute value) real number $\alpha$. Next, we choose a slightly smaller Darboux ball $U'\subset U$ centered at $0$, and a Hamiltonian isotopy $\psi$ supported in a neighborhood of $U'$ such that
\begin{align*}
&q\circ\Phi^{-1}\circ\psi((\Gamma(dh)\cup\Gamma(-dh))\cap(U\setminus U'))  \\
&= \lbrace x+i\delta\colon x\in(a,b)\rbrace\cup\lbrace x-i\delta\colon x\in(a,b)\rbrace
\end{align*}
agrees with the union of two horizontal arcs in $\mathbb{C}$. Since $\psi((\Gamma(dh)\cup\Gamma(-dh))$ has this form, we can form the union
\[ \Limm\coloneqq L_{loc}\cup\left(\psi(\Gamma(dh)\cup\Gamma(-dh))\setminus U' \right), \]
which is a smooth Lagrangian immersion.

\begin{lemma} The Lagrangian immersion $\Limm$ is cylindrical at infinity, and the cylindrical ends consist of two disjoint Hamiltonian isotopic copies of the periodized conormal bundles over the infinite edges of $V$. Furthermore, $\Limm$ is exact, with a locally constant primitive over the cylindrical ends.
\end{lemma}
\begin{proof}
The statement about cylindrical ends is clear. Recall that the smooth part $L'$ of the Lagrangian submanifold $\Lsing\subset T^*T^3$ is exact and cylindrical at infinity by construction. Since $\psi(\Gamma(dh))$ and $\psi(\Gamma(-dh))$ are both obtained by applying a Hamiltonian isotopy to $L'$, it follows that $\left(\psi(\Gamma(dh)\cup\Gamma(-dh))\setminus U' \right)$ has an induced primitive. By our assumption about the gradient flow of $h$ near the cusps of $L'$, it follows that $\Limm$ remains cylindrical at infinity, and hence that the primitive on $\left(\psi(\Gamma(dh)\cup\Gamma(-dh))\setminus U' \right)$ is necessarily constant on the cylindrical ends.

To finish the proof, we must show that the primitive on $\left(\psi(\Gamma(dh)\cup\Gamma(-dh))\setminus U' \right)$ extends over $L_{loc}$. Since we know that the projection $\Phi^{-1}\circ\psi((\Gamma(dh)\cup\Gamma(-dh)\cap(U\setminus U'))$ under $q$ is the union of two horizontal arcs in $\mathbb{C}$, it follows the restriction of the primitive to $\psi((\Gamma(dh)\cup\Gamma(-dh)\cap(U\setminus U'))$ can be written in the collar
\[ (T^2\sqcup T^2)\times(0,1)\cong\psi((\Gamma(dh)\cup\Gamma(-dh)\cap(U\setminus U'))\]
as the product of a linear function on the $(0,1)$-factor with a constant function on the $T^2\sqcup T^2$-factor. This extends uniquely to a primitive defined on $L_{loc}$ (which is determined by the area of this disk shown in Figure~\ref{immersedarc}.
\end{proof}

Let $\tLimm$ denote the $3$-manifold obtained as follows. Choose a collar neighborhood of the zeroth cusp of $L'$ which we identify with $T^2\times(0,1)$, so that by removing $T^2\times(0,\epsilon)$, we get a manifold with one $T^2$-boundary component. Then define $\tLimm\coloneqq (L'\setminus T^2\times(0,\epsilon))\cup_{T^2\times\lbrace\epsilon\rbrace}(L'\setminus T^2\times(0,\epsilon))$, where the gluing is by the identity on the boundary $2$-torus $T^2\times\lbrace\epsilon\rbrace$. Then we can think of $\Limm$ as the image of a Lagrangian immersion $\tLimm\to T^*T^3$.
\subsection{Grading}
In the next two subsections we will identify all local systems for which $CF^*(\Limm)$ is unobstructed. We begin by determining the gradings of generators of $CF^*(\Limm)$ associated to self-intersections.
\begin{lemma}\label{product}
$\Limm$ is graded, and the grading function $\alpha^{\#}\colon\tLimm\to\mathbb{R}$ is approximately $0$ at the critical points of $-h$ and approximately $1$ at the critical points of $h$.
\end{lemma}
\begin{proof} There is a quadratic volume form $\det^2$ on $T^*T^3$ given by
\[\bigwedge_{j=1}^3(dq_j+id\theta_j)^{\otimes 2} \, .\]
If $\mathcal{L}\to T^*T^3$ denotes the Lagrangian Grassmannian bundle over $T^*T^3$, then the Lagrangian immersion $\Limm$ determines a section $\sigma$ of $\mathcal{L}\mid_{\tLimm}$. We need to show that the composition $\alpha\coloneqq{\text{det}}^2\circ\sigma\colon \tLimm\to S^1$ admits a lift $\alpha^{\#}\colon \tLimm\to\mathbb{R}$.

The restriction of $\alpha$ to the cylindrical ends of $\Limm$ is constantly $1$, so we can take $\alpha^{\#}$ to be locally constant over the ends. Because the inclusion of the cylindrical ends into $\Limm$ induces an isomorphism on $H_1$, it follows that the lift $\alpha^{\#}$ can be defined globally. In the local model for $L_1$, the tangent spaces to $L_1$ change orientation as the curve goes around the origin, so we can take $\alpha^{\#}$ to have the values given in the statement of the lemma.
\end{proof}
The switching components of $\tLimm\times_{T^*T^3}\tLimm$ are either transverse double points, or the $2$-torus in $L_1$. Letting $\iota\colon\tLimm\to T^*T^3$ denote the immersion, each transverse double point is of the form $\iota(p_{-}) = \iota(p_+)$, where $p_{-}$ is a Morse critical point of $-h$ and $p_{+}$ is the corresponding Morse critical point of $h$. Since the two intersecting sheets of $\Limm$ at these points are locally the graphs of $-dh$ and $dh$, it follows from~\citep{seidelgraded} that the index of $p_{-}$ is given by
\begin{align*} \deg(p_{-}) &= \mu_{\mathrm{Morse}}(p_{-})-\alpha^{\#}(p_{-})+\alpha^{\#}(p_+) \\
&= \mu_{\mathrm{Morse}}(p_{-})-\alpha^{\#}(p_{-})+\alpha^{\#}(p_{-})+1 \\
&= \mu_{\mathrm{Morse}}(p_{-})+1.
\end{align*}
Symmetrically, one has that
\begin{align*}
\deg(p_+) = \mu_{\mathrm{Morse}}(p_+)-1
\end{align*}
for generators on the positive sheet. To compute gradings for the Floer generators coming from the $T^2$-family of self-intersection points, we consider the arc $U\subset\mathbb{C}$. A straightforward computation shows that the $T^2$ switching component of $\Limm$ contributes copies of $H^*(T^2)[-2]$ and $H^*(T^2)[1]$ to the Floer cochain space.

\subsection{Unobstructedness} We will see, in Lemma~\ref{unobstructed}, that $\Limm$ will only be an object of the wrapped Fukaya category for certain choices of local system. To determine these local systems, we will need to compute $\mathfrak{m}_0$. Recall that we described a particular choice of spin structure for $L'$ in Lemma~\ref{spinchoice}, and this induces a spin structure on $\tLimm$. Hereafter, we assume that $\tLimm$ is equipped with this spin structure.
\begin{lemma}\label{unobstructed}
Let $\nabla$ be a $GL(1)$-local system on $L'$, and let $\mu_0$ and $\lambda_0$ denote the holonomy of $\nabla$ about the meridian $m_0$ and longitude $\ell_0$. Then 
\[\mathfrak{m}_0 = \pm(1+\mu_0+\lambda_0^{-1}).\]
In particular, $(\tLimm,\nabla)$ is an object of the Fukaya category if and only if the holonomy of $\nabla$ satisfies
\begin{equation}
1+\mu_0^{-1}+\mu_0^{-1}\lambda_0^{-1} = 0. \label{gp-curvature}
\end{equation}
\end{lemma}
\begin{proof}
We will compute $\mathfrak{m}_0$, and show that it vanishes, using the integrable almost complex structure on $(\mathbb{C}^*)^3$, which we identify with $T^*T^3$. That this is justified will follow from our arguments below, where we classify all holomorphic teardrops in $\Limm$, and show that they are all simple. This means that $\mathfrak{m}_0$ will also vanish for a generic choice of $J$ satisfying the conclusion of Lemma~\ref{teardropregularity}.

We will start by classifying all holomorphic teardrops with boundary on $L_{loc}$. Using the decomposition in Lemma~\ref{productcoords}, we see that any holomorphic teardrop with boundary on $L_{loc}$ must be a product, where one factor is the holomorphic teardrop in $\mathbb{C}$ with boundary on $L_1$, and the other factor is either constant or a Maslov $2$ disk with boundary on the Clifford torus in $\mathbb{C}^2$. All three of these disks are regular with respect to the integrable almost complex structure on $T^*T^3\cong(\mathbb{C}^*)^3$ by the argument in~\cite{AbouzaidSylvan}.

We claim that there are no other holomorphic teardrops with boundary on $\Limm$. In particular, it will follow that all holomorphic teardrops with boundary on $\Limm$ are simple. Suppose to the contrary that there were such a teardrop. In that case, it would either need to have a corner on the $2$-dimensional switching component of $L_{loc}$, or a corner on a $0$-dimensional switching component corresponding to a critical point of the Morse function $h$ from Lemma~\ref{morsefn}. Let $U$ denote the Darboux chart around $(0,0)\in T^*T^3$ in which $L_{loc}$ was constructed.

In the former case, recall that any holomorphic teardrop with a corner on a $2$-dimensional switching component has symplectic area completely determined by the values primitive. By replacing $h$ with $\epsilon h$, for an arbitrarily small constant $\epsilon>0$, we can arrange for this \textit{a priori} energy bound to be made arbitrarily small. In particular, we can assume that no such teardrop can exit $U$, and thus it must agree with one of the teardrops already classified.

We must also rule out the existence of holomorphic teardrops with a corner at one of the $0$-dimensional switching components. Suppose that we were given such a teardrop $u\colon S\to T^*T^3$ (whose domain $S$ is the unit disk in $\mathbb{C}$ with one boundary puncture), and consider the projection $u(S)\cap U\to\mathbb{C}$ given by $w$ of Lemma~\ref{productcoords}. This must be a region in $\mathbb{C}$ whose boundary is the union of $L_1$ with a circular arc (contained in the image of $\partial U$ under this projection). Again, by exactness, we can construct $\Limm$ so that the area of this disk teardrop must be arbitrarily small. For an appropriate choice of $h$, then, it follows that the image of $u(S)\cap U$ must be the region inside $L_1$. But by the open mapping theorem, this cannot be the image of a holomorphic curve.

By our choice of coordinates near the singular point of $L$, it follows that the two Maslov $2$-disks in $\mathbb{C}^2$ will have boundary circles homologous to $m_0$ and $\ell_0^{-1}$ in $\tLimm$. The spin structure we have chosen for $L'$ induces a spin structure on $\tLimm$, and using this spin structure all three of the holomorphic teardrops described here contribute with the same sign.
\end{proof}
\section{Floer-theoretic support and the Grassmann--Pl{\"u}cker relation}\label{supportsection}
We will prove Theorem~\ref{main} in this section, showing that the Floer-theoretic support of an unobstructed Lagrangian  brane $(\Limm,\nabla)$ is given by the very affine part of a line in $\mathbb{P}^3_{\mathbb{K}}$. Recall that each local system $\nabla_p$ on the zero-section $T^3\subset T^*T^3$ corresponds, under the equivalence of Theorem~\ref{hms}, to the skyscraper sheaf $\mathcal{O}_p$ of the point $p\in(\mathbb{K}^*)^3$. Determining when the Floer cohomology group
\begin{align}\label{suppgp}
HW^*((\Limm,\nabla),(T^3,\nabla_p))
\end{align}
is nonzero thus amounts to calculating the support of the mirror sheaf to $(\Limm,\nabla)$, whose existence follows from Theorem~\ref{hms}.

It is not obvious how $\Limm$ and $T^3$ intersect each other, so instead of computing~\eqref{suppgp} directly, we will instead study the Floer cohomology of $(\Limm,\nabla)$ with other objects in the wrapped Fukaya category. These objects are supported on the periodized conormal lifts $L_{\langle q_i\rangle}$ of the coordinate lines $\langle q_i\rangle\subset Q$ in the cotangent fiber of $T^*T^3$. Any local system on $T^3$ determines a local system on $L_{\langle q_i\rangle}$.

Recall that if $\lbrace i,j,k\rbrace = \lbrace 1,2,3\rbrace$, then $L_{\langle q_i\rangle}$ can be identified with the product of $\langle q_i\rangle$ and the subtorus of $T^3$ given by $\left\lbrace\theta_i = \frac{1}{2}\right\rbrace$. Thus the generators $e_j,e_k\in H_1(T^3)$ determine a basis for $H_1(L_{\langle q_i\rangle})$.
\begin{definition} Given a local system $\nabla_p$ on $T^3$ with holonomy representation $\nabla_p(e_i) \coloneqq \alpha_i\in\mathbb{K}^*$, let $\nabla_{p,i}$ denote the local system whose holonomy representation is determined by $\hol_{\nabla_{p,i}}(e_j) = \alpha_j$ and $\hol_{\nabla_{p,i}}(e_k) = \alpha_k$.
\end{definition}
By Lemma~\ref{conormalsupport}, the objects $(L_{\langle q_i\rangle},\nabla_{p,i})$ are mirror to sheaves supported on algebraic subtori of $(\mathbb{K}^*)^3$ of the form
\[ C_{p,i}\coloneqq\lbrace z_j = \alpha_j\text{ and }z_k = \alpha_k\rbrace\subset(\mathbb{K}^*)^3 \, .\]
Such a curve contains the point $p\in\mathbb{K}^*$, and the mirror sheaf to $(L_{\langle q_i\rangle},\nabla_{p,i})$ turns out to be the structure sheaf $\mathcal{O}_{C_{p,i}}$, since this sheaf has rank $1$ at all points in its support by Lemma~\ref{conormalsupport}, and because $C_{p,i}\cong\mathbb{K}^*$. We will determine the support of the mirrors to $(\Limm,\nabla)$ by determining when
\begin{align}\label{testgroup}
HW^0((\Limm,\nabla),(L_{\langle q_i\rangle},\nabla_{p,i}))
\end{align}
is nonzero. The nonvanishing of this Floer cohomology group implies the nonvanishing of $\Ext^0$ between the corresponding mirror sheaves, so by the local-to-global spectral sequence, there must be a point in the intersection of their supports. Conversely, one has that $\mathcal{O}_{C_{p,j}}\otimes\mathcal{O}_{C_{p,k}}\cong\mathcal{O}_p\in D^b\mathrm{Coh}((\mathbb{K}^*)^3)$, so by hom-tensor adjunction in the derived category (and mirror symmetry), any point for which one of the associated Floer cohomology groups~\eqref{testgroup} vanishes cannot lie in the support of the mirror sheaf to $(\Limm,\nabla)$.
\begin{remark}
A previous version of this article phrased the calculation of the cohomology groups~\eqref{testgroup} in terms of Lagrangian correspondences, but, after carefully perturbing $\Limm$, we can determine when they are nonzero using elementary methods more easily.
\end{remark}
The advantage of considering~\eqref{testgroup}, is that part of the Floer cochain complex can be naturally identified with the Floer cochain complex of a tropical Lagrangian pair of pants in $T^*T^2$ with the $0$-section (cf. Lemma~\ref{pantssupport}). One can think of~\eqref{pantsdiff} as a `leading order term' in the wrapped Floer differential, in the sense the corresponding strips with boundary on $\Limm\cup L_{\langle q_i\rangle}$ will be the lowest-energy strips. The relevant brane structures on $\Lpants$ and $T^2$ are determined by the given brane data on $\Limm$ and $T^3$. We will compute this leading order term first, and then explain why this computation yields the expected (non-)vanishing results for~\eqref{testgroup}.

Recall that both $\Lsing$ and $\Lpants$ can be thought of as (closures of) the graphs of exact $1$-forms defined on coamoebae in $T^3$ and $T^2$, respectively.
\begin{definition}
In this section, we denote by $\Delta_3$ the coamoeba for $\Lsing$ $\Delta_2$ denote the coamoeba for $\Lpants$. We let $g_3\colon\Delta_3\to\mathbb{R}$ denote~\eqref{coamoebaprimitive} and $g_2\colon\Delta_2\to\mathbb{R}$ denote~\eqref{coamoebaprimitivepants}.
\end{definition}

Let $\lbrace i,j,k\rbrace = \lbrace 1,2,3\rbrace$. By intersecting $\Delta_3$ with the subtorus $T^2\cong\left\lbrace\theta_i = \frac{1}{2}\right\rbrace$ inside $T^3$, we obtain a copy of $\Delta_2$ (cf. Figure~\ref{Coamoeba}). Additionally, restricting $g_3$ to this subtorus, by setting $\theta_i = \frac{1}{2}$, yields a real-valued function on $\Delta_2$, which coincides with $g_2$ up to relabeling coordinates. Examining~\eqref{coamoeba1form}, it then follows that there is a subsurface of $\Lsing$, which we call $\Sigma_{i}$, which can be identified with the graph of a smooth function, namely $\frac{\partial g}{\partial\theta_i}$, on $\Lpants$. It will be convenient, for the purpose of calculating~\eqref{testgroup}, to modify $\Lsing$ by a Hamiltonian isotopy so that this subsurface looks like the graph of a constant function.
\begin{remark}\label{totallygeodesic}
In fact, the surface $\Sigma_{i}$ can be identified with a totally geodesic surface in the minimally-twisted five-component chain link complement: one can take the union of the faces of the ideal tetrahedra $T_{0,-}$ and $T_{0,+}$ in the ideal triangulation of $L'$ which intersect the $j$th, $k$th, and fourth vertices.
\end{remark}
In the following we identify $T^*T^3\cong(T^*S^1)^3\cong(\mathbb{R}\times S^1)^3$, where the $i$th factor comes last.
\begin{lemma}\label{intersections3d}
Up to Hamiltonian isotopy, we can assume that the intersection of $\Lsing$ with $(T^*S^1)^2\times\left(\mathbb{R}\times\left\lbrace\frac{1}{2}\right\rbrace\right)$ coincides with $\Lpants\subset T^*T^2$, and that $\Lsing\cap L_{\langle q_3\rangle}$ coincides with $\Lpants\cap T^2\subset T^*T^2$.
\end{lemma}
\begin{proof}
First observe that the intersection $\Lsing\cap(T^*S^1)^2\times\left(\mathbb{R}\times\left\lbrace\frac{1}{2}\right\rbrace\right)$, where $\Lsing$ is constructed using~\eqref{coamoeba1form}, consists of precisely two points, which project to the barycenters of the triangles in $\Delta_2 = \Delta_3\cap\left\lbrace\theta_i = \frac{1}{2}\right\rbrace$. These two intersection points correspond to the zeroes of the first two components of~\eqref{coamoeba1form}.

Recall that in the constructions of $\Lpants$ and $\Lsing$, we can replace $g_2$ and $g_3$ with $\lambda g_2$ and $\lambda g_3$, respectively, for an arbitrarily small real constant $\lambda>0$. Choosing sufficiently small $\lambda$, it follows that the pair of pants $\Sigma_i$ in $\Lsing$ given by the intersection $\Sigma_i\coloneqq\Lsing\cap(T^*S^1)\times(\mathbb{R}\times\left\lbrace\frac{1}{2}\right\rbrace)$, as discussed above, is arbitrarily $C^1$-close to $\Lpants\times\left\lbrace\left(0,\frac{1}{2}\right)\right\rbrace$. Abusing notation, we will call this isotropic surface as $\Lpants$ as well. 

The normal bundle $N\Sigma_{jk}$ of $\Sigma_{jk}$ in $\Lsing$ and $\Lpants\times\mathbb{R}\subset T^*T^3$ are smoothly isotopic, and since both of these are exact Lagrangian submanifolds, this isotopy has zero flux. Since these two submanifolds are also $C^1$-close, it follows that there is a Hamiltonian isotopy carrying $N\Sigma_{jk}$ to $\Lpants\times\mathbb{R}$. We can extend this by the identity outside of a small open neighborhood containing these two submanifolds, so that no new intersection points between $\Lsing$ and $(T^*S^1)^2$ lying away from $N\Sigma_{jk}$ are introduced. Thus $\Lsing$ will intersect $\Lpants\times\mathbb{R}$ in $\Lpants\times(-\epsilon,\epsilon)$, for some interval $(-\epsilon,\epsilon)\subset\mathbb{R}$. We can compose this with another Hamiltonian isotopy which makes $\Lsing$ intersect with $(T^*S^1)^2$ and $L_{\langle q_i\rangle}$ as desired.
\end{proof}
This lemma allows us to identify some of the holomorphic strips with boundary on $\Lsing\cup L_{\langle q_i\rangle}$. Recall that there is a symplectomorphism
\begin{align*}
T^*T^3 &\cong (\mathbb{C}^*)^3 \\
(q_j,\theta_j) &\mapsto e^{q_j+i\theta_j} \,,\quad j = 1,2,3 \,.
\end{align*}
\begin{corollary}\label{strips3d}
The only holomorphic strips bounded by $\Lsing\cup L_{\langle q_i\rangle}$ are contained in $(\mathbb{C}^*)^2\subset(\mathbb{C}^*)^3$, and coincide with the holomorphic strips bounded by $\Lpants\cup T^2$.
\end{corollary}
\begin{proof}
Any such strip must be the graph of a holomorphic function on a strip bounded by $\Lpants\cup T^2$. By Lemma~\ref{intersections3d}, this function must be constant on the boundary of the strip, and by the maximum principle it is also constant over the interior of the strip.
\end{proof}
The Morse function $h\colon L'\to[0,\infty)$ of Lemma~\ref{morsefn} used to construct $\Limm$ was obtained from the ideal triangulation $L'$ discussed in \S\ref{chainlinkdescription}. By Remark~\ref{totallygeodesic}, we can arrange $h$ so that it restricts to a Morse function on each of the surfaces $\Sigma_i$ (with two index $0$ and three index $1$ critical points each). This fact allows us to determine the intersection points of $\Limm$ with $L_{\langle q_i\rangle}$ and the holomorphic strips they bound using the results above.
\begin{corollary}\label{strips3dimm}
The set $\Limm\cap L_{\langle q_i\rangle}$ coincides with $\Gamma(dh)\cap L_{\langle q_i\rangle}\coprod\Gamma(-dh) L_{\langle q_i\rangle}$, which can in turn be naturally identified with
\[(\Lpants\cap T_2)\coprod(\Lpants\cap T_2) \,. \]
Similarly, $\Limm\cup L_{\langle q_i\rangle}$ bounds six holomorphic strips, three of which have a boundary component on $\Gamma(dh)$ and three of which have a boundary component on $\Gamma(-dh)$. Both sets of strips are in natural bijection with the set of strips identified in Corollary~\ref{strips3d}.
\end{corollary}
\begin{proof}
We can arrange for $\Gamma(\pm dh)$ to intersect $L_{\langle q_i\rangle}$ as described by choosing $h$ to be sufficiently $C^2$-small. This also implies the claim about holomorphic strips with boundary on $\Limm\cup L_{\langle q_i\rangle}$. There are no other points of intersection between $\Limm$ and $L_{\langle q_i\rangle}$, since the cone point of $\Lsing$ lies away from $L_{\langle q_i\rangle}$.
\end{proof}

The wrapped Floer complex will be generated by these intersection points, together with chords at infinity of a quadratic Hamiltonian, which we will now construct. Let $\pi\colon T^*T^3\to Q$ denote projection to the cotangent fiber, and fix a ball $D\subset Q$ centered at the origin. Using the notation of \S\ref{immersionbackground}, let $(T^*T^3)^{in}\coloneqq\pi^{-1}(D)$. We can assume without loss of generality that there is another ball $D'\subset D$ for which $\Limm$ can be written as the cone over a Legendrian (in $\pi^{-1}(\partial D')$) outside of $\pi^{-1}(D')$. This follows from our assumptions about the form of the Morse function of Lemma~\ref{morsefn} near the cusps of $L'$, which imply that outside $\pi^{-1}(D')$, we can assume that $\Limm$ is in tropical position. In other words, its restriction to the complement of $\pi^{-1}(D')$ is given by two disjoint copies of the periodized conormal lifts to the legs of $V$. These two copies differ from each other by a translation in the $T^3$-direction.

\begin{lemma}\label{wrappedcx}
For an appropriate choice wrapping Hamiltonian $H\colon T^*T^3\to\mathbb{R}$, the wrapped Floer cochain complex can be identified with
\begin{align*}
CW^*(\Limm,L_{\langle q_i\rangle}) \cong & CW^*(\Lpants,T^2)\oplus CW^*(\Lpants,T^2)[-1] \\
&\oplus \bigoplus_{n\in\mathbb{N}}H^*(T^2)\oplus\bigoplus_{n\in\mathbb{N}}H^*(T^2)[-1] \, .
\end{align*}
as a $\mathbb{K}$-vector space.
\end{lemma}
\begin{proof}
We will use a wrapping Hamiltonian which vanishes vanishes in $\pi^{-1}(D')$, and is quadratic outside of $\pi^{-1}(D)$. Note that the natural radial variable in the end of $T^*T^3$ is just the radial coordinate in the base $Q$. Let $\phi^1_H$ denote the time-$1$ flow of such a Hamiltonian. We can assume, without loss of generality, that $\Limm$ and $(\phi^1_H)^{-1}(L_{\langle q_i\rangle})$ do not intersect in $\pi^{-1}(D)\setminus\pi^{-1}(D')$. The intersection $\Limm\cap(\phi^1_H)^{-1}(L_{\langle q_i\rangle})$ is non-transverse, consisting of two countably infinite families of $2$-tori (where one family of intersections lies in $\Gamma(dh)$, and the other in $\Gamma(-dh)$). We can perturb the Legendrian link $\pi^{-1}(\partial D)\cap\Limm$ using perfect Morse functions on each of the (eight) $T^2$-components. The new Hamiltonian chords in the cylindrical end thus correspond to critical points of (copies of) this Morse function. The grading shift in one of the summands follows from Lemma~\ref{product}.
\end{proof}

The local system $\nabla$ restricts to a local system on each of the surfaces $\Sigma_{i}$. Letting $e_i,e_j,e_k\in H_1(T^3;\mathbb{Z})$ denote the generators considered in Lemma~\ref{inducedmap}, we see that we can identify $e_j$ and $e_k$ with generators for $H_1(\Sigma_{i};\mathbb{Z})$. Moreover, the holonomies of these local systems about these generators can be determined directly from Lemma~\ref{inducedmap}.
\begin{figure}
\begin{tikzpicture}
\begin{scope}[scale =  4]
 \draw[dashed] (0,0,0) -- (1,0,0);
 \draw[dashed] (0,0,0) -- (0,1,0);
 \draw (1,0,0) -- (1,1,0);
 \draw (1,0,0) -- (1,0,1);
 \draw (0,1,0) -- (1,1,0);
 \draw (0,1,0) -- (0,1,1);
 \draw (0,1,1) -- (1,1,1);
 \draw (1,0,1) -- (1,1,1);
 \draw (1,1,0) -- (1,1,1);

 \draw[red] (0,0,0) -- (1,0,1) -- (0,1,1) -- cycle;
 \draw[red] (0,0,0) -- (0,0,1);
 \draw[red] (1,0,1) -- (0,0,1);
 \draw[red] (0,1,1) -- (0,0,1);
 
 \draw[blue] (1,1,1) -- (1,0,0) -- (0,1,0) -- cycle;
 \draw[blue] (1,1,1) -- (1,1,0);
 \draw[blue] (1,0,0) -- (1,1,0);
 \draw[blue] (0,1,0) -- (1,1,0);

\draw[thick, green] (0,0.5,1)--(1,0.5,1);
\draw[thick, green] (0,0.5,1)--(0,0.5,0);

\node[anchor = south east] at (1,0.5,1) {\color{green} $m_2+\ell_2$};
\node[anchor = west] at (0,0.5,0) {\color{green} $m_1+\ell_1$};
\end{scope}

\begin{scope}[yshift = -200]
\draw[thick] (0,0,0) -- (-4,0,0);
 \draw[thick] (0,0,0) -- (0,0,5);
 \draw[thick] (0,0,0) -- (3,0,-3);

\node at (-2,0,0) [circle,fill = green,inner sep=1.5pt]{};

\node at (0,0,2.5) [circle,fill = green,inner sep=1.5pt]{};

\draw[thick, green] (-2,2,-2)--(-2,2,0.5);
\draw[] (-2,2.5,-2)--(-2,2.5,0.5);
\draw[] (-2,1.5,-2)--(-2,1.5,0.5);
\draw[] (-2,1.5,-2) -- (-2,2.5,-2);
\draw[] (-2,1.5,0.5) -- (-2,2.5,0.5);

\node[anchor = west] at (-2,2,-2) {\color{green} $m_1+\ell_1 = e_2$};

\draw[thick, green] (3,0,2.5) -- (1,0,2.5);
\draw[] (3,0.5,2.5) -- (1,0.5,2.5);
\draw[] (3,-0.5,2.5) -- (1,-0.5,2.5);
\draw[] (3,0.5,2.5) -- (3,-0.5,2.5);
\draw[] (1,0.5,2.5) -- (1,-0.5,2.5);

\node[anchor = west] at (3,0,2.5) {\color{green} $m_2+\ell_2 = e_1$};

\end{scope}
\end{tikzpicture}
\caption{The classes $e_1,e_2\in H_1(\Sigma_3)$, drawn in the fibers over points on the legs of a tropical pair of pants (bottom), and their images on the coamoeba $\Delta_3$ labeled by the corresponding elements of $H_1(L')$.}\label{restrictedlsfig}
\end{figure}
\begin{lemma}\label{restrictedls}
Let $\nabla_{i}$ denote the restriction of the local system $\nabla$ on $\Limm$ to $\Sigma_{i}$. These local systems are determined by the holonomy representations 
\[\hol_{\nabla_{i}}\colon H_1(\Sigma_{i};\mathbb{Z})\to\mathbb{K}^{*}\]
whose values on the generators described above are
\begin{align*}
\begin{matrix}
\hol_{\nabla_{1}}(e_2) = \lambda_3^{-1} \,, & \hol_{\nabla_{1}}(e_3) = \mu_2 \,, \\
\hol_{\nabla_{2}}(e_1) = \mu_3^{-1}\lambda_3^{-1}  \,, & \hol_{\nabla_{2}}(e_3) = \lambda_1 \,, \\
\hol_{\nabla_{3}}(e_1) = \mu_2\lambda_2 \,, & \hol_{\nabla_{3}}(e_2) =  \mu_1\lambda_1\,.
\end{matrix}
\end{align*}
Define $\rho_{i,j}\coloneqq\hol_{\nabla_{i}}(e_j)$ and $\rho_{i,k}\coloneqq\hol_{\nabla_{i}}(e_k)$.
\end{lemma}
\begin{proof}
We will determine the values of $\hol_{\nabla_3}(e_1)$ and $\hol_{\nabla_3}(e_2)$, referring to Figure~\ref{restrictedlsfig}. The class $e_1\in H_1(\Sigma_3)$ is, by definition, represented by a closed curve lying over the leg of a tropical pair of pants in the $-q_2$-direction (cf. Figure~\ref{restrictedlsfig}(bottom)). This means that, under the embedding $\Sigma_3\subset L'$, this homology class is represented by a curve lying on the second cusp of $L'$, which is mapped to a conormal fiber over the corresponding leg of $V$ (cf. Figure~\ref{restrictedlsfig}(top)). Let $T^2_2$ denote this cusp. The value of $\hol_{\nabla_3}(e_1)$ is the value of $\hol_{\nabla}(\gamma)$, where $\gamma\in H_1(L')$ is the unique class which lies in the image of $H_1(T^2_2)\to H_1(L')$ and which is mapped to $e_1$ under $H_1(L')\to H_1(T^3)$ (cf. Figure~\ref{restrictedls}). The result of Lemma~\ref{inducedmap} shows that this class must be $m_2+\ell_2$, since the result of that Lemma gives an isomorphism between $H_1(T^2_2)$ and the homology of the subtorus of $\mathrm{span}\lbrace e_1,e_2\rbrace\subset T^3$. Thus $\hol_{\nabla_3}(e_1) = \mu_2\lambda_2$. Similarly, one has that $\hol_{\nabla_3}(e_2) = \mu_1\lambda_1$, since the class $m_1+\ell_1$ in the homology of the first cusp of $L'$ is mapped to $e_2\in H_1(T^3)$, and hence is the image of $e_2\in H_1(\Sigma_3)$, by Lemma~\ref{inducedmap}. The holonomy representations for $\nabla_1$ and $\nabla_2$ are calculated similarly, and we omit the details.
\end{proof}

With the above understood, the computation of the leading order term of the Floer differential now reduces to the same argument as in the proof of Lemma~\ref{pantssupport}.
\begin{lemma}\label{hfcomp}
Let $\nabla_p$ be a local system on $T^3$, and consider the induced local systems $\nabla_{p,i}$ on $L_{\langle q_i\rangle}$ as above. The wrapped Floer cohomology group~\eqref{testgroup} is nonzero in the following cases:
\begin{description}
\item[(i)] $HW^*((\Limm,\nabla),(L_{\langle q_1\rangle},\nabla_{p,1}))\neq0$ if
\begin{align}
\lambda_{3}x_2-\mu_{2}^{-1}x_3+1 = 0 \,. \label{rel1}
\end{align}

\item[(ii)] $HW^*(\Limm,\nabla),(L_{\langle q_2\rangle},\nabla_{p,2}))\neq0$ if
\begin{align}
-\mu_{3}\lambda_{3}x_1+\lambda_{1}^{-1}x_3+1 = 0 \,. \label{rel2}
\end{align}

\item[(iii)] $HW^*(\Limm,\nabla),(L_{\langle q_3\rangle},\nabla_{p,3}))\neq0$ if
\begin{align}
\mu_{2}^{-1}\lambda_{2}^{-1}x_1-\mu_{1}^{-1}\lambda_{1}^{-1}x_2+1 = 0 \,. \label{rel3}
\end{align}

\end{description}
If any of these conditions are satisfied, the corresponding wrapped Floer cohomology group contains a nonzero summand in degrees $0$ and $1$.
\end{lemma}
\begin{proof}
By Corollary~\ref{strips3dimm}, the Floer cochain complex contains, as a $\mathbb{K}$-vector space, a direct summand naturally isomorphic to
\begin{align*}
CW^*(\Lpants,T^2)\oplus CW^*(\Lpants,T^2)[-1]
\end{align*}
where the grading shift on the second factor follows from Lemma~\ref{product}. Let $x_{-,0}\in CF^0(\Lpants,T^2)$ and $x_{+,0}\in CF^1(\Lpants,T^2)$ denote the generators corresponding to these intersection points. Similarly, let $x_{-,1}\in CF^0(\Lpants,T^2)[-1]$ and $x_{+,1}\in CF^1(\Lpants,T^2)[-1]$ denote the generators in the shifted summand. This is a subcomplex of $CW^*((\Limm,\nabla),(L_{\langle q_i\rangle},\nabla_{p,i}))$ since, for action reasons, there are no inhomogeneous strips with an input at $x_{-,0}$ or $x_{-,1}$ and with output on a chord in the cylindrical end. For degree reasons, we also see that $x_{-,0}$ and $x_{-,1}$ cannot be coboundaries.

Thus to determine when the generators of this subcomplex are nonzero in cohomology, it suffices to compute the leading order term of the Floer differential using the integrable almost complex structure on $T^*T^3\cong(\mathbb{C}^*)^3$. Again by Corollary~\ref{strips3dimm}, and by the argument in the proof of Lemma~\ref{pantssupport}, there is one possibly nontrivial Floer differential on each summand of this complex, and by Lemma~\ref{restrictedls}, they are given by 
\begin{align*}
x_{-,0}\mapsto(\pm\rho_{i,j}^{-1}x_j\pm\rho_{i,k}^{-1}x_k\pm 1) x_{+,0}\\
x_{-,1}\mapsto(\pm\rho_{i,j}^{-1}x_j\pm\rho_{i,k}^{-1}x_k\pm 1)x_{+,1} \, .
\end{align*}
Ignoring signs, these differentials coincide with the expressions appearing in the statement of the lemma.

To determine the signs, we first consider the restriction of the spin structure on $L'$ chosen in Lemma~\ref{spine} to a spin structure on the normal bundles to the surfaces $\Sigma_i$. Observe that the trivalent graph drawn on the right-hand side of Figure~\ref{strips} determines the $1$-skeleton in a CW decomposition of $\Sigma_i$. The intersection of the smoothed spine of Lemma~\ref{spine} with the faces $T_{0,\pm}$ can be thought of as a smoothing of this $1$-skeleton. A small thickening of this smoothed $1$-skeleton supports three canonically defined vector fields, which determine the spin structure. By our choice of branching in Lemma~\ref{spinchoice}, this smoothing always contains a smooth arc which looks like a smoothing of the concatenation of the red and green arcs in Figure~\ref{strips}. The other arc of the smoothed spine on $\Sigma_i$ looks like a smoothing of the concatenation of the blue arc with one of the other arcs (determined by the branching). This means that the sign of the last monomial, in each of~\eqref{rel1}--\eqref{rel3}, will always have the same sign as exactly one of the other monomials (depending on whether the two sheets of the smoothed spine intersect in the red arcs or green arcs of Figure~\ref{strips}). Hence the differentials are given by the left-hand sides of~\eqref{rel1}--\eqref{rel3}, up to an overall sign.
\end{proof}
\begin{remark}[Signs]
Although we have fixed a specific choice of spin structure, any other spin structure we could have chosen on $\tLimm$ would have had similar symmetries with respect to different pairs of coordinates.

In the proof of Theorem~\ref{main}, we will identify the space of unobstructed $GL(1)$-local systems on $\tLimm$ with a subspace of the Grassmannian $Gr(2,4)\subset\mathbb{P}^5$. There are $32$ automorphisms of $\mathbb{P}^5$ induced by scaling coordinates by $-1$. These correspond to the $32$ local systems on $M$, since $H_1(M;\mathbb{Z}/2)\cong(\mathbb{Z}/2)^5$. Hence a change of spin structure on $\tLimm$ corresponds to acting on $Gr(2,4)$ by an ambient automorphism of $\mathbb{P}^5$, which gives us canonical isomorphisms between the resulting moduli spaces of brane structures. 
\end{remark}
The geometry of the moduli space of lines in $\mathbb{P}^3$ is reflected in the Floer-theoretic results we have obtained so far, even before incorporating wrapping.
\begin{proposition}\label{plucker-unobs}
Suppose that $(\Limm,\nabla)$ is unobstructed. Then  the (homogenized) linear relations~\eqref{rel1}-\eqref{rel3} define a line in $\mathbb{P}^3$.
\end{proposition}
\begin{proof}
The relations~\eqref{rel1}-\eqref{rel3}, determine a subvariety of $(\mathbb{K}^*)^3$. We will check that this subvariety is the very affine part of a line in $\mathbb{P}^3$. Multiplying Equation~\eqref{rel1} by $\mu_1^{-1}$, Equation~\eqref{rel2} by $\mu_1^{-1}\mu_3^{-1}$, and Equation~\eqref{rel3} by $\mu_3^{-1}$, and applying Lemma~\ref{relations}, we obtain:
\begin{align*}
\mu_1^{-1}\mu_2^{-1}\mu_4 x_2-\mu_1^{-1}\mu_2^{-1}x_3+\mu_1^{-1} &= 0 \,,\\
-\mu_1^{-1}\mu_2^{-1}\mu_4 x_1+\mu_0\mu_1^{-1}\mu_2^{-1}\mu_3^{-1}x_3+\mu_1^{-1}\mu_3^{-1} &=0 \,, \\
\mu_1^{-1}\mu_2^{-1}x_2 - \mu_0\mu_1^{-1}\mu_2^{-1}\mu_3^{-1}x_2 + \mu_3^{-1} &= 0 \,.
\end{align*}
Setting
\begin{align}\label{gpemb}
\begin{matrix}
\phi_{12} \coloneqq \mu_1^{-1}\mu_2^{-1}\mu_4 \,, & \phi_{13} \coloneqq -\mu_1^{-1}\mu_2^{-1} \,, & \phi_{14} \coloneqq \mu_1^{-1} \,, \\
\phi_{23} \coloneqq \mu_0\mu_1^{-1}\mu_2^{-1}\mu_3^{-1} \,, & \phi_{24}\coloneqq \mu_1^{-1}\mu_3^{-1} & \phi_{34}\coloneqq \mu_3^{-1} \,,
\end{matrix}
\end{align}
we can rewrite these linear relations more compactly as
\begin{align*}
\phi_{12}x_2+\phi_{13}x_3+\phi_{14} &= 0, \\
-\phi_{12}x_1+\phi_{23}x_3+\phi_{24} &= 0, \\
-\phi_{13}x_1-\phi_{23}x_2+\phi_{34} &= 0.
\end{align*}
To check that these define a line in $\mathbb{P}^3$, it suffices to check that the Grassmann--Pl{\"u}cker relation
\[ \phi_{12}\phi_{34}-\phi_{13}\phi_{24}+\phi_{14}\phi_{23} = 0,\]
is satisfied. Note that replacing $\phi_{12}$ with $-\phi_{12}$ yields the usual sign conventions for the Pl{\"u}cker relations. One verifies this relation by computing
\begin{align*}
&\phi_{23}^{-1}\phi_{14}^{-1}\phi_{12}\phi_{34} \\
&= (\mu_0^{-1}\mu_1\mu_2\mu_3)\mu_1(\mu_1^{-1}\mu_2^{-1}\mu_4)(\mu_3^{-1}) \\
&= \mu_0^{-1}\mu_1\mu_4 \\
&= \mu_0^{-1}\lambda_0^{-1}
\end{align*}
and
\begin{align*}
&\phi_{23}^{-1}\phi_{14}^{-1}\phi_{13}\phi_{24} \\
&= (\mu_0^{-1}\mu_1\mu_2\mu_3)\mu_1(-\mu_1^{-1}\mu_2^{-1})(\mu_1^{-1}\mu_3^{-1}) \\
&= -\mu_0^{-1}.
\end{align*}
But by Lemma~\ref{unobstructed}, the relation $\mu_0^{-1}\lambda_0^{-1}+\mu_0^{-1}+1$ says precisely that $(\Limm,\nabla)$ is unobstructed.
\end{proof}
\begin{remark}\label{zariskiopen}
The assignment~\eqref{gpemb} determines a map from the space of unobstructed local systems on $\Limm$ to $(\mathbb{K}^*)^5\subset\mathbb{P}^5$. This map is injective, since $\phi_{14} = \mu_1^{-1}$ and $\phi_{34} = \mu_3^{-1}$, and these, together with the other $\phi_{ij}$ determine the values of $\mu_0$, $\mu_2$, and $\mu_4$. Since the space of unobstructed local systems on $\Limm$ is $4$-dimensional, it follows that~\eqref{gpemb} maps this space surjectively onto $Gr(2,4)\cap(\mathbb{K}^*)^5\subset\mathbb{P}^5$. Thus line corresponding to such a point in $Gr(2,4)$ also corresponds to an unobstructed local system, which we denote by $\nabla_C$, on $\Limm$.
\end{remark}
We still need to rule out the possibility that the wrapped Floer cohomology can be nonzero even when the Floer cohomology groups of Lemma~\ref{hfcomp} vanish. In other words, we must show that there are not more points in the support of the mirror sheaf to $(\Limm,\nabla)$ than expected. There are inhomogeneous strips in the end of $T^*T^3$ which contribute to the Floer differential. For example, one has a family of strips which correspond to flowlines of the perfect Morse function on $T^2$ chosen in the proof of Lemma~\ref{wrappedcx}. The existence of these strips implies that none of Hamiltonian chords in the cylindrical end can be cocycles, unless the holonomy of $\nabla_{p,i}$ restricts to the same local system on $T^2$ as $\nabla$ does near (one of the copies of) the $i$th cusp of $L'$. This is because for any other choice of local system, the Floer differential restricted to an $H^*(T^2)$-summand of the cochain complex will be nonzero. The potential non-vanishing of these Floer cocycles thus gives us a constraint on the local system $\nabla_p$ we chose on $T^3$. By considering the Floer cohomology groups of $(\Limm,\nabla)$ with the Lagrangian branes supported on the \textit{other} submanifolds $L_{\langle q_j\rangle}$ and $L_{\langle q_k\rangle}$, we can check that this constraint would be inconsistent if it were the case that $\mathcal{O}_p$ were in the support of the mirror to $(\Limm,\nabla)$.
\begin{lemma}\label{support-necessary}
Fix an unobstructed local system $\nabla$ on $\Limm$, and suppose that there is a local system $\nabla_p$ on $T^3$ which does not satisfy one of the relations~\eqref{rel1}-\eqref{rel3}. Then the Floer cohomology group~\eqref{suppgp} taken with these local systems vanishes. 
\end{lemma}
\begin{proof}
If the point $p\in(\mathbb{K})^*$ were contained in the support of the mirror object to $(\Limm,\nabla)$, then it would follow, by applying hom-tensor adjunction in the derived category to $\mathcal{O}_{C_{p,j}}$, $\mathcal{O}_{C_{p,i}}$, and the mirror sheaf with grading shifted down by $3$, that for all $i = 1,2,3$, the wrapped Floer cohomology groups~\eqref{testgroup} must be nonzero. This would imply that some degree $0$ wrapped Floer cochain must represent a nonzero cohomology class.

If it were the case that, say,~\eqref{rel1} were not satisfied, then this cocycle would have to contain an element of one of the $H^0(T^2)$-summands as a summand. Denote such an element by $m\in H^0(T^2)$. The value of the Floer differential at $m$ is potentially nonzero, with image contained in $H^1(T^2)$. We claim that there are no other pseudoholomorphic strips which have an element of $H^1(T^2)$ as an output. There cannot be such a pseudoholomorphic strip with an input in $CW^0(\Lpants,T^2)$, as this would not respect the action filtration. For topological reasons there cannot be a strip with an input at one of the other degree $0$ generators in the cylindrical end. If there were, it would project to a bigon in $\mathbb{C}^*$ with a boundary component on the positive real locus and another boundary component on its image under the usual wrapping Hamiltonian on $\mathbb{C}^*$, but no such bigons exist.

This implies that the local system on $L_{\langle q_1\rangle}$ must be chosen so that this component of the Floer differential vanishes. By Lemma~\ref{inducedmap}, this can only happen if the holonomies of $\nabla_{p,1}$ are given by $x_2 = \mu_1\lambda_1$ and $x_3 = \lambda_1$. This implies that the local system $\nabla_p$ on $T^3$, is partially determined, as it must have holonomies $x_2 = \mu_1\lambda_1$ and $x_3 = \lambda_1$.

Next, suppose that we have a nonzero element of $HW^0((\Limm,\nabla),(L_{\langle q_2\rangle},\nabla_{p,2}))$ which is represented by a degree zero cochain that has an element of $\bigoplus_{n\in\mathbb{N}} H^0(T^2)$ as a summand. By the argument above, and Lemma~\ref{inducedmap}, this would imply that $x_3 = \mu_2$ (and also that $x_1 = \mu_2\lambda_2$) and by Lemma~\ref{relations}, this would in turn imply that $\mu_0 = 1$. In that case, it would follow that $\lambda_0^{-1}$, and hence $\mu_0\lambda^{-1}$, cannot be $1$, by Lemma~\ref{unobstructed}. Since the Floer cohomology group $HW^0((\Limm,\nabla),(L_{\langle q_3\rangle},\nabla_{p,3}))$ is also nonzero, it follows that $\nabla$ and $\nabla_p$ must satisfy~\eqref{rel3} (as $\nabla_p$ can not also be made to coincide with the restriction of $\nabla$ to a link of the third cusp). We claim that this is a contradiction, since the holonomies $\rho_{3,1}$ and $\rho_{3,2}$ of $\nabla$ are, by definition, the holonomies of cycles coming from the first and second cusps of $L'$, respectively. Our discussion above implies that these agree with $x_1$ and $x_2$, and so~\eqref{rel3} cannot vanish.

Therefore any degree zero cocycle must lie in the $CW^0(\Lpants,T^2)$-summand. This implies that $\nabla$ and $\nabla_p$ must satisfy~\eqref{rel2}, and we can assume that~\eqref{rel3} is also satisfied, since we have already considered the case where it is not. By the Grassmann--Pl{\"u}cker relation, which holds by Proposition~\ref{plucker-unobs}, we know that the ideal in $\mathbb{K}[x_1,x_2,x_3]$ generated by the linear forms appearing in~\eqref{rel2} and~\eqref{rel3} must also contain~\eqref{rel1}, and this is a contradiction.  Note that we can repeat this argument with different permutations of the coordinates on $(\mathbb{C}^*)^3$.
\end{proof}
\begin{proof}[Proof of Theorem~\ref{main}]
Consider a point of $Gr(2,4)$ that is not mapped to a coordinate hyperplane in $\mathbb{P}^5$ under the Pl{\"u}cker embedding. Let $\nabla_C$ denote the unobstructed local system on $\Limm$ described in Remark~\ref{zariskiopen}. By Lemma~\ref{hfcomp}, a sufficient condition for an object $(T^3,\nabla_p)$ to lie in the Floer-theoretic support of $(\Limm,\nabla_C)$ is that the holonomy of $\nabla_p$ satisfies the relations given in Lemma~\ref{hfcomp}. That these conditions are necessary follows from Lemma~\ref{support-necessary}.
\end{proof}

\bibliographystyle{abbrv}
\bibliography{myref}
\end{document}